\documentclass[reqno]{amsart}

\usepackage{verbatim}
\usepackage[textsize=scriptsize]{todonotes}
\usepackage{tikz-cd}
\usepackage{etoolbox}
\usepackage{etex}
\usepackage[hyperfootnotes=false]{hyperref}
\usepackage[multiple]{footmisc}
\usepackage{cleveref}
\usepackage[T1]{fontenc}
\usepackage{enumitem}

\usepackage{chemarr}
\usepackage{amssymb}
\usepackage{stmaryrd}
\usepackage{amsmath}
\usepackage{comment}
\usepackage{cleveref}
\usepackage{mathtools}
\usepackage{rotating}
\usepackage{wrapfig}
\usepackage{outlines}
\usepackage{graphicx}

\usepackage{aliascnt}

\let\oldtheorem\newtheorem
\RenewDocumentCommand{\newtheorem}{s m o m O{}}{%
\IfBooleanTF{#1}%
{\oldtheorem{#2}{#4}}%
{\IfNoValueTF{#3}{\oldtheorem{#2}{#4}[#5]}%
{\newaliascnt{#2}{#3}%
\oldtheorem{#2}[#2]{#4}%
\aliascntresetthe{#2}}}}

\theoremstyle{definition}
\newtheorem{nul}{}[section]
\newtheorem{dfn}[nul]{Definition}

\newtheorem{rmk}[nul]{Remark}

\newtheorem{cnstr}[nul]{Construction}

\newtheorem{ntn}[nul]{Notation}
\newtheorem{exm}[nul]{Example}

\newtheorem{rec}[nul]{Recollection}

\newtheorem{wrn}[nul]{Warning}
\newtheorem{qst}[nul]{Question}

\newtheorem*{dfn*}{Definition}
\newtheorem*{axm*}{Axiom}
\newtheorem*{ntn*}{Notation}
\newtheorem*{exm*}{Example}
\newtheorem*{exr*}{Exercise}
\newtheorem*{int*}{Intuition}
\newtheorem*{qst*}{Question}
\newtheorem*{rmk*}{Remark}

\theoremstyle{plain}

\newtheorem{thm}[nul]{Theorem}
\newtheorem{prop}[nul]{Proposition}
\newtheorem{lem}[nul]{Lemma}

\newtheorem{cor}[nul]{Corollary}

\newtheorem*{thm*}{Theorem}
\newtheorem*{prop*}{Proposition}
\newtheorem*{cor*}{Corollary}
\newtheorem*{lem*}{Lemma}
\newtheorem*{cnj*}{Conjecture}

\let\oldwidetilde\widetilde
\protected\def\widetilde{\oldwidetilde}

\DeclareMathOperator{\Map}{\mathrm{Map}}

\DeclareMathOperator{\gr}{\mathrm{gr}}
\DeclareMathOperator{\fil}{\mathrm{fil}}
\DeclareMathOperator{\can}{\mathrm{can}}
\DeclareMathOperator{\fib}{\mathrm{fib}}

\DeclareMathOperator{\coker}{\mathrm{coker}}

\DeclareMathOperator{\Ss}{\mathbb{S}}
\DeclareMathOperator{\SP}{\mathbb{S}}

\DeclareMathOperator{\FF}{\mathbb{F}}
\DeclareMathOperator{\F}{\mathbb{F}}
\DeclareMathOperator{\QQ}{\mathbb{Q}}
\DeclareMathOperator{\NN}{\mathbb{N}}
\DeclareMathOperator{\A}{\mathbb{A}}
\DeclareMathOperator{\C}{\mathcal{C}}
\DeclareMathOperator{\cC}{\mathcal{C}}
\DeclareMathOperator{\cD}{\mathcal{D}}

\DeclareMathOperator{\E}{\mathbb{E}}

\DeclareMathOperator{\Spec}{\text{Spec}}
\DeclareMathOperator{\Spc}{\text{Spc}}
\DeclareMathOperator{\Psh}{\text{Psh}}

\DeclareMathOperator{\CycSp}{\mathrm{CycSp}}
\DeclareMathOperator{\Cat}{\mathrm{Cat}}

\DeclareMathOperator{\et}{\text{\'et}}

\DeclareMathOperator{\Fil}{\mathbf{Fil}}

\DeclareMathOperator{\Alg}{\mathrm{Alg}}
\DeclareMathOperator{\Sp}{\mathrm{Sp}}

\DeclareMathOperator{\const}{\mathrm{const}}
\DeclareMathOperator{\cyc}{\mathrm{cyc}}
\DeclareMathOperator{\ev}{\mathrm{ev}}

\newcommand{\BP}{\mathrm{BP}}

\newcommand{\THH}{\mathrm{THH}}
\newcommand{\TP}{\mathrm{TP}}
\newcommand{\TC}{\mathrm{TC}}
\newcommand{\TR}{\mathrm{TR}}
\newcommand{\mot}{\mathrm{mot}}

\newcommand{\Syn}{\mathrm{Syn}}
\newcommand{\Mod}{\mathrm{Mod}}

\newcommand{\EE}{\mathbb{E}}

\newcommand{\CC}{\mathbb{C}}

\newcommand{\Z}{\mathbb{Z}}
\newcommand{\ZZ}{\mathbb{Z}}

\def\Spf{\mathrm{Spf}}

\def\H{\mathrm{H}}

\def\L{\mathcal{L}}

\def\G{\mathbb{G}}

\def\ku{\mathrm{ku}}

\def\Gal{\mathrm{Gal}}

\def\HT{\mathrm{HT}}
\def\Nyg{\mathrm{Nyg}}
\def\Zp{\mathbb{Z}_p}
\def\N{\mathcal{N}}
\def\O{\mathcal{O}}
\def\Fr{\mathrm{Fr}}
\def\syn{\mathrm{Syn}}
\def\Ani{\mathrm{an}}
\def\CAlg{\mathrm{CAlg}}
\def\Fr{\mathrm{Fr}}
\def\TR{\mathrm{TR}}

\def\Spa{\mathrm{Spa}}
\def\Q{\mathbb{Q}}
\def\RHom{\mathrm{RHom}}

\usepackage{relsize}
\usepackage[bbgreekl]{mathbbol}
\usepackage{amsfonts}

\DeclareSymbolFontAlphabet{\mathbbl}{bbold}
\def\prism{\mathbbl{\Delta}}

\DeclarePairedDelimiter\abs{\lvert}{\rvert}%
\makeatletter
\let\oldabs\abs
\def\abs{\@ifstar{\oldabs}{\oldabs*}}

\usepackage{tikz}
\usetikzlibrary{matrix,arrows,decorations}
\usepackage{tikz-cd}

\usepackage{adjustbox}

\let\oldtocsection=\tocsection
 
\let\oldtocsubsection=\tocsubsection
 
\let\oldtocsubsubsection=\tocsubsubsection
 
\renewcommand{\tocsection}[2]{\hspace{0em}\oldtocsection{#1}{#2}}
\renewcommand{\tocsubsection}[2]{\hspace{1em}\oldtocsubsection{#1}{#2}}
\renewcommand{\tocsubsubsection}[2]{\hspace{2em}\oldtocsubsubsection{#1}{#2}}

\usepackage{etoolbox}
\newtoggle{draft}
\toggletrue{draft}

\iftoggle{draft} {
\usepackage[margin=1.5in]{geometry}
\newcommand{\NB}[1]{\todo[color=gray!40]{#1}}
\newcommand{\TODO}[1]{\todo[color=red]{#1}}
}{ % else
\usepackage[margin=1.5in]{geometry}
\newcommand{\NB}[1]{}
\newcommand{\TODO}[1]{}
\renewcommand{\todo}[1]{}
\renewcommand{\todo}[1]{}
}
\title{Crystallinity for syntomic cohomology, \'etale cohomology, and algebraic $K$-theory}%the mod $(p,v_1^{p^{n-2}})$ $K$-theory of $\Z/p^{n}$}
\author{Jeremy Hahn}
\address{Department of Mathematics, MIT, Cambridge, MA, USA}
\email{jhahn01@mit.edu}
\author{Ishan Levy}
\address{Department of Mathematics, Institute for Advanced Study, USA}
\email{ishanl@ias.edu}
\author{Andrew Senger}
\address{Department of Mathematics, University of Maryland, College Park, MD, USA}
\email{senger@umd.edu}

\begin{document}
\begin{abstract}
We prove for $n\geq c-1$ that the functor taking an animated ring $R$ to its mod $(p^c,v_1^{p^n})$ syntomic cohomology factors through the functor $R \mapsto R/p^{c(n+2)}$, a phenomenon we term crystallinity for mod $(p^c,v_1^{p^n})$ syntomic cohomology. As an application, we completely and explicitly compute the mod $(p,v_1 ^{p^{n}-1})$ algebraic $K$-theory of $\Z/p^{k}$ whenever $k \geq n+2$ and $p>2$.  As a second application, we deduce crystallinity for the mod $p^c$ syntomic complexes associated to smooth $p$-adic formal schemes, and in particular for the Galois equivariant mod $p^c$ \'etale cohomologies of their adic generic fibers.  Finally, we strengthen known $p$-adic convergence theorems for the topological Hochschild homology of ring spectra, and as a result relate crystallinity for algebraic $K$-theory to Lichtenbaum--Quillen theorems.
%	We deduce the mod $(p,v_1^{p^n})$ algebraic $K$-theory of $\mathbb{Z}/p^{k}$ whenever $p$ is an odd prime.
%	When $p=2$ and $n>1$, we deduce the associated graded of the motivic filtration on the mod $(2,v_1^{2^n})$ algebraic $K$-theory of $\mathbb{Z}/2^k$.
\end{abstract}
\maketitle

\setcounter{tocdepth}{1}
\tableofcontents
\vbadness 5000
	
%----------------------------------------------------------------------%

% \newpage
\section{Introduction} \label{sec:intro}
Let $X$ denote a smooth qcqs $p$-adic formal $\mathbb{Z}_p$-scheme.  The existence of crystalline cohomology implies, rather miraculously, that the ($p$-completed, derived) deRham cohomology of $X$ depends functorially on the (derived) mod $p$ reduction $X_{p=0}$. Here, we prove that similar crystallinity properties are a general phenomenon, and apply them to the computation of the algebraic $K$-theory of $\Z/p^n$. %Recent years have seen the development and application of new $p$-adic cohomology theories, including prismatic and syntomic cohomology.% Prismatic cohomology can be viewed as a deformation of deRham cohomology, while the syntomic complexes $\mathbb{F}_p(*)(X)$ are closely related to mod $p$ algebraic $K$-theory and \'etale cohomology.

This paper achieves three main goals:
\begin{enumerate}
\item To determine new crystallinity properties for $p$-adic cohomology theories of $p$-adic formal schemes, including prismatic and syntomic cohomology.  We prove in particular that, for each $n \ge c-1\geq 0$, $\ZZ/p^c(*)(X) / v_1^{p^n}$ depends functorially on $X_{p^{c(n+2)}=0}$.  That is to say, while no $p$-power reduction of $X$ completely recovers the syntomic complexes $\ZZ/p^c(*)(X)$, $v_1$-power reductions of syntomic cohomology may be functorially recovered from $p$-power reductions of $X$.

  When $X$ is smooth over $\mathbb{Z}_p$, we can deduce results before killing any power of $v_1$.  In that case, we functorially recover the mod $p^c$ syntomic complexes $\Z/p^c (*)(X)$ from $X_{p^j = 0}$, where $j$ is an integer depending only on $c$ and the dimension of $X$. As a consequence, we may recover the mod $p^c$ \'etale cohomology complexes $R\Gamma_{\et} (X_{\eta, \mathbb{C}_p},\mathbb{Z}/p^c(*))$ of the adic generic fiber as $\Gal(\Q_p)$-representations.
    If $X$ is affine and smooth over some $\mathcal{O}_K$, then the \'etale cohomology groups $H^*_{\et} (X_{\eta, \mathbb{C}_p},\mathbb{Z}/p^c(*))$ can be recovered as $\Gal(\Q_p)$-representations from $X_{p^j=0}$ for an integer $j$ depending on $c$ alone, without reference to dimension.

    %the mod $p^c$ \'etale cohomology complexes $R\Gamma_{\et} (X_{\eta},\mathbb{Z}/p^c(*))$ of the adic generic fiber of $X$ can be functorially recovered from $X_{p^j=0}$, where $j$ is an integer depending on $c$ and the dimension of $X$.  If $X$ is affine, then the cohomology groups $H^*_{\et} (X_{\eta},\mathbb{Z}/p^c(*))$ can be recovered from $X_{p^j=0}$ for an integer $j$ depending on $c$ alone.
		
  \item We prove crystallinity for algebraic $K$-theory. Namely, for $R$ any $p$-complete $\mathbb{Z}$-algebra we prove that the mod $(p^{j_0},v_1^{j_1})$ algebraic $K$-theory of $R$ depends only on the derived quotient $R/p^i$, for some fixed $i$ depending only on $j_0$ and $j_1$. 	In the course of the argument, we strengthen previous $p$-adic continuity results for $\mathrm{TR}$, for example establishing the rapid convergence of the tower $\tau_{\le k}\mathrm{TR}(\mathbb{S}/p^\bullet)/p^a$ for any fixed integers $k,a \ge 0$.
	
  Moreover, we prove analogous results for algebras over any ring spectrum $A$ satisfying the Lichtenbaum--Quillen property.  More specifically, we prove that if the mod $(p^{j_0}, v_1 ^{j_1}, \dots v_{n}^{j_n})$ reduction of $\TR(A)$ is bounded, then the mod $(p^{j_0}, v_1 ^{j_1}, \dots v_{n}^{j_n})$ topological cyclic homology $\TC$ of any connective $A$-algebra $R$ may be functorially recovered from the mod $(p^{i_0}, v_1 ^{i_1}, \dots, v_{n-1} ^{i_{n-1}})$ reduction of $R$, for $(i_0, \dots, i_{n-1})$ sufficiently large.

%    For example, this may be applied to obtain crystallinity results over the base rings $R = \Z_p, \ku, \mathrm{ko}, \mathrm{BP}\langle n \rangle, j_{\zeta}, \tmf$. 

%	To establish a general crystallinity result for $p$-adic cohomology theories applied to algebras over a connective ring spectrum satisfying the Lichtenbaum--Quillen property. More specifically, we prove that if the mod $(p^{j_0}, v_1 ^{j_1}, \dots v_{n}^{j_n})$ reduction of $\TR(R)$ is bounded, then the mod $(p^{j_0}, v_1 ^{j_1}, \dots v_{n}^{j_n})$ topological cyclic homology $\TC$ and algebraic $K$-theory of an $R$-algebras $S$ may be functorially recovered from the mod $(p^{i_0}, v_1 ^{i_1}, \dots, v_{n-1} ^{i_{n-1}})$ reduction of $S$, for $(i_0, \dots, i_{n-1})$ sufficiently large.

%    For example, this may be applied to obtain crystallinity results over the base rings $R = \Z_p, \ku, \mathrm{ko}, \mathrm{BP}\langle n \rangle, j_{\zeta}, \tmf$. 

\item To study the algebraic $K$-theory of $\Z/p^n$.  Using \cite{BMS}, this may be understood through syntomic cohomology.  Our crystallinity result implies that, so long as $k \le p^{n-2}$, the mod $(p,v_1^k)$ syntomic cohomology of $\Z/p^n$ agrees with the mod $(p,v_1^k)$ syntomic cohomology of the trivial square-zero extension $\Z_p \oplus \Z_p[1]$.  In particular, the answer is independent of $n$ in this range!

We furthermore explicitly compute the mod $p$ syntomic cohomology, with all Breuil--Kisin twists, of the trivial square-zero extension $\Z_p \oplus \Z_p[1]$.  To make this calculation, we first prove that classic splitting results for $\TC$ of square-zero extensions, due to Lindenstrauss and McCarthy, respect the motivic filtration.  Even given these splittings, the feasibility of the computation rests on several tricks and happy coincidences, and in particular the surprise that it suffices to compute both $\mathrm{can}$ and $\varphi$ up to $v_1$-adic associated graded.
\end{enumerate}

One upshot of all this is the first closed form, complete calculation of \[\pi_* \left( K(\mathbb{Z}/p^n) / (p,v_1^k) \right),\] whenever $k < p^{n-2}$ and $p$ is odd.\footnote{A powerful algorithm to compute the integral homotopy groups of $K(\mathbb{Z}/p^n)$ was recently developed in \cite{kzpnother}.  We view this algorithm as complementary to our results, as we explain in \Cref{sec:intro-comparison}.} When $p=2$ and $k<2^{n-2}$ is divisible by $4$, we obtain the associated graded of the motivic filtration on this algebraic $K$-theory.

%This paper has two goals:
%\begin{enumerate}
%  \item To study the \emph{crystallinity} properties of reduced variants of $p$-adic cohomology theories such as prismatic and syntomic cohomology.
%    This yields one possible answer to the question: given a $p$-adic formal scheme $X$, what does the locus $X_{p^n = 0}$ know about the prismatic and syntomic cohomology of $X$?
%  \item To study the algebraic $K$-theory of $\Z/p^n$.
%    Using Bhatt--Morrow--Scholze's motivic spectral sequence \cite{BMS}, this may be studied through syntomic cohomology.
%    Using our crystallinity results, we are able to give a completely explicity closed form computation of mod $(p,v_1 ^{k})$ syntomic cohomology of $\Z/p^n$ when $k \leq p^{n-2}$.
%    By crystallinity, the answer is independent of $n$ in this range!
%\end{enumerate}

\subsection{Crystallinity for prismatic and syntomic cohomology}

The ``crystalline miracle'' may be formulated as follows:
the $p$-complete derived de Rham cohomology of $p$-complete animated commutative rings
\[A \mapsto \mathrm{dR}_{A/\Z_p}\]
factors through the derived mod $p$ reduction
\[A \mapsto A/p,\]
via the derived crystalline cohomology functor.
This plays a foundational role in $p$-adic geometry and Hodge theory, as crystalline cohomology defines a good $p$-adic cohomology theory for characteristic $p$ schemes.
%has many important consequences, such as the existence of an action of Frobenius on the derived de Rham cohomology.

In recent years, new $p$-adic cohomology theories have been introduced and used to study mixed characteristic phenomena: prismatic cohomology, which is a deformation of de Rham cohomology, and syntomic cohomology, which is closely related to algebraic $K$-theory and \'etale cohomology.

Our main result is an analogue of the crystalline miracle for certain reduced variants of these cohomology theories: they factor through reduction mod $p^n$.
One novel aspect of this theorem is that it provides a concrete understanding of how the prismatic or syntomic cohomology of a $p$-adic formal scheme $X$ is built up from its mod $p^n$ reductions $X_{p^n = 0}$.
%
%Our main result studies the extent to which a similar theorem is true for prismatic and syntomic cohomology, at least after reduction modulo $(p,v_1 ^{p^n})$.\footnote{We recall the meaning of the class $v_1$ in PLACE. This class is the same as the $v_1$ of chromatic homotopy theory.}
%Instead of proving that these invariants factor through the mod $p$ reduction, we study when they factor through the mod $p^n$ reduction.

\begin{thm} \label{thm:mainintro}
  Let $n,c\geq 0$.
  The following functors on $\CAlg_p ^\Ani$ factor through the mod $p^{c(n+2)}$ reduction functor $\CAlg_p^\Ani \to \CAlg_{\Z/p^{c(n+2)}}^\Ani$:
  \begin{enumerate}
    \item mod $(p^c,v_1 ^{p^n})$ syntomic cohomology\footnote{We recall the meaning of the class $v_1$ in \Cref{dfn:prism-v1} and \Cref{rmk:v1}. This class is the same as the $v_1$ of chromatic homotopy theory.} for $n\geq c-1$
      \[R \mapsto \ZZ/p^c (\ast) (R) /v_1 ^{p^n}\]
    \item mod $(p^c,v_1 ^{p^n})$ Nygaard-filtered derived prismatic cohomology for $n\geq c-1$
      \[R \mapsto N^{\geq \ast} \prism_{R} \{\ast\} /(p^c, v_1^{p^{n}})\]
    \item mod $(p^c,v_1 ^{p^{n+1}})$ derived prismatic cohomology for $n+1\geq c-1$
      \[R \mapsto \prism_{R} \{\ast\} /(p^c, v_1^{p^{n+1}})\]
    \item mod $(p^{kc}, (F^{n+1-k})^* I)$ derived prismatic cohomology, where $k \geq 1$, $F$ is the Frobenius, and $I$ is the Hodge-Tate ideal
      \[R \mapsto \prism_{R} \{\ast\} /(p^{kc}, (F^{n+1-k})^* I)\]
  \end{enumerate}
%
%  Given any $n \geq 0$, the mod $(p,v_1 ^{p^n})$ derived syntomic cohomology functor
%  \[\F_p (\bullet) / v_1 ^{p^n} : \CAlg_p^\Ani \to \CAlg_p^{\Ani,\gr}\]
%  factors through the functor which takes $R$ to its (derived) quotient $R/p^{n+2}$.
%
%  In other words, we construct a commutative diagram
%  \begin{center}
%    \begin{tikzcd}
%      \CAlg_p^\Ani \ar[rr,"A \mapsto \F_p (\bullet) (A) / v_1 ^{p^n}"] \ar[dr, "A \mapsto A/p^{n+2}"] & & \CAlg_p ^{\Ani,\gr} \\
%      & \Calg_{\Z/p^{n+2}} ^\Ani. \ar[ur] &
%    \end{tikzcd}
%  \end{center}
\end{thm}

\begin{rmk} \label{rmk:stronger-a-theorem}
  Parts (2) and (3) are consequences of the stronger \Cref{thm:main}(5) about
   mod $(p^c, a^{p^{n+1}} \mu^{p^n})$ Nygaard-filtered derived prismatic cohomology
      \[N^{\geq \ast} \prism_{R} \{\ast\} /(p^c, a^{p^{n+1}} \mu^{p^n}),\]
			where $v_1=a \mu$.  The element $a$ here is closely related to the class $a_{\lambda}:S^{-\lambda} \to S^{0}$ of $S^1$-equivariant homotopy theory.
\end{rmk}

\Cref{thm:mainintro} came to us as a surprise. At the outset of this project, our goal was to study the mod $(p,v_1)$ algebraic $K$-theory of $\Z/p^n$.
We began to suspect a version of \Cref{thm:mainintro} after carrying out the computation and finding something unexpected: the answer does not depend on $n \ge 2$.
%We began to suspect that some version of \Cref{thm:main} was true after carrying out this computation and finding something unexpected: the answer did not depend on $n \geq 2$.
%\Cref{thm:main} came about through our attempts to understand this phenomenon.

\begin{rmk}
  At least for mod $p$ syntomic cohomology, \Cref{thm:mainintro} is optimal in the following sense: for every $n \geq 2$, there is a pair of rings $A_1,A_2$ such that $A_1/p^n \cong A_2/p^n$ but $\F_p (\ast) (A_1) /v_1 ^{p^n} \not \simeq \F_p (\ast) (A_2) /v_1 ^{p^n}$.

  In fact, $A_1 = \Z/p^n$ and $A_2 = \Z/p^{n+1}$ work, as shown by Achim Krause and the third author in \cite{KS}.
\end{rmk}

%\begin{rmk}
  %Originally, we computed $\pi_* K(\Z/p^n) / (p,v_1)$.
  %After doing so, we found something unexpected: this answer did not depend on the choice of $n \geq 2$!
  %This paper represents our attempt to explain this phenomenon.
%
  %Our methods were particularly influenced by BHATT-LURIE and PETROV.
%\end{rmk}

Our proof of \Cref{thm:mainintro}, which was inspired by \cite[Example 5.15]{BLprismatization} and \cite[Section 6.3]{Petrov}, makes essential use of the stacky approach to prismatic and syntomic cohomology \cite{Drinfeld,APC,BLprismatization,fgauge}.
As a consequence, we actually prove the following stacky refinement of \Cref{thm:mainintro}.
\begin{thm}\label{thm:stackintro}
  The functors
  \begin{enumerate}
    \item \[X \mapsto (X^\syn)_{p^c=v_1^{p^n}=0} \]
    \item \[X \mapsto (X^\Nyg)_{p^c=v_1^{p^n}=0} \]
    \item \[X \mapsto (X^\prism)_{p^c=v_1^{p^{n+1}}=0}\]
    \item \[X \mapsto (X^\prism)_{p^{kc}=(F^{n+1-k})^* I=0}\]
  \end{enumerate}
  factor through the functor $X \mapsto X_{p^{c(n+2)}=0}$, where $X_{p^{(n+2)}=0}$ is regarded as derived $\Z/p^{c(n+2)}$-scheme. For (1),(2) we require $n \geq c-1$, and for (3) we require $n+1 \geq c-1$.
  Here, $X$ is a derived $p$-adic formal scheme, and all of the vanishing loci are taken in a derived sense.
\end{thm}

\begin{rmk}
  As in \Cref{rmk:stronger-a-theorem}, parts (2) and (3) are consequences of a stronger theorem about $(X^\Nyg)_{p^c=a^{p^{n+1}} \mu^{p^n}=0}$ (\Cref{thm:stack}(5)).
\end{rmk}

\begin{rmk}
  \Cref{thm:stackintro} implies that the map $(X^\syn)_{p=v_1^{p^n}=0} \to (X^{\syn})_{p=0}$ factors through $((X_{p^{n+2} = 0}) ^\syn)_{p=0}$
  On the other hand, Antieau--Krause--Nikolaus have shown that the map $((X_{p^n = 0}) ^\syn)_{p=0} \to (X^\syn)_{p=0}$ factors through $(X^\syn)_{p=v_1^{1+p+\dots+p^{n-1}}=0}$ \cite[Theorem 1.8]{kzpnother}.
  These results may be interpreted as saying that, on $(X^\syn)_{p=0}$, the $v_1$-adic filtration and the filtration induced by the $p$-adic filtration on $X$ are commensurate.
  We view this as a quantitative aspect of the chromatic redshift philosophy in algebraic $K$-theory.
\end{rmk}

\subsection{Applications to the \'etale cohomology and $F$-gauge} \label{sec:intro-etale}

In the case of $F$-smooth qcqs $p$-adic formal schemes $X$, we are able to functorially recover the mod $p^c$ syntomic cohomology from $X_{p^n=0}$ when $n$ is sufficiently large, without modding out by a power of $v_1$.
As a consequence, we are also able to recover the \'etale cohomology of the generic fiber from the $X_{p^n=0}$ for $n$ sufficiently large.
Examples of $F$-smooth $p$-adic formal schemes are given by those that are smooth over $\Spf (\O_K)$, where $\O_K$ is the ring of integers in a local or perfectoid field $K$.

%The following theorem may be viewed as a sort of $p$-adic continuity of the assignment
%\[X \mapsto H^*_{\mathrm{syn}}(X;\ZZ/p^c(j))\] for such rings.
%The statement of this result depends on the notion of \emph{affine cohomological dimension} that we introduce in \Cref{dfn:affcohdim}.
%The affine dimension is always bounded above by the Zariski cohomological dimension, but may be smaller.
%In particular, every affine $X$ has affine cohomological dimension equal to zero.

\begin{thm} \label{thm:genericintro}
	Let $X$ denote a $p$-torsionfree $F$-smooth qcqs affine $p$-adic formal scheme, and suppose that $p+c\geq 5$. %\footnote{Here, we mean the dimension of the underlying space of $X$. This is equal to the dimension of the scheme $X_{p=0}$.}
	Then the syntomic and \'etale mod $p^c$ cohomology rings\footnote{Here, $X_\eta$ is the adic generic fiber of $X$, i.e. $X \times_{\Spa(\Z_p, \Z_p)} \Spa(\Q_p, \Z_p)$ in the category of (pre-)adic spaces.}
	\[H^i_{\mathrm{syn}}(X;\ZZ/p^c(j)) \text{ and } H_{\et}^i(X_\eta; \ZZ/p^c(j))\]
	may be recovered functorially from $X_{p^{3c} = 0}$.
	
	More precisely, these functors factor through the essential image of the functor $X \mapsto X_{p^{3c}=0}$ from such $X$ into $\ZZ/p^{3c}$-schemes.
	%  In particular, if $X_1$ and $X_2$ as above satisfy $(X_1)_{p^{b(d)} = 0} \cong (X_2)_{p^{b(d)} = 0}$, then $b_{i,j} (X_1;\F_p) = b_{i,j} (X_2;\F_p)$.
\end{thm}

\begin{rmk}\label{rmk:notaffine}
	The affineness assumption above is not necessary. For a more general $X$, we show the cohomology depends on $X_{p^{b'_c(d)}=0}$, where $b'_c(d)$ is a function depending on $d$, the \textit{affine cohomological dimension} of $X$ (see \Cref{dfn:affcohdim}). See \Cref{thm:generic} for a precise statement.
\end{rmk}

\begin{wrn}
	We note that \Cref{thm:genericintro} recovers the cohomology rings functorially, but not the complexes.
\end{wrn}

\begin{rmk}
  Let $X$ be smooth over the ring of integers $\mathcal{O}_{K}$ of a finite extension $K$ of $\QQ_p$. As in \Cref{rmk:notaffine}, the functoriality shows that we can then recover the \'etale cohomology of $X_{\eta, \mathbb{C}_p}$ as a mod $p^c$ Galois representation from $X_{p^{b'_c(d')}=0}$, where $d'$ is the affine cohomological dimension of $X\times_{\Spf(\mathcal{O}_K)} \Spf(\mathcal{O}_{\CC_p})$.
\end{rmk}

%\begin{rmk}
	%
	%since the mod $p^k$ reduction of the Galois representation only depends on the mod $p^\ell$ reduction of $X$, where $\ell$ depends on $k$ and the dimension of $X$.
%\end{rmk}

If we restrict to $p$-adic formal schemes that are smooth over $\Spf \Zp$, then we refine our results by recovering the mod $p^c$ $F$-gauge of $X$ from the mod $p^n$ reduction of $X$.

\begin{thm} \label{thm:FLintro}
	Let $\pi : X \to \Spf \Zp$ denote a smooth $p$-adic formal scheme of relative dimension $d$, and let $b_c ''(d) = \max\left(c+1, \lceil \log_p \left( \frac{d+1}{p-1} \right) + 2 \rceil \right)$.
	Then the associated mod $p^c$ $F$-gauge $R\pi ^\syn _* \mathcal{O}_{(X^\syn)_{p^c=0}} \in \mathcal{D} ((\Zp ^{\syn})_{p=0})$ may be functorially recovered from $X_{p^{cb''_c (d)} = 0}$.
\end{thm}

Taking the \'etale realization, one obtains the following corollary:

\begin{cor}
  Given a smooth $p$-adic formal scheme $\pi : X \to \Spf \Zp$ of relative dimension $d$, the mod $p^c$ \'etale cohomology complex of the generic fiber
	\[R\Gamma_{\et} (X_{\eta, \mathbb{C}_p}, \mathbb{Z}/p^c)\]
  may be functorially recovered from $X_{p^{b''_c (d)} = 0}$ as a continuous $\Gal(\Q_p)$-representation.%  Functoriality in particular means that we recover the $\mathrm{Gal}(\mathbb{Q}_p)$ action on this complex. %the complex as a continuous representation of $\mathrm{Gal}(\mathbb{Q}_p)$.
\end{cor}

\begin{rmk}
	For example, when $d \leq p-2$ the above theorem shows that the mod $p$ \'etale cohomology of the generic fiber may be recovered from the mod $p^2$ reduction of $X$, a result that one could also imagine proving using Fontaine--Laffaille theory.
	When $d \leq p^2-p-1$, our result shows that the mod $p$ \'etale cohomology may be recovered from the mod $p^3$ reduction of $X$, and the mod $p^2$ \'etale cohomology may be recovered from the mod $p^6$ reduction of $X$.
\end{rmk}

\begin{rmk}
	In fact, the proof of \Cref{thm:FLintro} establishes something stronger.
	Namely, given $d \geq 0$, there is a cohomology theory on smooth morphisms $\overline{X} \to \Spec (\Z/p^{cb_c''(d)})$ of relative dimension $\leq d$ taking values in $\mathcal{D} ((\Zp ^{\syn})_{p^c=0})$ that recovers the mod $p^c$ $F$-gauge of $X$ when $\overline{X} = X_{p^{cb''(d)}=0}$.
	In particular, taking the \'etale realization, we find that the ``mod $p^c$ \'etale cohomology of the generic fiber'' may be defined without the assumption that $\overline{X}$ lifts to a formal scheme over $\Zp$.
\end{rmk}

\subsection{Crystallinity for the cohomology of ring spectra}

We are also able to prove a very general crystallinity theorem for the $\THH$ and algebraic $K$-theory of ring spectra.
From the point of view of this result, the key input is that the base ring $R$ that one works over satisfies the Lichtenbaum--Quillen property: $\TR (R) \otimes V$ should be truncated for some type $n+2$-complex $V$.

To account for the fact that generalized Moore spectra do not admit $\E_\infty$-structures, as well as to accommodate interesting examples of $R$ which do not admit $\E_\infty$-structures, this theorem will be stated in the context of $\E_m$-rings.

Before we state the theorem, we need to fix multiplicative structures on generalized Moore spectra.
In what follows, we fix positive integers $m$, $c$ and $b$, an $\mathbb{E}_m$-algebra structure on $\Ss/p^c$, a tower of $\mathbb{E}_{m+1}$-algebras
  \[\dots \to \Ss/p^{b+2} \to \Ss/p^{b+1} \to \Ss/p^{b},\]
  and an $\mathbb{E}_{m+1}$-map
  \[\Ss/p^{b+i_0} \to W = \Ss / (p^{b+i_0}, \dots, v_n ^{i_n})\]
  for some $i_0, \dots, i_n$.
It is a consequence of \cite{burklund2022multiplicative} that this is possible for suitable choices of $c,b, i_0, \dots, i_n$.

\begin{thm} \label{thm:general-crys}
	Let $R$ denote a connective $\mathbb{E}_{m+1}$-algebra for $m\geq1$.
  Fix $n \ge 0$ and $q=2^k-1$, and let $\Ss/p^c \to V$ denote an $\EE_{m}$-map where $V$ is a type $n+2$-complex such that $\TR(R) \otimes  V$ is $q$-truncated.
	
  Suppose further that
  $i_0 \geq (q+1)(2c+1)$
  and
  $i_\ell \geq \frac{q+2}{2p^\ell-2},$
  for all $1 \leq \ell \leq n$.
  
  Let $\Alg_{\EE_m}'(\Mod(R)_{\geq 0})$ be the full subcategory of $p$-complete connective $\EE_m$-$R$-algebras such that either $\pi_0$ is commutative or $p$ is nilpotent.
  Then the functors
  \begin{align*}
    \THH(-) \otimes V &: \Alg_{\EE_m}(\Mod(R))\to \Alg_{\EE_{m-1}}(\Mod_{\CycSp}(\SP/p^c)) \\
    K(-) \otimes V &: \Alg'_{\EE_m}(\Mod(R)_{\geq 0})\to \Alg_{\EE_{m-1}} (\mathrm{Sp}_{\geq 0})
  \end{align*}
  factor through the functors
  \[- \otimes W : \Alg_{\EE_m}(\Mod(R))\to \Alg_{\EE_m}(\Mod(R\otimes W))\]
  \[- \otimes W : \Alg'_{\EE_m}(\Mod(R))\to \Alg_{\EE_m}(\Mod(R\otimes W))\]
  respectively.
  %Then the functors on $\Alg_{\EE_m}(\Mod(R)^{\geq 0})$
  %\[A \mapsto \THH(A) \otimes V \text{ and } K(A) \otimes V\]
  %with values in $\Alg_{\EE_{m-1}}(\Mod_{\CycSp}((\tau^{\cyc}_{\leq q}\SP)/p^c))$ and $\Alg_{\EE_{m-1}} (\mathrm{Sp})$, respectively,
  %factor through the functor $\Alg_{\EE_m}(\Mod(R))\to \Alg_{\EE_m}(\Mod(R\otimes W))$
	%\[A \mapsto A \otimes W.\]
%
	%Then the functor $\Alg_{\EE_m}(\Mod(R))\to \Alg_{\EE_{m-1}}(\Mod_{\CycSp}((\tau^{\cyc}_{\leq q}\SP)/p^c))$ 
	%\[A \mapsto \THH(A) \otimes V\]
	%factors through the functor $\Alg_{\EE_m}(\Mod(R))\to \Alg_{\EE_m}(\Mod(R\otimes W))$
	%\[A \mapsto A \otimes W.\]
%
  %Similarly, the functor $\Alg_{\EE_m}(\Mod(R)^{\geq 0})\to \Alg_{\EE_{m-1}} (\mathrm{Sp}^{\geq 0})$ 
	%\[A \mapsto K(A) \otimes V\]
  %factors through the functor $\Alg_{\EE_m}(\Mod(R)^{\geq 0})\to \Alg_{\EE_m}(\Mod(R\otimes W)^{\geq 0})$
	%\[A \mapsto A \otimes W.\]
\end{thm}

\begin{rmk}
  The theorem applies to the $\E_\infty$-rings $R = \Z_p, \ku, \mathrm{ko}, \ell, \ell^{h\Z}$ and to the $\E_3$-rings $R = \BP\langle n \rangle$ by \cite{ausoni2002algebraic, ausoni2010algebraic, hahn2022redshift, algKrealK, telescope}. In particular, applying this to $R=\mathbb{Z}_p$, we learn that the mod $(p^{j_0},v_1^{j_1})$ algebraic $K$-theory of a $p$-complete $\mathbb{Z}$-algebra $A$ depends only on the derived quotient $A/p^i$, where $i$ is an integer depending only on $j_0$ and $j_1$.
\end{rmk}

The proof of the above theorem begins with the fact that $\THH(R) \otimes V$ is $q$-truncated in the cyclotomic $t$-structure, and hence receives an $\E_m$-map from $(\tau_{\leq q}^{\mathrm{cyc}} \Ss) / p^c$.
The main step in the proof of the theorem is the construction of a retraction of $(\tau_{\leq q} ^{\mathrm{cyc}} \Ss) / p^c \to \THH(W) \otimes (\tau_{\leq q}^{\mathrm{cyc}} \Ss)/p^c$ in cyclotomic $\E_m$-rings.
This, in turn, we deduce from the following theorem:

\begin{thm}\label{thm:thhcontinuity}
  Let $q = 2^k-1$. Then the tower of $\E_m - \tau^{\mathrm{cyc}} _{\leq q} \Ss$-algebras in cyclotomic spectra
\[\{i \mapsto \THH(\mathbb{S}/p^{\max(i,b)})\otimes (\tau_{\leq q}^{\cyc}\mathbb{S})/p^c)\}\]
is rapidly converging to $\tau^{\mathrm{cyc}}_{\le q} (\mathbb{S}/p^c)$ at rate $2^k(2c+1)$ and constant $b$.
\end{thm}

A corollary of this theorem is the following:

\begin{cor}
	Let $R$ be an $\EE_{m+1}$-algebra with $\THH(R)$ bounded below.
	For each $k \ge 0$, let $q = 2^k-1$. Then the tower of $\mathbb{E}_m$-algebras in topological Cartier modules
	\[\{i \mapsto \tau_{\leq q}(\TR(R/p^{\max(i,b)})/p^c)\}\]
	is rapidly converging to $\tau_{\leq q}(\TR(R)/p^c)$ at rate $2^k(2c+1)$ and constant $b$.
\end{cor}

Results similar to the above corollary have been proven before: \cite[Theorem 5.20]{CMM} proves a similar result for $\TC$ as a spectrum, and \cite[Theorem 0.4]{dundas2017finite} proves a result for the homotopy groups of $\TR^n$. Our result strengthens these, because the data of $\TR$ as a topological Cartier module is equivalent to the full data of $\THH$ as a cyclotomic spectrum \cite{antieau2021cartier}, from which $\TC$ and the homotopy groups of $\TR^n$ can be derived.  Furthermore, we prove rapid convergence in the category of $\EE_m$-algebras, which is stronger than rapid convergence in the category of spectra.

%Our result strengthens these by working for $\TR$ as a topological Cartier module (or equivalently $\THH$ as a cyclotomic spectrum) and also giving rapid convergence as $\EE_m$-algebras.

\subsection{$K$-theory of $\Z/p^n$}

The original purpose of this project was the explicit computation of the mod $(p,v_1)$ $K$-theory of $\Z/p^n$, via mod $(p,v_1)$ syntomic cohomology.  We achieve this, and more, based on the following corollary of \Cref{thm:main}:

\begin{cor} 
  Let $\Z_p \langle \epsilon \rangle$ denote the free animated $\Z_p$ algebra on a class $\epsilon$ in degree $1$.
  For $n \geq k+2$, there are equivalences
  \[\F_p (*) (\Z/p^n)  / v_1 ^{p^k} \simeq \F_p (*) (\Zp \langle\epsilon\rangle) / v_1 ^{p^k}.\]
  In particular, the answer is independent of $n$.
\end{cor}

To compute the mod $(p,v_1^{p^k})$ syntomic cohomology of $\Z/p^{n}$, for any $n \ge k+2$, it thus suffices to compute the mod $p$ syntomic cohomology of $\Z_p \langle \epsilon \rangle$ as an $\F_p [v_1]$-module. The animated ring $\Z_p \langle \epsilon \rangle$ is isomorphic to the trivial square-zero extension of $\Z_p$ by $\Z_p[1]$. To compute its syntomic cohomology, we study the interaction of a formula for the topological cyclic homology of a square-zero extension with the Bhatt--Morrow--Scholze motivic filtration in \Cref{sec:epsilon}. Using this, we are able to leverage our computational knowledge of $\THH(\Z_p)$ as a cyclotomic spectrum give a complete computation of the mod $p$ syntomic cohomology of $\Z_p \langle \epsilon \rangle$.

\begin{thm} \label{thm:zpn-comp}
  There is an isomorphism of $\F_p [v_1]$-modules
  \begin{align*}
    &\pi_* \F_p (*) (\Z_p \langle \epsilon \rangle) \cong \\
    &\left(\pi_* \F_p (*) (\Z_p) \right) \oplus \bigoplus_{\stackrel{k > 0}{p \nmid k}} \bigoplus_{n=0} ^\infty \left( A_{n,k} \oplus B_{n,k} \oplus C_{n,k} \oplus \bigoplus_{r=1} ^n \left(D_{n,k,r} \oplus E_{n,k,r} \right) \right),
  \end{align*}
  where $A_{n,k}, B_{n,k}, C_{n,k}, D_{n,k,r}, E_{n,k,r}$ are explicit collections of elements described in \Cref{dfn:trfamilies}.
	More specifically, linearly independent $\mathbb{F}_p[v_1]$-module generators for $A_{n,k},B_{n,k},C_{n,k},D_{n,k,r},$ and $E_{n,k,r}$ are enumerated in \Cref{dfn:tr-family-generators}, and the $v_1$-torsion order of each of these generators is enumerated as \Cref{lemma:trfamilies}.
\end{thm}

\begin{rmk}
The retract \[\mathbb{Z}_p \to \mathbb{Z}_p\langle \epsilon\rangle \to \mathbb{Z}_p\] of rings leads to the $\pi_* \F_p (*) (\Z_p)$ summand in \Cref{thm:zpn-comp}.
Note that $\pi_* \F_p (*) (\Z_p)$ was computed by \cite{LiuWang}.  It is a free $\mathbb{F}_p[v_1]$ module on $p+3$ explicit basis elements.
\end{rmk}

\begin{rmk}
The computation constituting \Cref{thm:zpn-comp} is far from straightforward, even given the splitting results of \Cref{sec:epsilon}.  We perform several intermediate steps of the computation only at the level of $v_1$-adic associated graded.  It is then a striking fact that, for bidegree reasons alone, the final answer can be uniquely assembled from this partial information. 
\end{rmk}

From \Cref{thm:zpn-comp}, one is easily able to deduce the mod $(p,v_1 ^{k})$ syntomic cohomology of $\Z/p^n$ so long as $k \le p^{n-2}$.  In order to deduce results about algebraic $K$-theory, one first needs to understand both differentials and extension problems in the motivic spectral sequence converging to $\pi_*\TC(\mathbb{Z}/p^n)/(p,v_1^{k})$.  Controlling the behavior of this spectral sequence turns out to be somewhat more subtle than controlling the behavior of the motivic spectral sequence converging to the integral groups $\pi_*\TC(\mathbb{Z}/p^n)$, essentially because, when displayed in Adams grading, one has to worry about classes appearing below the $x$-axis. Nonetheless, in \Cref{sec:consequences} we use a result of Achim Krause and the third author to prove that the following theorem:

\begin{thm} \label{thm:intro-tc}
  Let $p$ denote a prime, $n \geq 2$, and $k \leq p^{n-2}$.

  If $p$ is odd, then there is an isomorphism of $\F_p$-vector spaces
  \[\pi_* \TC(\Z/p^n) / (p,v_1 ^{k}) \cong \pi_* \bigoplus_{i \geq 0} \F_p (i) (\Z_p \langle \epsilon \rangle) [2i] / (p,v_1^{k}).\]

  If $p=2$ and $k$ is divisible by $4$, then the graded abelian group $\pi_* \TC(\Z/2^n) / (2,v_1 ^{k})$ admits a filtration with associated graded given by $\pi_* \bigoplus_{i \geq 0} \F_2 (i) (\Z_2 \langle \epsilon \rangle) [2i] / (2,v_1^{k})$.
\end{thm}

\begin{rmk}
  At the prime $2$, the filtration described at the end of \Cref{thm:intro-tc} does not split, and indeed $\pi_* \TC(\Z/2^n)/(2,v_1^{k})$ is not an $\mathbb{F}_2$-vector space.  Instead, there are motivic $2$-extensions mapped forward from $\pi_*(\TC(\Z_2)/2)$ \cite[Theorem 8.21]{LiuWang}. We do not know whether this accounts for all of the $2$-extensions.
\end{rmk}

Applying the Dundas--Goodwillie--McCarthy theorem and the known computation of the algebraic $K$-theory and topological cyclic homology of $\F_p$, we obtain the following theorem bridging the gap between $\TC$ and algebraic $K$-theory:

\begin{thm}
  Let $p$ denote a prime, $n \geq 2$, and let $k \leq p^{n-2}-1$. If $p=2$, then assume that $k$ is divisible by $4$.
  Then there is an exact sequence
  \[0 \to \F_p \{v_1 ^k \partial\} \to \pi_* K(\Z/p^n)/(p,v_1 ^k) \to \pi_* \TC(\Z/p^n)/(p,v_1 ^k) \to \F_p \{\partial\} \to 0.\]
\end{thm}

\begin{rmk}
We have to exclude the case of $k=p^{n-2}$ because we do not know whether $\partial \in \pi_{-1} \TC (\Z/p^n)/p$ is $v_1^{p^{n-2}}$-torsion.
\end{rmk}

Recall from \cite{kzpnother} that the even $K$-groups of $\Z/p^n$ vanish in degrees large relative to $n$. In \cite{KS}, Achim Krause and the third author use \Cref{thm:main} and \Cref{thm:zpn-comp} to give precise conditions for when these $K$-groups vanish. As a final result, we prove the following consequence of \Cref{thm:main}:

\begin{thm} \label{thm:intro-cyclic}
  For $n \geq 2$ and $i\ge 0$, the map
  \[K_{2i} (\Z_p) \to K_{2i} (\Z/p^n)\]
  is surjective.
  In particular, $K_{2i}(\Z/p^n)$ is cyclic.
\end{thm}

\subsection{Comparison with the results of \cite{kzpnother}} \label{sec:intro-comparison}
In \cite{kzpnother}, Antieau--Krause--Nikolaus combine the algebra of prismatic envelopes with a clever filtration argument to give a finitary algebraic model for the syntomic cohomology of $\O_K / \varpi^n$ of a given weight.
Using this, they give an algorithm which computes the algebraic $K$-groups of $\O_K/ \varpi^n$, and have written a computer program which implements this algorithm. %in a range, and have written a computer program which implements this program.
Results of their program appear in tables \cite[Figure 1 \& Appendix A]{kzpnother} \cite[Figures 1-4]{kzpnotherannounce}.

In contrast, we give a closed form description of the mod $(p,v_1 ^{p^{n-2}})$ syntomic cohomology of $\Z/p^n$ in all weights.
In particular:
\begin{itemize}
  \item The methods of \cite{kzpnother} compute integral $K$-groups, while our methods are at best only able to get at mod $p^c$ $K$-groups.
  \item As $p$ and $n$ get large, the compute time of \cite{kzpnother}'s algorithm also becomes large.
    On the other hand, as $p$ and $n$ get large, the degree of $v_1 ^{p^{n-2}}$ becomes large, so that mod $(p,v_1 ^{p^{n-2}})$ $K$-theory agrees with mod $p$ $K$-theory in a larger and larger range.
		
    For example, $48$ hours on a high-performance cluster was enough to compute $K_* (\Z/7^7)$ through degree $*=15$ using the program of \cite{kzpnother}, but not $K_{16}(\Z/7^7)$ \cite[Appendix A.4]{kzpnother}. Our results determine $\pi_*(K(\Z/7^7)/7)$ through degree $\approx \abs{v_1^{7^{5}}} = 201684$.

  \item One may compare our results with the tables of Antieau--Krause--Nikolaus and see that they are consistent.
  \item The results of \cite{kzpnother} apply to $\O_K / \varpi^n$, where $\O_K$ is a characteristic $(0,p)$ complete discrete valuation ring with finite residue field and uniformizer $\varpi$.
    Here, we have restricted our computations to the case of $\Z/p^n$.
\end{itemize}

\begin{rmk}
In degrees $i$ large enough that their even vanishing theorem applies, Antieau--Krause--Nikolaus apply work of Angeltveit to prove that the order of $K_{2i-1}(\mathbb{Z}/p^n)$ is exactly $(p^i-1)p^{i(n-1)}$ \cite[Corollary 1.6]{kzpnother}.  The other most basic question one could ask about this abelian group is its rank, i.e. how many summands appear when it is written as a direct sum of cyclic groups of prime power order.  Again applying even vanishing, computing this rank is equivalent to computing the mod $p$ algebraic $K$-group $\pi_{2i-1} (K(\mathbb{Z}/p^n) / p)$.  In fact, in low degrees where even vanishing does not apply, the combination of our \Cref{thm:intro-cyclic} with \cite{KS} still reduces computing the rank of $K_{2i-1}(\mathbb{Z}/p^n)$ to the computation of mod $p$ $K$-groups.  Thus, for example, we now know the rank of $K_{*}(\mathbb{Z}/7^7)$ for all $* \le 201682$.
\end{rmk}

In order to obtain a complete, closed form computation of $\pi_*\left( K(\Z/p^n) / p \right)$, the natural strategy is to use the $v_1$-Bockstein spectral sequence
\[\mathrm{E}_1=\pi_*(\TC(\Z/p^n) / (p,v_1)) [v_1] \implies \pi_*(\TC(\Z/p^n) / p).\]
This spectral sequence has finitely many pages, because $v_1$ acts nilpotently on $\pi_*(\TC(\Z/p^n)/p)$.
We give in this paper explicit names to every element on the $\mathrm{E}_1$-page of the Bockstein spectral sequence.  We also explain how to compute the first several differentials by determining the homotopy groups of $\TC(\Z/p^n) / (p,v_1^{p^{n-2}})$, but there remain a few undetermined differentials.

By our work here, elements on the $\mathrm{E}_1$-page appear in five infinite families. For low values of $n$, Bockstein differentials on the initial elements in each family may be understood using \cite{kzpnother} and high performance computing.  Preliminary experiments in this direction give hope that the Bockstein spectral sequence follows a reasonably controlled pattern.  After the $v_1$-Bockstein spectral sequence is completely understood, one might turn to the $p$-Bockstein spectral sequence determining integral algebraic $K$-groups.

\subsection{Questions}
Our work raises several additional questions that we have not attempted to address, but which we would like to highlight here.

To begin with, we have elected in this paper to stick to $\Z/p^n$, rather than more general rings of the form $\O_K / \varpi^n$. %, so as not to get too bogged down.
It would be interesting to extend our computations to the case of $\O_K / \varpi^n$.
In the case that $p^2 \mid \varpi^n$, the mod $(p,v_1)$ syntomic cohomology of $\O_K / \varpi^n$ may be studied using \Cref{thm:main}.
In the case that $\varpi^n \mid p$, $\O_K / \varpi^n$ is a truncated polynomial algebra over a finite field, so its $K$-theory is known by \cite{HM}.

\begin{qst}
  How does the mod $(p,v_1)$ syntomic cohomology of $\O_K / \varpi^n$ behave when $p^2 \nmid \varpi^n$ and $\varpi^n \nmid p$?
  Curiously, \cite[Figure 2]{kzpnotherannounce} shows that the syntomic cohomology of $\F_2[z] / z^3$ and $\Z_2[2^{1/2}]/2^{3/2}$ agree in a range.
\end{qst}

When $A$ is a $\Z_p [\zeta_p]$-algebra, $v_1$ admits a natural $(p-1)$st root $\beta$ in the mod $p$ syntomic cohomology of $A$.
We wonder if our theorem can be refined to study the mod $(p, \beta^k)$ syntomic cohomology in this case.

\begin{qst}
  Is there a refined version of \Cref{thm:main} for $p$-complete animated $\Z_p [\zeta_p]$-algebras, where we quotient on the inside by powers of $(\zeta_p-1)$ and on the outside by $(p,\beta^k)$, where $\beta \in H^0 (\F_p (1) (\Z_p [\zeta_p]))$ is the Bott element, a $(p-1)$st root of $v_1$?

  In general, can \Cref{thm:main} be improved in the ramified situation? Can modding out by $p^2$ be replaced by modding out by $p \varpi$?
\end{qst}

It also seems likely that \Cref{thm:FL} should generalize to smooth $p$-adic formal schemes over more general $\mathcal{O}_K$.
The missing ingredient that one would have to supply is a version of \Cref{lem:vol} for the syntomification of $\mathcal{O}_K$.

\begin{qst}
	Is there also a version of \Cref{thm:FL} where $X$ is allowed to be smooth of relative dimension $d$ over $\Spf \mathcal{O}_K$ for a ring of integers in a finite extension of $\ZZ_p$ (as opposed to just $\ZZ_p$)?
\end{qst}

It would also be interesting to find optimal constants for our results.

\begin{qst}
	What are optimal quantitative versions of the crystallinity results, \Cref{thm:stackintro}, \Cref{thm:genericintro}, \Cref{thm:FLintro}, \Cref{thm:general-crys}? For example, if you fix $i\geq0, c\geq 0$, what is the smallest $n$ such that the functor $X \mapsto (X^\prism)_{p^c=v_1^{p^{c-1}i}=0}$ factors through $X \mapsto X_{p^n=0}$?
\end{qst}

If one were able to compute the mod $p^c$ syntomic cohomology of $\Z_p \langle \epsilon \rangle$ along with its action of $v_1 ^{p^{c-1}}$, then one could use crystallinity to deduce the mod $(p^c, v_1 ^{p^{c-1}})$ syntomic cohomology of $\Z/p^n$ for $n \geq c^2+c$.

\begin{qst}
  Can one give a closed-form computation of the mod $p^c$ syntomic cohomology of $\Z_p \langle \epsilon \rangle$ along with its action of $v_1 ^{p^{c-1}}$?
\end{qst}

We wonder if there is an absolute version of \Cref{thm:THHcrysmain}. A basic version of this would be:

\begin{qst}
	Fix a connective $\EE_m$-algebra $V$ that is a type $n$ complex. Consider the category of connective $\EE_{\infty}$-rings $R$ such that $\THH(R)\otimes V$ is $c$-truncated for some $c$. Then does the functor $R \mapsto \THH(R)\otimes V$ factor through the functor $R \to R\otimes W$ to $\EE_{m'}$-algebras under $W$ for some connective $\EE_{m'}$-algebra $W$ that is a type $n-1$ complex?
\end{qst}

Finally, it seems very likely that \Cref{thm:general-crys} is compatible the motivic filtration on $\THH$.

\begin{qst}
  Can \Cref{thm:general-crys} be refined to be compatible with the motivic filtration on $\THH$? In this case, the target category should be some variant of $\E_n$-rings in the category of synthetic cyclotomic spectra studied in \cite{CycSyn}.
\end{qst}

\subsection{Outline}
In \Cref{sec:prismatize}, we carry out our crystallinity argument for mod $p$ prismatic cohomology, proving a weak mod $p$ version of part (3) of \Cref{thm:mainintro} and \Cref{thm:stackintro}.
In \Cref{sec:Nygaard}, we extend this argument from the prismatization to the Nygaard-filtered prismatization and the syntomification, proving the entire mod $p$ version of \Cref{thm:mainintro} and \Cref{thm:stackintro}. In \Cref{sec:rapid}, we study the notion of rapid convergence of towers, and prove \Cref{thm:mainintro} and \Cref{thm:stackintro} from the mod $p$ version. In \Cref{sec:generic}, we prove \Cref{thm:genericintro} and \Cref{thm:FLintro} as consequences.

In \Cref{sec:trcont}, we turn to proving the key $p$-adic continuity theorem for $\THH$ as a cyclotomic spectrum \Cref{thm:thhcontinuity}. We then use this in \Cref{sec:crysthhk} to deduce \Cref{thm:general-crys}.

Next, we study the motivic filtration on the topological cyclic homology of $R\langle \epsilon \rangle$ in \Cref{sec:epsilon}.
In \Cref{sec:compute}, we use this to explicitly compute the mod $p$ syntomic cohomology of $\Z_p \langle \epsilon \rangle$.
Finally, we study consequences for the $K$-theory and topological cyclic homology of $\Z/p^n$ in \Cref{sec:consequences}.

\subsection{Acknowledgements}
We would like to thank Ben Antieau and Achim Krause for a tremendous number of interesting discussions related to this work. We also thank Robert Burklund for helpful discussions.
We would particularly like to thank Ben for pointing out \cite[Example 5.15]{BLprismatization} to us, which got us started on the right track. Finally, we thank Sanath
Devalapurkar for explaining to us his key technique identifying the Hodge--Tate stack of $\mathbb{Z}/p^n$ by means of the specific identity $p^n=V(p^{n-1})$ in the Witt vectors of $\mathbb{Z}/p^n$ (cf. \cite[Section 6.3]{Petrov}).

During the course of this work, Jeremy Hahn was supported by the Sloan Research Foundation.  Ishan Levy was supported by the NSF Graduate Research Fellowship under Grant No. 1745302 and the Clay Research Fellowship. Andrew Senger was supported by NSF grant DMS-2103236.

\section{Crystallinity for prismatization}\label{sec:prismatize}

In this and the following section, we will prove the following preliminary version of \Cref{thm:stackintro}.
In \Cref{sec:rapid}, we will deduce the full statement of \Cref{thm:stackintro} using a formal argument.

\begin{thm}\label{thm:stackmodp}
	The functors
	\begin{enumerate}
		\item \[X \mapsto (X^\syn)_{p=v_1^{p^n}=0} \]
		\item \[X \mapsto (X^\Nyg)_{p=v_1^{p^n}=0} \]
		\item \[X \mapsto (X^\prism)_{p=v_1^{p^{n+1}}=0}\]
		\item \[X \mapsto (X^\prism)_{p^k=(F^{n+1-k})^* I=0}\]
	\end{enumerate}
	factor through the functor $X \mapsto X_{p^{n+2}=0}$, where $X_{p^{n+2}=0}$ is regarded as derived $\Z/p^{n+2}$-scheme.
	Here, $X$ is a derived $p$-adic formal scheme, and all of the vanishing loci are taken in a derived sense.
\end{thm}

As a corollary, we obtain the same result for syntomic and derived prismatic cohomology:

\begin{cor} \label{thm:mainmodp}
	Let $n \geq 0$.
	The following functors on $\CAlg_p ^\Ani$ factor through the mod $p^{n+2}$ reduction functor $\CAlg_p ^\Ani \to \CAlg_{\Z/p^{n+2}} ^\Ani$:
	\begin{enumerate}
		\item mod $(p,v_1 ^{p^n})$ syntomic cohomology\footnote{We recall the meaning of the class $v_1$ in \Cref{dfn:prism-v1} and \Cref{rmk:v1}. This class is the same as the $v_1$ of chromatic homotopy theory.}
		\[R \mapsto \F_p (\ast) (R) /v_1 ^{p^n}\]
		\item mod $(p,v_1 ^{p^n})$ Nygaard-filtered derived prismatic cohomology
		\[R \mapsto N^{\geq \ast} \prism_{R} \{\ast\} /(p, v_1^{p^{n}})\]
		\item mod $(p,v_1 ^{p^{n+1}})$ derived prismatic cohomology
		\[R \mapsto \prism_{R} \{\ast\} /(p, v_1^{p^{n+1}})\]
		\item mod $(p^k, (F^{n+1-k})^* I)$ derived prismatic cohomology, where $k \geq 1$, $F$ is the Frobenius, and $I$ is the Hodge-Tate ideal
		\[R \mapsto \prism_{R} \{\ast\} /(p^k, (F^{n+1-k})^* I)\]
	\end{enumerate}
	%
	%  Given any $n \geq 0$, the mod $(p,v_1 ^{p^n})$ derived syntomic cohomology functor
	%  \[\F_p (\bullet) / v_1 ^{p^n} : \CAlg_p^\Ani \to \CAlg_p^{\Ani,\gr}\]
	%  factors through the functor which takes $R$ to its (derived) quotient $R/p^{n+2}$.
	%
	%  In other words, we construct a commutative diagram
	%  \begin{center}
		%    \begin{tikzcd}
			%      \CAlg_p^\Ani \ar[rr,"A \mapsto \F_p (\bullet) (A) / v_1 ^{p^n}"] \ar[dr, "A \mapsto A/p^{n+2}"] & & \CAlg_p ^{\Ani,\gr} \\
			%      & \Calg_{\Z/p^{n+2}} ^\Ani. \ar[ur] &
			%    \end{tikzcd}
		%  \end{center}
\end{cor}

\begin{proof}
	We indicate how (4) is a consequence of \Cref{thm:stackmodp}(4), as the cases (1)-(3) are essentially the same. Derived prismatic cohomology is left Kan extended from $p$-completed polynomial algebras, it suffices to prove the theorem for $p$-completed polynomial algebras.
	But this follows from combining \Cref{thm:stackmodp}(4) with \Cref{thm:prism-comparison}, since $p$-completed polynomial algebras and their mod $p^n$ reductions are $p$-quasisyntomic.
\end{proof}

%\todo{fix citations to intro}
Our goal in this section is to prove \Cref{thm:stackmodp}(4).
This will also serve as a warmup for the other parts of \Cref{thm:stackmodp}, since we will not need to engage here with Nygaard-filtered prismatization or syntomification.

\subsection{Background on prismatization}

We begin with some background on prismatization from \cite{Drinfeld, APC, BLprismatization}.

\begin{dfn}
  A \emph{Cartier--Witt divisor} over a ($p$-nilpotent) ring $R$ consists of an invertible $W(R)$-module $I$ and a $W(R)$-module map
  \[I \to W(R),\]
  such that $I \to W(R) \to R$ is nilpotent and
  $I \to W(R) \xrightarrow{\delta} W(R)$ generates the unit ideal.
	
	A morphism $(I \xrightarrow{\alpha} W(R)) \to (J \xrightarrow{\beta} W(R))$ of Cartier--Witt divisors is a commuting diagram of $W(R)$-modules
	\[
	\begin{tikzcd}
	I \arrow{r}{\alpha} \arrow{d} & W(R) \\
	J \arrow{ur}{\beta}.
	\end{tikzcd}
	\]
%
  %We define
  %$W_\prim$
  By \cite[Lemma 4.2.5]{Drinfeld}, every morphism of Cartier--Witt divisors is an isomorphism.
\end{dfn}

\begin{dfn}
  We define $\Zp^\prism$ to be the $p$-adic formal stack whose functor of points sends a $p$-nilpotent ring $R$ to its category of Cartier--Witt divisors, so that
  \[\Zp^\prism (R) = \{(I, I \to W(R))\}.\]
%
 %Prismatization of $\Zp$ is $W_\prim / W^\times$.
  %Concretely, that means...
\end{dfn}

\begin{dfn}
  We say that a Cartier--Witt divisor $I \to W(R)$ is \emph{Hodge--Tate} if the composite $I \to W(R) \to R$ is zero.
  We let $\Zp^\HT \subset \Zp^\prism$ denote the closed substack of Hodge--Tate divisors.
\end{dfn}

\begin{dfn}
  There is a Frobenius map
  $F : \Zp^\prism \to \Zp ^\prism$
  which sends a Cartier--Witt divisor
  $I \xrightarrow{\alpha} W(R)$
  to
  $F^* I \coloneqq I \otimes_{W(R),F} W(R) \xrightarrow{F^* \alpha} W(R)$.
\end{dfn}

\begin{rmk}
  Because the Frobenius map $F: W(R) \to W(R)$ agrees with the map induced by the Frobenius of $R$ when $R$ is an $\F_p$-algebra, the prismatic Frobenius
  \[F : \Zp^\prism \to \Zp^\prism\]
  is a lift of the mod $p$ Frobenius.
\end{rmk}

\begin{dfn}
  There is an animated ring stack $\G^\prism$ lying over $\Spf \Zp^\prism$, which sends a Cartier--Witt divisor
  $I \to W(R)$
  to its animated quotient $W(R) / I$.

  Given a derived $p$-adic formal scheme $X$, we define its prismatization $X^\prism \to \Spf \Zp^\prism$ via
  $X^{\prism} (R) = X(\G^\prism (R))$.
  We define its Hodge--Tate locus to be $X^{\HT} \coloneqq X^{\prism} \times_{\Z_p ^\prism} \Z_p^{\HT}$.
\end{dfn}

\begin{exm}
  Explicitly, we have $(\Z/p^n)^\prism (R) = \{(I, \alpha: I \to W(R), x \in I) \vert \alpha(x) = p^n\}$.
\end{exm}

\begin{dfn}
  There is a line bundle $\O\{1\}$ over $\Zp^\prism$ which corresponds to the Breuil--Kisin twist, see \cite[Section 2 \& Example 3.3.8]{APC} and \cite[Section 4.9]{Drinfeld}.
  Heuristically, it is defined as follows. Let $\mathcal{I}$ denote the line bundle on $\Zp^\prism$ which sends $I \to W(R)$ to $I \otimes_{W(R)} R$.
  The map $I \to W(R)$ induces a map $\mathcal{I} \to \O_{\Zp^\prism}$ cutting out the Hodge-Tate locus.
  Then $\O\{1\} \coloneqq \bigotimes_{n=0} ^\infty (F^n)^* \mathcal{I}$.
\end{dfn}

\begin{rmk}
  By definition, we have $F^* \O_{\Zp^\prism} \{1\} \cong \O_{\Zp^\prism} \{1\} \otimes \mathcal{I}^{-1}$.
  Reducing this mod $p$, we find that 
  \[\O_{(\Zp)^\prism_{p=0}} \{p\} \cong \O_{(\Zp)_{p=0} ^\prism} \{1\} \otimes \mathcal{I}_{p=0} ^{-1},\]
  hence that
  \[\mathcal{I}_{p=0} \cong \O_{(\Zp)^\prism_{p=0}} \{1-p\}.\]
\end{rmk}

\begin{dfn} \label{dfn:prism-v1}
  We define $v_1$ to be the section of $\O_{(\Zp)^\prism _{p=0}} \{p-1\}$ corresponding to the canonical map $\mathcal{I}_{p=0} \to \O_{(\Zp)^\prism _{p=0}}$ under the isomorphism
  \[\mathcal{I}_{p=0} \cong \O_{(\Zp)_{p=0}} \{1-p\}.\]
\end{dfn}

\begin{rmk}
  By definition, the closed substack $(\Zp^{\HT})_{p=0} \subset (\Zp^{\prism})_{p=0}$ agrees with $(\Zp^{\prism})_{p=v_1=0}$.
\end{rmk}

Finally, we recall the comparison between the cohomology of the prismatization and prismatic cohomology.

\begin{thm}[{{\cite[Corollaries 8.13 \& 8.17]{BLprismatization}}}] \label{thm:prism-comparison}
  Let $X$ be a qcqs $p$-quasisyntomic $p$-adic formal scheme.
  Then there is a natural isomorphism
  \[\prism_{X} \{i\} \cong R\Gamma(X^\prism, \O_{X^\prism} \{i\}).\]
\end{thm}

\subsection{Proving crystallinity}

\begin{dfn}
Given a morphism of formal stacks $A \to \mathbb{Z}_p^{\prism}$, and an integer $n \ge 0$, we denote by $F^{-n}A$ the pullback
\[
\begin{tikzcd}
F^{-n}A \arrow{r} \arrow{d} & A \arrow{d} \\
\Zp^{\prism} \arrow{r}{F^n} & \Zp^{\prism},
\end{tikzcd}
\]
where $F^{n}$ denotes the $n$-fold iterated composite of $F:\Zp^{\prism} \to \Zp^{\prism}$.
\end{dfn}

%\begin{lem}
  %The (a priori derived) $p$-adic formal stacks $(\Zp) ^\prism_{p^{k} = 0}$, $(\Z/p^n) ^\prism_{p^{k} = 0} $ and $F^{-1} (\Z/p^n)$ are classical.
%\end{lem}
%
%\begin{proof}
  %For $(\Zp ^\prism)_{p^k =0}$, it suffices to check this after pullback along the flat cover $\Spf \Z_p [[x]] \to \Z_p ^{\prism}$, where it easy.
  %
  %For , we argue using the
%\end{proof}
%
%This means that we need only work with classical rings when making arguments about the functor of points.

Our main result is as follows:

\begin{thm} \label{thm:prism-main}
  For integers $n \geq 0$ and $k \geq 1$, there exist dashed arrows making the following diagram commute:
  \begin{center}
    \begin{tikzcd}
      F^{-n} \left( (\Zp^{\HT})_{p^{k+1} = 0} \right)\ar[rr] \ar[rd, dashed] & & (\Zp)^{\prism}_{p^{k+1}=0} \\
       & (\Z/p^{n+k+1})_{p^{k+1}=0}^\prism  \ar[ru], &
    \end{tikzcd}
  \end{center}
such that the arrows are compatible as $k$ and $n$ vary.
\end{thm}

\begin{cor}
  Taking $k=1$ and reducing modulo $p$, we obtain the diagram
\[
    \begin{tikzcd}
      (\Zp)^{\prism}_{p = v_1^{p^n} = 0} \ar[rr] \ar[rd] & & (\Zp)_{p=0}^\prism \\
       & (\Z/p^{n+2})_{p=0}^\prism.  \ar[ru] &
    \end{tikzcd}
\]
\end{cor}

Using \Cref{thm:prism-main}, we are able to prove \Cref{thm:stackmodp}(4).

\begin{proof}[Proof of \Cref{thm:stackmodp}(4)]
  Letting $I$ denote the Hodge--Tate ideal, we note that $\Zp^{\HT} = (\Zp^{\prism}) _{I=0}$ by definition, so that we may apply \Cref{thm:prism-main} to
  form the diagram
  \begin{center}
    \begin{tikzcd}
      (X^\prism)_{p^k=(F^{n+1-k})^* I = 0} \ar[d] \ar[r] & (\Zp^\prism)_{p=(F^{n+1-k})^* I = 0} \ar[d] \\
      (X_{p^{n+2}=0})^{\prism} \ar[d] \ar[r] & (\Z/p^{n+2})^\prism \ar[d] \\
      X^\prism \ar[r] & \Zp^\prism.
    \end{tikzcd}
  \end{center}
  The outer square is a pullback by definition, and the bottom square is a pullback because prismatization preserves limits of derived p-adic formal schemes.
  It follows that the top square is also a pullback, so that 
  \[(X^\prism)_{p=(F^{n+1-k})^* I = 0} \cong (X_{p^{n+2}=0})^\prism \times_{(\Z/p^{n+2})^\prism} (\Zp^\prism)_{p=(F^{n+1-k})^* I = 0}\]
  gives the desired factorization of
  \[X \mapsto (X^\prism)_{p=(F^{n+1-k})^* I = 0}\]
  through
  \[X \mapsto X_{p^{n+2}=0}.\]
\end{proof}

Our proof of \Cref{thm:prism-main} will be an induction on $n$.
The base inductive case is provided by the following result of Bhatt--Lurie:

\begin{prop} \label{prop:deRham-point}
 There is a commutative diagram:
  \begin{center}
    \begin{tikzcd}
      \Z_p^{\HT} \ar[r] \ar[d] & \Z_p ^\prism \ar[d, "F"] \\
      \F_p^{\prism} \ar[r] & \Z_p^\prism
    \end{tikzcd}
  \end{center}
In particular, the natural map $\Z_p^{\HT} \subset \Z_p^\prism$ factors through $F^{-1} ( \F_p^{\prism} )$.
\end{prop}

\begin{proof}
This square is stated as \cite[Proposition 3.6.6]{APC}, once one identifies \[\rho_{dR}:\Spf \Zp \to \Zp^{\prism}\] with the natural map $\F_p^{\prism} \to \Zp^{\prism}$ as in \cite[Remark 3.13]{BLprismatization}.
\end{proof}

The inductive step is provided by the following lemma:

\begin{lem} \label{lem:prism-contract}
  For each $n \geq 1$, there is a commuting diagram:
  \begin{center}
    \begin{tikzcd}
      F^{-1} (\Z/p^n)^{\prism} _{p^{n+1}=0} \ar[d, dashed] \ar[r] & (\Z/p^n)^\prism _{p^{n+1}=0} \ar[dd] \\
      (\Z/p^{n+1})^\prism _{p^{n+1}=0} \ar[d] & \\
      (\Z_p)^{\prism} _{p^{n+1}=0} \ar[r, "F"] & (\Z_p)^\prism _{p^{n+1}=0},
    \end{tikzcd}
  \end{center}
  where the outer square is a pullback.
\end{lem}

\begin{proof}
  Let $(I \xrightarrow{\alpha} W(S), x) \in F^{-1} (\Z/p^n)^{\prism} (S)$, where $x \in F^* I$ is a lift of $p^n \in W(S)$.

  Now, we have a commutative square
  \begin{center}
    \begin{tikzcd}
      I \ar[d, "\alpha"] & F^* I \ar[l, "V"] \ar[d, "F^*\alpha"] \\
      W(S) & W(S) \ar[l, "V"],
    \end{tikzcd}
  \end{center}
  where $V$ is induced by the Witt vector Verschiebung.
  Then $V(x) \in I$ is a lift of $V(p^n) \in W(S)$. By assumption $p^{n+1} = 0$ on $S$, so that $V(p^n) = p^{n+1} \in W(S)$ (by \cite[Lemma 6.14]{Petrov}, following ideas of Devalapurkar), and $V(x)$ is the desired lift of $p^{n+1} \in W(S)$.
\end{proof}

We are now ready to prove \Cref{thm:prism-main}:

\begin{proof}[Proof of \Cref{thm:prism-main}]
	We will induct on $n \ge 0$ to construct a dashed arrow
	  \begin{center}
    \begin{tikzcd}
      F^{-n} \left( (\Zp^{\HT})_{p^{k+1} = 0} \right)\ar[rr] \ar[rd, dashed] & & (\Zp)^\prism _{p^{k+1} = 0} \\
       & (\Z/p^{n+k+1})^\prism _{p^{k+1} = 0}  \ar[ru], &
    \end{tikzcd}
  \end{center}
	for fixed $k \ge 1$.  To begin, we first construct the dashed arrow when $n=0$.  Reducing \Cref{prop:deRham-point} modulo $p^{k+1}$, there is a diagram
	\[
	\begin{tikzcd}
	(\Zp^{\HT})_{p^{k+1}=0} \ar[r] \ar[d,dashed] & (\Zp^{\prism})_{p^{k+1}=0} \\
	F^{-1} ((\F_p^{\prism})_{p^{k+1}=0}) \ar[ur]
	\end{tikzcd}
	\]
	Since the map $F^{-1} (\F_p^{\prism}) \to \Zp^{\prism}$ factors through $F^{-1} (\Z/p^k)^{\prism}$, the above commutative diagram in particular induces a diagram
	\[
	\begin{tikzcd}
	(\Zp^{\HT})_{p^{k+1}=0} \ar[d] \ar[r] & (\Zp^{\prism})_{p^{k+1}=0} \\
	F^{-1}((\mathbb{Z}/p^{k})^{\prism}_{p^{k+1}=0}) \ar[ur]
	\end{tikzcd}
	\]
	Applying \Cref{lem:prism-contract}, we conclude the desired diagram when $n=0$.
	
	Now we proceed with the inductive step.  In particular, suppose for the sake of induction that we have constructed a diagram 
		  \begin{center}
    \begin{tikzcd}
      F^{-n+1} \left( (\Zp^{\HT})_{p^{k+1} = 0} \right)\ar[rr] \ar[rd, dashed] & & (\Zp)^\prism _{p^{k+1} = 0} \\
       & (\Z/p^{n+k})^\prism _{p^{k+1} = 0}  \ar[ru], &
    \end{tikzcd}
  \end{center}
	Applying $F^{-1},$ we obtain a diagram
			  \begin{center}
    \begin{tikzcd}
      F^{-n} \left( (\Zp^{\HT})_{p^{k+1} = 0} \right)\ar[rr] \ar[rd, dashed] & & (\Zp)^\prism _{p^{k+1} = 0} \\
       & F^{-1} (\Z/p^{n+k})^\prism _{p^{k+1} = 0}  \ar[ru], &
    \end{tikzcd}
  \end{center}
	The desired diagram is then constructed by applying \Cref{lem:prism-contract}.
\end{proof}

\section{Crystallinity for syntomification}\label{sec:Nygaard}

The argument we give for \Cref{thm:stackmodp}(1-3) here is very much analogous to the argument given for \Cref{thm:stackmodp}(4) in the previous section, though we must deal with more complicated formal stacks.

\subsection{Background on Nygaardification and syntomification}

Here, we recall some background on the Nygaard-filtered prismatization from \cite{Drinfeld,fgauge}.
Let $W_S$ denote the Witt ring scheme over a base scheme $S$.

\begin{dfn}
  A (relatively) affine $W_S$-module $M$ is said to be \emph{admissible} if there exists an $S$-line bundle $L_M$, an invertible $W_S$-module scheme $M'$, and a short exact sequence
  \[0 \to L_M ^\# \to M \to F_* M' \to 0.\]
  By \cite[Lemma 3.12.7]{Drinfeld}, $L_M$, $M'$ and the short exact sequence may be recovered functorially from $M$.
\end{dfn}

\begin{exm} \label{exm:defineF'}
  Any invertible $W_S$-module $I$
  may be tensored with the exact sequence
  \[0 \to \G_a ^\# \to W_S \xrightarrow{F} F_* W_S \to 0 \]
  to obtain an exact sequence
  \[0 \to I \otimes_{W_S} \G_a ^\# \to I \to I \otimes_{W_S} F_* W_S \to 0.\]
  Since $I \otimes_{W_S} F_* W_S \cong F_* (I \otimes_{W_S,F} W_S)$, $I$ is an admissible $W_S$-module.
\end{exm}

\begin{exm} \label{exm:start-dR}
  Given a map $d : I \to W_S$, we may define an admissible $W_S$-module $M$ as a pullback in the following diagram:
  \begin{center}
    \begin{tikzcd}
      0 \ar[r] & \G_a ^\# \ar[r] \ar[d,"="] & M \ar[r] \ar[d] & F_* I \ar[r] \ar[d,"F_* d"] & 0 \\
      0 \ar[r] & \G_a ^\# \ar[r] & W_S \ar[r,"F"] & F_* W_S \ar[r] & 0.
    \end{tikzcd}
  \end{center}
\end{exm}

\begin{rmk}
  It is not true in general that if $M$ is an admissible $W_S$-module, then $F^* M = M \otimes_{W_S, F} W_S$ is an admissible $W_S$-module. Correspondingly, the Nygaard-filtered prismatization of $\Z_p$ will \emph{not} admit a lift of the mod $p$ Frobenius.

  If $S$ is an $\F_p$-scheme, then we may consider $\Fr_S ^* M$, which is an admissible $W_S$-module.
\end{rmk}

\begin{lem}
  Admissible $W_S$-modules $M$ are affine commutative flat group schemes over $S$.
\end{lem}

\begin{proof}
  We just need to prove the flatness of $M$.
  Zariski locally, $M$ lies in a short exact sequence
  \[0 \to \G_a ^\# \to M \to F_* W_S \to 0,\]
  so this follows from flatness of $\G_a ^\#$ and $F_* W_S$.
  %The affineness and flatness are what need to be checked.
  %These may be checked fpqc locally on $S$, so that it is true when $M$ is an invertible $W$-module and we may assume $M$ fits into a pullback square (CITE):
  %\begin{center}
    %\begin{tikzcd}
      %M \ar[r] \ar[d] & F_* I \ar[d] \\
      %W \ar[r, "F"] & F_* W 
    %\end{tikzcd}
  %\end{center}
  %Flatness and affineness of $M$ now follow from flatness and affiness of $W \to S$ and $F$.
\end{proof}

Applying \cite[Section IV.3.4]{DG}, we obtain the following corollary:

\begin{cor} \label{cor:V-exists}
  For $\F_p$-schemes $S$, admissible $W_S$-modules are equipped with natural Frobenius and Verschiebung homomorphisms
  \[F : M \to \Fr_S ^* M, \, \, V : \Fr_S ^* M \to M\]
  satisfying the relation $FV = VF = p$.
\end{cor}

\begin{dfn}
  A \emph{filtered Cartier--Witt divisor} over $S$ consists of an admissible $W_S$-module $M$ and a $W_S$-module map $d : M \to W_S$ which fits into a map of short exact sequences
  \begin{center}
    \begin{tikzcd}
      0 \ar[r] & L_M ^\# \ar[r] \ar[d,"d^\#"] & M \ar[r] \ar[d,"d"] & F_* M' \ar[r] \ar[d,"F_* d'"] & 0 \\
      0 \ar[r] & \G_a ^\# \ar[r] & W_S \ar[r,"F"] & F_* W_S \ar[r] & 0,
    \end{tikzcd}
  \end{center}
  where $d' : M' \to W_S$ is a Cartier--Witt divisor.
\end{dfn}

\begin{dfn}
  We define $\Zp ^\Nyg$ to be the formal stack whose functor of points sends a $p$-nilpotent scheme $S$ to its underlying groupoid of filtered Cartier--Witt divisors, so that
  \[\Zp^\Nyg (S) = \{(M, M \to W_S)\}.\]
\end{dfn}

\begin{dfn}
  The assignment $(d : M \to W_S) \mapsto (d' : M' \to W_S)$ in the above definition defines a morphism of stacks
  \[F' : \Zp ^\Nyg \to \Zp ^\prism.\]
\end{dfn}

\begin{dfn}
  By \cite[Lemma 3.12.4(ii)]{Drinfeld}, the map $d^\# : L_M ^\# \to \G_a ^\#$ corresponds to a unique map $a(d) : L_M \to \G_a$
  The assignment $(d : M \to W_S) \mapsto (a(d) : L_M \to \G_a)$ determines a morphism of stacks
  \[a : \Zp^\Nyg \to \A^1/\G_m.\]
\end{dfn}

\begin{dfn}
  We have the animated ring stack $\G^\Nyg$ over $\Zp ^\Nyg$ given by the animated quotient $W_S / M$ of fpqc sheaves.\footnote{Any filtered Cartier--Witt divisor is a quasi-ideal by \cite[Lemma 3.12.12]{Drinfeld}.}
  Given a derived $p$-adic formal scheme $X$, we define its Nygaard-filtered prismatization $X^\Nyg \to \Spf \Zp^{\Nyg}$ via
  $X^\Nyg (R) = X (\G^\Nyg (R))$.
\end{dfn}

\begin{exm}
  For example, $(\Z/p^n)^\Nyg$ is the fpqc sheafification of the functor $R \mapsto \{(M, d : M \to W_R, x \in M(R) \vert d(x) = p^n \in W(R))\}$.
\end{exm}

\begin{dfn}
  By \Cref{exm:defineF'}, any Cartier--Witt divisor is also a filtered Cartier--Witt divisor. This determines a morphism
  $j_{HT} : \Zp ^\prism \to \Zp^\Nyg,$
  which turns out to be an open immersion, and whose composition with
  $F' : \Zp ^\Nyg \to \Zp ^\prism$
  is the Frobenius map on $\Zp^\prism$.
\end{dfn}

\begin{dfn}
  Given a Cartier--Witt divisor $ d : I \to W_S$, we may construct a filtered Cartier--Witt divisor by taking $M$ to be a pullback in the following diagram:
  \begin{center}
    \begin{tikzcd}
      0 \ar[r] & \G_a ^\# \ar[r] \ar[d, "\mathrm{id}"] & M \ar[r] \ar[d] & F_* I \ar[r] \ar[d, "F_* d"] & 0 \\
      0 \ar[r] & \G_a ^\# \ar[r] & W_S \ar[r,"F"] & F_* W_S \ar[r] & 0
    \end{tikzcd}
  \end{center}
  By \Cref{exm:start-dR}, this determines a morphism
  \[j_{dR} : \Zp ^\prism \to \Zp^\Nyg.\]
  This identifies $\Zp ^\prism$ with $(\Zp^\Nyg)_{a \neq 0}$, the fiber of
  $a : \Zp^\Nyg \to \A^1 / \G_m$ over $\G_m / \G_m$.
\end{dfn}

\begin{dfn} \label{dfn:NygBundles}
  Pulling back $\O_{\Zp ^\prism} \{1\}$ along $F' : \Zp ^\Nyg \to \Zp ^\prism$
  and $\O_{\A^1/\G_m} (1)$ along $a : \Zp ^\Nyg \to \A^1 / \G_m$,
  we obtain a bigraded family of line bundles on $\Zp ^\Nyg$:
  $(F') ^* \O_{\Zp ^\prism} \{i\} \otimes a^* \O_{\A^1/\G_m} (j)$.

  We set
  \[\O_{\Zp ^\Nyg} \{1\} \coloneqq (F') ^* \O_{\Zp ^\prism} \{1\} \otimes a^* \O_{\A^1/\G_m} (-1).\]
\end{dfn}

\begin{dfn}
  Let $S$ be an $\F_p$-scheme and $M$ be an admissible $W_S$-module.
  Then the Frobenius induces a $W_S$-linear map
  $F : M \to F_* (\Fr_S ^* M)$.
  The map $F$ kills $L_M ^\#$, hence factors through a map 
  $F_* M' \to F_* (\Fr_S ^* M),$
  which gives rise to a map
  $f_M : M' \to \Fr_S ^* M.$

  This determines a natural map $\phi_M : L_{M'} \to L_{\Fr_S ^* M} \cong L_M ^{\otimes p},$
  i.e. a map of line bundles
  \[\mu : (F')^{*} \mathcal{I}_{p=0} \to a^* \O_{\A^1/ \G_m, p=0} (-p).\]

  We define $v_1$ to be the section of $\O_{(\Zp)^\Nyg_{p=0}} \{p-1\} \cong (F')^*  \mathcal{I}_{p=0} ^{-1} \otimes a^* \O_{\A^1/ \G_m, p=0} (1-p)$ corresponding to the map
  \[a \mu : (F')^{*} \mathcal{I}_{p=0} \to t^* \O_{\A^1/ \G_m, p=0} (1-p).\]
\end{dfn}

\begin{rmk} \label{rmk:v1}
  The open immersion $j_{HT, p=0} : (\Zp) ^\prism _{p =0} \to (\Zp) _{p=0} ^\Nyg$ identifies $(\Zp)_{p=0} ^\prism$ with $(\Zp ^\Nyg)_{p=0, \mu \neq 0}$.

  When pulled back along $j_{HT}$, $a$ becomes identified with the canonical map $\mathcal{I} \to \O_{\Zp ^{\prism}}$.

  Similarly, when pulled back along $j_{dR, p = 0}$, $\mu$ becomes identified with $\mathcal{I} \to \O_{\Zp ^{\prism}}$.

  It follows that $v_1 = a \mu$ pulls back along both $j_{HT, p=0}$ and $j_{dR, p=0}$ to the section $v_1$ of $\O_{(\Zp)^{\Nyg} _{p=0}} \{p-1\}$ identified in \Cref{dfn:prism-v1}.
\end{rmk}

%Finally, we have the comparison theorem, which will hopefully appear in future work of BHATT-LURIE. \todo{wat to do}

\subsection{Drinfeld's description of $(\Zp^{\Nyg})_{p=a^{p}\mu=0}$}

In this section, we recall a compact presentation of $(\Zp^{\Nyg})_{p=a^{p}\mu=0}$ due to Drinfeld \cite[Section 7.4]{Drinfeld}.

\begin{rec}
  Recall the stack $\A^1 / \G_m$, whose points are given by
  $\A^1 / \G_m (S) = (\L, a : \L \to \O_S),$ where $\L$ is a line bundle on $S$.
\end{rec}

\begin{dfn} \label{dfn:T}
  We define a stack $T$ over $(\A^1 / \G_m)_{p=0}$ by
  $T (S) = (\L, a : \L \to \O_S, u : \O_S \to \L^{\otimes p}), a^p u = 0).$
\end{dfn}

\begin{dfn} \label{dfn:Gmt}
  We define a group scheme $\G_{m,a} ^\#$ over $(\A^1 / \G_m)_{p=0}$ by
  \[\G_{m,a}^\# (S) = \left\{ \alpha \in (\L^{\otimes p})^\# (S) \vert (1-(a^p)^\#) (\alpha) \in \G_m ^\# (S) \right\}\]
  with group operation
  $\alpha_1 * \alpha_2 = \alpha_1 + \alpha_2 - [a^p] (\alpha_1) \alpha_2.$
  %\[
  %\G_{m,t} (S) =
  %\left\{ \begin{pmatrix} 1 & \alpha \\ 0 & 1-[t^{\otimes p}] (\alpha)  \end{pmatrix}, \alpha \in (\L^{\otimes p})^{\#} \text{ such that } 1-[t^{\otimes p}] (\alpha) \in \G_m ^{\#} \right\}.
  %\]
\end{dfn}

\begin{dfn}
  There is a map $T \to \Zp^\Nyg$ over $\A^1 / \G_m$ given by sending $(\L, a, u)$ to the admissible module $M$ defined by the following pullback square:
  \begin{center}
    \begin{tikzcd}
      0 \ar[r] & \L^{\#} \ar[r] \ar[d,"\cong"] & M \ar[r] \ar[d] & F_* W_S \ar[r] \ar[d, "{F_* [u]}"] & 0 \\
  
      0 \ar[r] & \L^{\#} \ar[r] & {[}\L{]} \ar[r, "F"] & {[}\L{]} \otimes_{W_S} F_* W_S \ar[r] & 0,
    \end{tikzcd}
  \end{center}
  which is made into a filtered Cartier--Witt divisor via the map $\xi : M \to [\L] \oplus F_* W_S \xrightarrow{[a], V} W_S$, which fits into the diagram:
  \begin{center}
    \begin{tikzcd}
      0 \ar[r] & \L^{\#} \ar[r] \ar[d,"a^\#"] & M \ar[r] \ar[d,"d"] & F_* W_S \ar[r] \ar[d, "p"] & 0 \\

      0 \ar[r] & \G_a ^{\#} \ar[r] & W_S \ar[r] & F_* W_S \ar[r] & 0.
    \end{tikzcd}
  \end{center}
  Here $[\L]$ is the Teichm\"uller lift of the line bundle $\L$ to an invertible $W_S$-module.
\end{dfn}

Endow the morphism $T \to \Zp ^\Nyg$ with the action of $\G_{m,t} ^\#$ which is trivial on $T$ and $\Zp ^{\Nyg}$ and acts on the filtered Cartier--Witt divisor $d : M \to W_S$ via the action on $M \hookrightarrow [\L] \oplus F_* W_S$ given by
\[ \alpha \cdot (x,y) = (x + V(\alpha y), (1-[a^{\otimes p}] (\alpha)) y).  \]
 
\begin{prop}[{{\cite[Corollary 7.4.6]{Drinfeld}}}]
  The induced morphism $T / \G_{m,t} ^\# \to \Zp ^\Nyg$ identifies the source with $(\Zp^{\Nyg}) _{p=a^p \mu=0}$.
\end{prop}

\begin{rmk}
  If we let $S \hookrightarrow T$ denote the closed substack cut out by the replacing the equation $a^p \mu = 0$ with $a \mu =0$, then we find that $S / \G_{m,t} ^\#$ identifies with $(\Zp^{\Nyg}) _{p=v_1=0}$.

  Additionally, $S \to \Z_p ^{\Nyg}$ may be identified with $(\F_p)^{\Nyg} _{p=0} \to \Zp ^\Nyg$, as follows from examining \cite[Section 5.4]{fgauge}.
\end{rmk}

%
%\begin{rmk}
%There is an equivalence
%\[(F')^{-1}(\Zp^{\HT})_{p=0} = (\Zp^{\Nyg})_{p=t^p\mu=0}\] 
%\end{rmk}
%
%
%\subsection{Trying to understand Andy/Drinfeld}
%
%\begin{rec}
%  Recall the stack $\A^1 / \G_m$, whose points are given by
%  $\A^1 / \G_m (S) = (\L, a : \L \to \O_S).$
%\end{rec}
%
%\begin{dfn}
%We define a stack $\mathcal{S}$ over $\A^1 / \G_m$ by
%\[\mathcal{S}(S) = (\L,a:\L \to \O_S, \mu:\O_S \to \L^{\otimes p},a \mu = 0).\] 
%\end{dfn}
%
%
%Note that there are two substacks of $\mathcal{S}$, defined by the equations $a \ne 0$ and $\mu \ne 0$, which are isomorphic to $\Spec(\mathbb{F}_p)$  and $\Spec(\mathbb{F}_p)/\mu_p$.
%
%\begin{dfn}
%  There is a map $T \to \Zp^\Nyg$ over $\A^1 / \G_m$ given by sending $(\L, t, u)$ to the admissible module $M$ defined by the following pullback square:
%  \begin{center}
%    \begin{tikzcd}
%      0 \ar[r] & \L^{\#} \ar[r] \ar[d,"\cong"] & M \ar[r] \ar[d] & F_* W_S \ar[r] \ar[d, "{F_* [u]}"] & 0 \\
%  
%      0 \ar[r] & \L^{\#} \ar[r] & {[}\L{]} \ar[r, "F"] & {[}\L{]} \otimes_{W_S} F_* W_S \ar[r] & 0,
%    \end{tikzcd}
%  \end{center}
%  which is made into a filtered Cartier--Witt divisor via the map $\xi : M \to [\L] \oplus F_* W_S \xrightarrow{[t], V} W_S$, which fits into the diagram:
%  \begin{center}
%    \begin{tikzcd}
%      0 \ar[r] & \L^{\#} \ar[r] \ar[d,"t"] & M \ar[r] \ar[d,"\xi"] & F_* W_S \ar[r] \ar[d, "p"] & 0 \\
%
%      0 \ar[r] & \G_a ^{\#} \ar[r] & W_S \ar[r] & F_* W_S \ar[r] & 0.
%    \end{tikzcd}
%  \end{center}
%\end{dfn}

\subsection{Syntomificiation}

\begin{dfn}
  We define the syntomification $X^\Syn$ of a derived $p$-adic formal scheme $X$ via the pushout square:
  \begin{center}
    \begin{tikzcd}[column sep = large, row sep = large]
      X^\prism \coprod X^\prism \ar[d, "\Delta"] \ar[r,"{(j_{HT}, j_{dR})}"] & X^{\Nyg} \ar[d, "j_\Nyg"] \\
      X^\prism \ar[r,"j_{\prism}"] & X^\Syn.
    \end{tikzcd}
  \end{center}
\end{dfn}

\begin{dfn}
  The line bundles $\O_{\Zp ^\Nyg} \{1\}$ and $\O_{\Zp ^\prism} \{1\}$ glue to a line bundle $\O_{\Zp ^\syn}$ on $\Zp ^\syn$.
  Moreover, the sections $v_1$ of $\O_{(\Zp) ^\Nyg _{p =0}} \{p-1\}$ and $\O_{(\Zp) ^\prism _{p =0}} \{p-1\}$ glue to give a section $v_1$ of $\O_{(\Zp) ^\syn _{p =0}} \{p-1\}$.
\end{dfn}

Finally, the following theorem compares Nygaard filtered prismatic cohomology and syntomic cohomology to the coherent cohomology of the Nygaardification and syntomifications, and will appear in forthcoming work of Bhatt--Lurie (see also \cite[Theorem 5.5.10 and Remark 5.5.18]{fgauge}).

\begin{thm}[\cite{fgaugeBL}] \label{thm:synandnyg-comparison}
  Let $X$ be a qcqs $p$-quasisyntomic $p$-adic formal scheme.
  Then there is a natural isomorphism
  \[\Zp (i) (X) \cong R\Gamma(X^\syn, \O_{X^{\syn}} \{i\}).\]
  	and a natural isomorphism
  \[\N^{\geq i} \prism_{X} \{i\} \cong R\Gamma(X^\Nyg, \O_{X^{\Nyg}} \{i\}).\]
\end{thm}

\subsection{The base case}

The main purpose of this subsection is to produce and study the following diagram, which serves as a base case to our later inductive proof of \Cref{thm:mainmodp}(1-3) and \Cref{thm:stackmodp}(1-3).

\begin{cnstr}\label{cnstr:nygfactor}
There is a commutative diagram
\[
\begin{tikzcd}
 (\Zp)^{\Nyg}_{p=v_1=0} \ar[rr] \ar[dd,"F"] & & (\Zp)_{p=a^{p}\mu=0}^{\Nyg} \ar[ld,dashed] \ar[dd,"F"]  \ar[r] & (\Zp)_{p=0}^{\Nyg} \ar[dd,"F"]\\ 
 &(\mathbb{F}_p)^{\Nyg}_{p=0} \ar[ld,dashed] \\
 (\Zp)^{\Nyg}_{p=v_1=0} \ar[rr] & & (\Zp)_{p=a^{p}\mu=0}^{\Nyg} \ar[r] & (\Zp)_{p=0}^{\Nyg},
\end{tikzcd}
\]
such that the composite map $(\mathbb{F}_p)^{\Nyg}_{p=0} \to (\mathbb{Z}_p)^{\Nyg}_{p=0}$ is the mod $p$ reduction of the functor $(\--)^{\Nyg}$ applied to the ring homomorphism $\mathbb{Z}_p \to \mathbb{F}_p$.
\end{cnstr}

\begin{proof}
Suppose $\mathcal{G}$ is a group scheme over $\mathbb{F}_p$ on which the Frobenius map is constant at the identity. Then, if $\mathcal{G}$ acts trivially on a stack $\mathcal{T}$, there is a natural diagram
\[
\begin{tikzcd}
\mathcal{T} \arrow{d}{F} & \mathcal{T}/\mathcal{G} \arrow{d}{F}  \arrow{l} \\
\mathcal{T} \arrow{r} &  \mathcal{T}/\mathcal{G}.
\end{tikzcd}
\]

%Suppose $\mathcal{G}$ is a group scheme over $\mathbb{F}_p$, and let $\mathcal{G}[F]$ denote the kernel of Frobenius on $\mathcal{G}$.  Then, if $\mathcal{G}[F]$ acts trivially on a stack $\mathcal{T}$, there is a natural diagram
%\[
%\begin{tikzcd}
%\mathcal{T} \arrow{d}{F} & \mathcal{T}/\mathcal{G}[F] \arrow{d}{F}  \arrow{l} \\
%\mathcal{T} \arrow{r} &  \mathcal{T}/\mathcal{G}[F].
%\end{tikzcd}
%\] 

  In the case that $\mathcal{G} = \G_{m,t} ^\#$ and $\mathcal{T} = T$ are as in \Cref{dfn:Gmt} and \Cref{dfn:T}, we obtain a diagram
\[
\begin{tikzcd}
T \arrow{d}{F} & (\Zp)_{p=a^{p}\mu=0}^{\Nyg} \arrow{d}{F}  \arrow{l} \\
T \arrow{r} & (\Zp)_{p=a^{p}\mu=0}^{\Nyg} .
\end{tikzcd}
\] 
Similarly, we have
\[
\begin{tikzcd}
S \arrow{d}{F} & (\Zp)_{p=v_1=0}^{\Nyg} \arrow{d}{F}  \arrow{l} \\
S \arrow{r} & (\Zp)_{p=v_1=0}^{\Nyg} .
\end{tikzcd}
\] 

  Using the identification of $S \to \Z_p ^{\Nyg}$ with $(\F_p)^{\Nyg} _{p=0} \to \Zp^{\Nyg}$, it therefore suffices to note that there is a further factorization,
\[
\begin{tikzcd}
S \arrow[dashed]{rd} & T \arrow{d}{F} \arrow[dashed]{l}  \\
& T
\end{tikzcd}
\]
  coming from the factorization of graded rings
\[
\begin{tikzcd}
  \F_p [\mu,a]/(\mu a) \arrow[dashed]{r} & \F_p [\mu,a]/(\mu a^p) \\
  &\F_p [\mu,a]/(\mu a^p). \arrow[dashed]{lu} \arrow{u}{F}
\end{tikzcd}
\]
\end{proof}

In order to understand prismatization or syntomification, and not only Nygaardification, it is necessary to comment on the restrictions of the above diagram under the open immersions 
\[j_{HT},j_{dR}:\Zp^{\prism} \to \Zp^{\Nyg}.\]
We do so in the remainder of this section.

\begin{prop}\label{prop:restoprism}
The restriction of the subdiagram
\[
\begin{tikzcd}
   & (\Zp)_{p=a^{p}\mu=0}^{\Nyg} \ar[ld] \ar[dd,"F"]  \ar[r] & (\Zp)_{p=0}^{\Nyg} \ar[dd,"F"]\\ 
 (\mathbb{F}_p)^{\Nyg}_{p=0} \ar[rd] \\
  & (\Zp)_{p=a^{p}\mu=0}^{\Nyg} \ar[r] & (\Zp)_{p=0}^{\Nyg},
\end{tikzcd}
\]
under $j_{HT}$ is a diagram 
\[
\begin{tikzcd}
   & (\Zp)_{p=v_1^p=0}^{\prism} \ar[ld] \ar[dd,"F"]  \ar[r] & (\Zp)_{p=0}^{\prism} \ar[dd,"F"]\\ 
 (\mathbb{F}_p)^{\prism}_{p=0} \ar[rd] \\
  & (\Zp)_{p=v_1^p=0}^{\prism} \ar[r] & (\Zp)_{p=0}^{\prism}.
\end{tikzcd}
\]
\end{prop}

\begin{proof}
This follows from the facts that the maps in \Cref{cnstr:nygfactor} are compatible with inverting $\mu$, and that the restriction of $(\Zp)^{\Nyg}_{p=a^p \mu=0}$ along $j_{HT}$ is $(\Zp)^{\prism}_{p=v_1^p=0}$.
\end{proof}

%\begin{rmk}
%Recall that $v_1$ is a unit multiple of $a \mu$ on $(\Zp)^{\Nyg}_{p=0}$.  In particular, there is a natural map 
%\[(\Zp)^{\Nyg}_{p=v_1=0} = (\Zp)^{\Nyg}_{p=a \mu=0} \to (\Zp)^{\Nyg}_{p=a^p \mu=0}.\]
%combines with Proposition ???  to produce a commutative diagram
%\[
%\begin{tikzcd}
%& (\Zp)_{p=v_1=0}^{\Nyg} \ar[ld] \ar[d,"F"]  \ar[r] & (\Zp)_{p=0}^{\Nyg} \ar[d,"F"]\\
%(\mathbb{F}_p)^{\Nyg}_{p=0} \ar[r] & (\Zp)_{p=v_1=0}^{\Nyg} \ar[r] & (\Zp)_{p=0}^{\Nyg}.
%\end{tikzcd}
%\]
%\end{rmk}

\begin{prop} \label{prop:syn-base}
The restrictions of the subdiagram 
\[
\begin{tikzcd}
  (\Zp)^{\Nyg}_{p=v_1=0} \ar[d,"F"]  \ar[r] & (\Zp)_{p=0}^{\Nyg} \ar[d,"F"]\\
  (\mathbb{F}_p)^{\Nyg}_{p=0} \ar[r] & (\Zp)_{p=0}^{\Nyg},
\end{tikzcd}
\]
along $j_{HT}$ and $j_{dR}$ are given by
\[
\begin{tikzcd}
\Spec(\mathbb{F}_p) / (W[F] \times \mu_p) \ar[r] \ar[d] & (\Zp)_{p=0}^{\prism} \ar[d,"F"] \\
\Spec(\mathbb{F}_p) \ar[r] & (\Zp)_{p=0}^{\prism},
\end{tikzcd}
\]
and
\[
\begin{tikzcd}
\Spec(\mathbb{F}_p) / ((W^{\times})[F]) \ar[r] \ar[d] & (\Zp)_{p=0}^{\prism} \ar[d,"F"] \\
\Spec(\mathbb{F}_p) \ar[r] & (\Zp)_{p=0}^{\prism},
\end{tikzcd}
\]
respectively.  These diagrams are identified under the isomorphism $W[F] \times \mu_p \cong (W^{\times})[F]$ of \cite[Lemma 3.3.4]{Drinfeld}. Thus, there exists a diagram
\[
\begin{tikzcd}
(\Zp)^{\Syn}_{p=v_1=0} \ar[r] \ar[d] & (\Zp)_{p=0}^{\Syn} \ar[d,"F"] \\
(\mathbb{F}_p)^{\Syn}_{p=0} \ar[r] & (\Zp)_{p=0}^{\Syn}
\end{tikzcd}
\]
\end{prop}

\begin{proof}
This is a consequence of \cite[Lemma 7.5.2 and Remark 7.5.3]{Drinfeld}.
\end{proof}

\subsection{A factorization via Verschiebung}

Our goal in this subsection is to construct the diagram of the following proposition, which will be the key tool powering our inductive understanding of $\Z/p^{n+1}$ in terms of $\Z/p^n$:

\begin{prop} \label{prop:nyg-contract}
  There is a commuting diagram:
  \begin{center}
    \begin{tikzcd}
      F^{-1} (\Z/p^n)^{\Nyg} _{p=0} \ar[dashed,d] \ar[r] & (\Z/p^n)^\Nyg _{p=0} \ar[dd] \\
      (\Z/p^{n+1})^\Nyg _{p=0} \ar[d] & \\
      (\Z_p)^{\Nyg} _{p=0} \ar[r, "F"] & (\Z_p)^\Nyg _{p=0},
    \end{tikzcd}
  \end{center}
  where the outer square is a pullback.
  Moreover, the restrictions of this diagram under the two open immersions
  \[j_{HT}, j_{dR} : (\Z_p)^\prism _{p=0} \hookrightarrow (\Z_p)^\Nyg _{p=0}\]
	are isomorphic, so that there exists a diagram
  \begin{center}
    \begin{tikzcd}
      F^{-1} (\Z/p^n)^{\Syn} _{p=0} \ar[d] \ar[r] & (\Z/p^n)^\Syn _{p=0} \ar[dd] \\
      (\Z/p^{n+1})^\Syn _{p=0} \ar[d] & \\
      (\Z_p)^{\Syn} _{p=0} \ar[r, "F"] & (\Z_p)^\Syn _{p=0}.
    \end{tikzcd}
  \end{center}
\end{prop}

%Before we prove \Cref{prop:nyg-contract}, we record the following observations:

%\begin{lem}
%  Let $M$ denote an admissible $W_S$-module. Then $M$ is a (relatively) affine commutative flat group scheme over $S$.
%\end{lem}

%\begin{proof}
%  The affineness and flatness are what need to be checked.
%  These may be checked fpqc locally on $S$, so that it is true when $M$ is an invertible $W$-module and we may assume $M$ fits into a pullback square:
%  \begin{center}
%    \begin{tikzcd}
%      M \ar[r] \ar[d] & F_* I \ar[d] \\
%      W \ar[r, "F"] & F_* W 
%    \end{tikzcd}
%  \end{center}
%  Flatness and affineness of $M$ now follow from flatness and affiness of $W \to S$ and $F$.
%\end{proof}

%\begin{cor}
%  If $S$ lies over $\F_p$, then there is a natural Verschiebung map
%  $V_S : \Fr_S ^* M \to M$.
%\end{cor}

\begin{proof}%[Proof of \Cref{prop:nyg-contract}]
	We start by constructing the dashed morphism.
  For a test $\mathbb{F}_p$-algebra $R$, an object of $F^{-1}(\Z/p^n)(R)$ is a tuple $(M,d:M \to W_R,x \in \Fr_R ^*M(R))$ such that $(\Fr_R ^*d)(x)=p^n$.  
	
  Now, using \Cref{cor:V-exists}, we have a commutative square
  \begin{center}
    \begin{tikzcd}
      M(R) \ar[d,"d"] & \Fr_R ^*M(R) \ar[l, "V"] \ar[d,"\Fr_R ^*d"] \\
      W(R) & W(R) \ar[l, "V"].
    \end{tikzcd}
  \end{center}
  In particular, $V(x) \in M(R)$ defines a lift of $V(p^n) \in W(R)$. By the assumption $p=0$ on $R$, $V(p^n) = p^{n+1} \in W(R)$, and we may take $V(x)$ as the desired lift of $p^{n+1} \in W(R)$.

  What remains is to check that the pullbacks under $j_{HT}$ and $j_{dR}$ agree.
  %Note that, since this is a diagram of stacks, we need not only to check that the morphisms agree after pullback, but also the $2$-morphism witnessing the commutativity of the diagram.

  Pulling back under $j_{HT}$, it is immediate that this agrees with the diagram constructed in \Cref{lem:prism-contract}.
  It therefore remains to check that we recover the same diagram if we pull back using $j_{dR}$.

%  The proof is exactly the same, using the Verschiebung map for admissible $W$-modules.  Andy suggests to see French version Gabriel and Demazure (Section 4.3.4).

  Starting with a Cartier--Witt divisor $d : I \to W$, its image under $j_{dR}$ is given by the pullback
\[
\begin{tikzcd}
  M \arrow{r} \arrow{d} & F_*I \arrow{d}{F_* d} \\
W \arrow{r}{F} & F_*W.
\end{tikzcd}
\]
  To understand how $j_{dR} : R^{\prism} \to R^{\Nyg}$ works for general $R$, we need to recall the isomorphism $j_{dR} ^* \G^{\Nyg} \cong \G^{\prism} \cong j_{HT} ^* \G^{\Nyg}$,
  which is induced via the ring isomorphism $W/M \cong F_* W / F_* I$ coming from the above pullback square.

  For $R = \Z/p^n$, this may be understood as follows: using the above pullback square, lifts of $p^n$ from $F_* W$ to $F_* I$ correspond to lifts of $p^n$ from $W$ to $M$.
  What we need to check is that the first part of this proof, which takes $x \in F^* M(R)$, a lift of $p^n$ to $V(x) \in M(R)$, a lift of $p^{n+1}$, is compatible with this correspondence.

  This holds by the naturality of $V$, which implies that we have a commutative cube:
\[\begin{tikzcd}[row sep={40,between origins}, column sep={40,between origins}]
      & \Fr_R ^*M \ar{rr}\ar{dd}\ar{dl} & & F_*\Fr_R ^*I \ar{dd}\ar{dl} \\
    \Fr_R ^*W \ar[crossing over]{rr} \ar{dd} & & F_*\Fr_R ^*W \\
      & M  \ar{rr} \ar{dl} & &  F_*I, \ar{dl} \\
    W \ar{rr} && F_*W \ar[from=uu,crossing over]
\end{tikzcd}\]
where the vertical maps are given by $V$.
%
%  Pulling back along $F$, we obtain a pullback square:
%\[
%\begin{tikzcd}
%F^*M \arrow{r} \arrow{d} & F_*F^*I \arrow{d} \\
%F^*W \arrow{r} & F_*F^* W,
%\end{tikzcd}
%\]
%under which 
%
%To understand how $j_{dR}$ works for $R^{\prism} \to R^{\Nyg}$, when $R$ is more general than $\mathbb{Z}_p$:  Given the extra data of a map $R \to W/M \cong F_*W / F_*I=W/I$, where the last equivalence is one of rings but not of $W$-algebras.
%
%
%There is a pullback square
%\[
%\begin{tikzcd}
%F^*M \arrow{r} \arrow{d} & F_*F^*I \arrow{d} \\
%F^*W \arrow{r} & F_*F^* W
%\end{tikzcd}
%\]
%which implies that a lift of $p^n$ from $F^*W$ to $F^*M$ is the same data as a lift of $p^n$ from $F_*F^*W$ to $F_*F^*I$.
%
%Now, by naturality of $V$, we have a commuting cube
%
%\[\begin{tikzcd}[row sep={40,between origins}, column sep={40,between origins}]
%      & F^*M \ar{rr}\ar{dd}\ar{dl} & & F_*F^*I \ar{dd}\ar{dl} \\
%    F^*W \ar[crossing over]{rr} \ar{dd} & & F_*F^*W \\
%      & M  \ar{rr} \ar{dl} & &  F_*I \ar{dl} \\
%    W \ar{rr} && F_*W \ar[from=uu,crossing over]
%\end{tikzcd}\]
%where the vertical map are $V$.

%  To observe the compatibility with $j_{HT}$ and $j_{dR}$, it suffices to look at the examples of $(I \to W, x)$ and a pullback
%  \begin{center}
%    \begin{tikzcd}
%      M \ar[d] \ar[r] & F_* I \ar[d] \\
%      W \ar[r,"F"] & F_* W 
%    \end{tikzcd}
%  \end{center}
%  comma $(p^n, x)$, and just see the compatibility.
\end{proof}

\subsection{Crystallinity}

\begin{thm}\label{thm:univexample}
  For integers $n \geq 0$, there exist compatible commutative diagrams
  \begin{center}
    \begin{tikzcd}
      (\Zp)^{\syn}_{p = v_1^{p^n} = 0} \ar[rr] \ar[rd] & & (\Zp)_{p=0}^\syn \\
       & (\Z/p^{n+2})_{p=0}^\syn,  \ar[ru] &
    \end{tikzcd}
  \end{center}
	  \begin{center}
    \begin{tikzcd}
      (\Zp)^{\Nyg}_{p = v_1^{p^n} = 0} \ar[rr] \ar[rd] & & (\Zp)_{p=0}^\Nyg \\
       & (\Z/p^{n+2})_{p=0}^\Nyg,  \ar[ru] &
    \end{tikzcd}
  \end{center}
	and
	  \begin{center}
    \begin{tikzcd}
      (\Zp)^{\prism}_{p = v_1^{p^n} = 0} \ar[rr] \ar[rd] & & (\Zp)_{p=0}^\prism \\
       & (\Z/p^{n+2})_{p=0}^\prism.  \ar[ru] &
    \end{tikzcd}
  \end{center}
\end{thm}

\begin{proof}
We present the proof for syntomification, noting that the arguments for prizmatization and for Nygaardification are perfectly analogous. Our goal is to induct on $n \ge 0$ to construct a dashed arrow
	  \begin{center}
    \begin{tikzcd}
      F^{-n} \left( (\Zp)^{\syn}_{p=v_1=0} \right)\ar[rr] \ar[rd, dashed] & & (\Zp)_{p=0}^\syn \\
       & (\Z/p^{n+2})_{p=0}^\syn  \ar[ru]. &
    \end{tikzcd}
  \end{center}
	To begin, we first construct the dashed arrow when $n=0$.  By the last sentence of \Cref{prop:syn-base}, there is a diagram
	\[
	\begin{tikzcd}
	(\Zp)^{\syn}_{p=v_1=0} \ar[d] \ar[r] & (\Zp)^{\syn}_{p=0} \\
	F^{-1}((\mathbb{F}_p)^{\syn}_{p=0}) \ar[ur]
	\end{tikzcd}
	\]
	The case $n=0$ thus follows from the last sentence of \Cref{prop:nyg-contract}.
		
	Now we proceed with the inductive step.  In particular, suppose for the sake of induction that we have constructed a diagram 
	
		  \begin{center}
    \begin{tikzcd}
      F^{-n+1} \left( (\Zp)^{\syn}_{p=v_1=0} \right)\ar[rr] \ar[rd, dashed] & & (\Zp)_{p=0}^\syn \\
       & (\Z/p^{n+1})_{p=0}^\syn  \ar[ru]. &
    \end{tikzcd}
  \end{center}
	
	Applying $F^{-1},$ we obtain a diagram
			  \begin{center}
    \begin{tikzcd}
      F^{-n} \left( (\Zp)^{\syn}_{p=v_1=0} \right)\ar[rr] \ar[rd, dashed] & & (\Zp)_{p=0}^\syn \\
       & F^{-1} (\Z/p^{n+1})_{p=0}^\syn \ar[ru], &
    \end{tikzcd}
  \end{center}
	The desired diagram is then constructed by applying the last sentence of \Cref{prop:nyg-contract}.
\end{proof}

For prismatization and Nygaardification individually, but not syntomification, we can improve the above result:

\begin{thm} \label{thm:a-improvement}
 For each $n \geq 0$, there exists a commutative diagram
  \begin{center}
    \begin{tikzcd}
      (\Zp)^{\Nyg}_{p = (a^p \mu)^{p^n} = 0} \ar[rr] \ar[rd] & & (\Zp)_{p=0}^\Nyg \\
       & (\Z/p^{n+2})_{p=0}^\Nyg  \ar[ru] &
    \end{tikzcd}
  \end{center}
	which restricts, under $j_{HT}$, to a commutative diagram
  \begin{center}
    \begin{tikzcd}
      (\Zp^{\prism})_{p = v_1^{p^{n+1}} = 0} \ar[rr] \ar[rd] & & (\Zp)_{p=0}^\prism \\
       & (\Z/p^{n+2})_{p=0}^\prism  \ar[ru] &
    \end{tikzcd}
  \end{center}
\end{thm}

\begin{proof}
	We run an argument similar to that of \Cref{thm:univexample}. Our goal is to induct on $n \ge 0$ to construct a dashed arrow
	\begin{center}
		\begin{tikzcd}
			F^{-n} \left( (\Zp)^{\Nyg}_{p=(a^p\mu)=0} \right)\ar[rr] \ar[rd, dashed] & & (\Zp)_{p=0}^\Nyg \\
			& (\Z/p^{n+2})_{p=0}^\Nyg  \ar[ru]. &
		\end{tikzcd}
	\end{center}
	%compatible with restriction along $j_{HT}$, since, the locus $a^p\mu=0$ becomes $v_1^p=0$ along this restriction.
	To begin, we first construct the dashed arrow when $n=0$.  By \Cref{prop:restoprism}, there is a diagram
	\[
	\begin{tikzcd}
		(\Zp)^{\Nyg}_{p=a^p\mu=0} \ar[d] \ar[r] & (\Zp)^{\Nyg}_{p=0} \\
		F^{-1}((\mathbb{F}_p)^{\Nyg}_{p=0}) \ar[ur]
	\end{tikzcd}
	\]
	compatible with restriction to the prismatization.
	Combining this with \Cref{prop:nyg-contract} gives the case $n=0$.
	
	Now we proceed with the inductive step.  In particular, suppose for the sake of induction that we have constructed a diagram 
	
	\begin{center}
		\begin{tikzcd}
			F^{-n+1} \left( (\Zp)^{\Nyg}_{p=a^p\mu=0} \right)\ar[rr] \ar[rd, dashed] & & (\Zp)_{p=0}^\Nyg \\
			& (\Z/p^{n+1})_{p=0}^\Nyg  \ar[ru]. &
		\end{tikzcd}
	\end{center}
	
	Applying $F^{-1},$ we obtain a diagram
	\begin{center}
		\begin{tikzcd}
			F^{-n} \left( (\Zp)^{\Nyg}_{p=a^p\mu=0} \right)\ar[rr] \ar[rd, dashed] & & (\Zp)_{p=0}^\Nyg \\
			& F^{-1} (\Z/p^{n+1})_{p=0}^\Nyg \ar[ru], &
		\end{tikzcd}
	\end{center}
	The desired diagram is then constructed by applying \Cref{prop:nyg-contract}.
\end{proof}

%The following corollaries may be seen as two of the main results of the paper:
%
%\begin{cor}\label{thm:cryssyntomic}
%The functor taking an animated ring $R$ to the mod $(p,v_1^{p^n})$ derived syntomic cohomology of $R$ factors through the functor $R \mapsto R/p^n$.
%\end{cor}

\begin{proof}[Proof of \Cref{thm:stackmodp}(1)-(3)]
%  We claim that it is enough to prove \Cref{thm:stackmodp}.
%  Indeed, if we assume \Cref{thm:stackmodp}, then it follows from \Cref{thm:prism-comparison} and \Cref{thm:synandnyg-comparison} that \Cref{thm:mainmodp} holds when restricted to the category of $p$-complete rings $R$ such that both $R$ and $R/p^{n+2}$ are $p$-quasisyntomic.
%  In particular, it holds for the category of $p$-completed polynomial $\Zp$-algebras.
%  %that 
%  %, by taking global sections of the Breuil-Kisin line bundles, we obtain \Cref{thm:main} for $p$-completed polynomial $\Zp$-algebras over by \Cref{thm:prism-comparison} and \Cref{thm:synandnyg-comparison}, because $p$-completed polynomial $\Zp$-algebras are quasi-syntomic.
%%
%%
%  %, and global sections of $(-)^{\prism}, (-)^{\Nyg}, (-)^{\Syn}$ recover prismatic cohomology, Nygaard-filtered prismatic cohomology, and syntomic cohomology respectively by \Cref{thm:prism-comparison} and \Cref{thm:synandnyg-comparison}).  Since $p$-adic completions of polynomial $\mathbb{Z}_p$-algebras are quasi-syntomic, the rersult follows for the full subcategory consisting of these . Since these functors commute with sifted colimits and are left-Kan extended from the case of polynomial algebras over $\Zp$, the result folllows.
%  Since all of the functors under consideration are left Kan extended from this subcategory, we obtain the full statement of \Cref{thm:main}.

  First we prove \Cref{thm:stackmodp}(2).
  The proof of \Cref{thm:stackmodp}(3) is exactly analogous, except that we substitute \Cref{thm:a-improvement} for \Cref{thm:univexample}.

Using \Cref{thm:univexample}, we may produce a diagram of stacks:
  \begin{center}
    \begin{tikzcd}
      (X^\Nyg)_{p=v_1 ^{p^n} = 0} \ar[d] \ar[r] & (\Zp^\Nyg)_{p=v_1 ^{p^n} = 0} \ar[d] \\
      (X_{p^{n+2=0}}) ^\Nyg \ar[d] \ar[r] & (\Z/p^{n+2})^\Nyg \ar[d] \\
      (X^\Nyg) \ar[r] & (\Zp^\Nyg).
    \end{tikzcd}
  \end{center}
  
  The outer square is a pullback by definition, and the bottom square is a pullback because $(-)^{\Nyg}$ preserves limits of derived p-adic formal schemes.
  It follows that the top square is also a pullback, so that the formula $(X_{p^{n+2}=0}) ^\Nyg\times_{(\ZZ/p^{n+2})^{\Nyg}}(\ZZ_p^{\Nyg})_{p=v_1^{p^n}=0}$ gives a functorial description of $(X^{\Nyg})_{p=v_1^{p^n}=0}$ only depending on $X_{p^{n+2}=0}$.
  
  Finally, to prove \Cref{thm:stackmodp}(1), we note that by \Cref{thm:univexample}, the functorial descriptions of $(X^{\Nyg})_{p=v_1^{p^n}=0}$ and $(X^{\prism})_{p=v_1^{p^n}=0}$ in terms of $X_{p^{n+2}=0}$ are compatible with restriction along $j_{HT}$ and $j_{dR}$. We may thus glue these loci together to obtain $(X^{\Syn})_{p=v_1^{p^n}=0}$ functorially.
\end{proof}

%\begin{cor}
%The functor taking an animated ring $R$ to the mod $(p,v_1^{p^{n+1}})$ derived prismatic cohomology of $R$ factors through the functor $R \mapsto R/p^n$.
%\end{cor}
%
%\begin{rmk}
%The functor taking an animated ring $R$ to the mod $(p,v_1^{p^{n+1}})$ derived syntomic cohomology of $R$ \emph{does not} factor through the functor $R \mapsto R/p^n$.
%\end{rmk}

\section{Prismatization and syntomification modulo larger powers $p$} \label{sec:rapid}
\label{section:largerpowersofp}
%We will say that a tower is \emph{rapidly converging} if there exist it is rapidly converging for some unspecified rate $a$ and constant $N$.%  Similarly, we may say that a tower is \emph{rapidly converging at rate $a$} while leaving the constant unspecified.

In the previous sections, we proved crystallinity of prismatization and syntomification after reduction modulo $(p,v_1^{p^n})$. In the second part of this section, by induction on $c$, we obtain crystallinity results modulo $(p^c,v_1^{p^n})$ when $n\geq c-1$. This will be an application of some general results about rapid convergence of towers that we prove in the first part of this section.

\subsection{Rapid convergence of towers}

A tower $X$ of objects in a category $\mathcal{C}$ is a diagram
\[\{\cdots \to X_{i+1} \to X_{i} \to X_{i-1} \to \cdots \to X_0,\}\]
indexed by integers $i \ge 0$, i.e an object in $\cC^{\NN_{\geq}}$, where $\NN_{\geq}$ is the poset of natural numbers with a map $j \to i$ if $j\geq i$.

Given a tower $X$, and $n \in \NN$, we use $X[n]$ to denote the shifted tower given by precomposition with $+n:\NN_{\geq} \to \NN_{\geq}$, and $\tau^n$ to denote the natural transformation $X[n+k] \to X[k]$ as well as the maps $X_{k+n} \to X_k$. For an object $W \in \cC$ we use $W_{\const}$ to denote the constant tower.
\begin{dfn}\label{dfn:rapidconvergence}
	For $r \geq 0, N\geq0$, we say that a tower $X$ is \textit{rapidly converging} at rate $r$ and constant $N$ to $W$ if there are maps of towers $X[N+r] \to W_{\const} \to X[N]$ such that the composite as well as the composite $W_{\const}[r] \to X[N+r] \to W_{\const}$ are homotopic to $\tau^r$.
\end{dfn}

\begin{rmk}
It follows from the definitions that a rapidly converging tower $X$ is pro-constant at $W$. In particular, its limit exists and is equal to $W$.
\end{rmk}

In this section, we study stability properties of rapid convergence. We first note that the assignment $n \mapsto \Map(X[n],Y)$ makes $\cC^{\NN_{\geq}}$ naturally enriched over $\Spc^{\NN_{\leq}}$, i.e there is an increasingly filtered mapping space, which we denote $\Map(X,Y)_*, * \in \NN_{\leq}$. One way to see this is as follows: without loss of generality, $\cC$ can be assumed to be small, since $\Cat$ is presentable. In this case, it suffices to show that the category of presheaves $\Psh(\cC^{\NN_{\geq}})$ is naturally enriched, since then $\cC^{\NN_{\geq}}$ is an enriched subcategory. But $\NN_{\geq}$ is symmetric monoidal, and so it acts on $\cC^{\NN_{\geq}}$. Thus $\Psh(\cC^{\NN_{\geq}})$ is acted on by $\Psh(\NN_{\geq}) = \Spc^{\NN_{\leq}}$, which by \cite[Theorem 1.2]{heine2023equivalence} makes it $\Spc^{\NN_{\leq}}$-enriched. %More concretely, the enrichment associates to objects $x,y \in \cC^{\NN_{\geq}}$ the increasingly filtered space $\Map(x,y)_\bullet, \bullet \in \NN_{\leq}$, where $\Map(x,y)_n=\Map(x[n],y)$, where $y[n]$ is obtained from $y$ by precomposition with $+n: \NN_{\geq} \to \NN_{\geq}$.

There is a symmetric monoidal functor $\Spc^{\NN_{\leq}} \to \Spc$ given by taking the colimit, and so we can use this to give $\cC^{\NN_{\geq}}$ a \textit{different} category structure, which we denote $\cC^{\NN_{\geq}}_{\ev}$. The category $\cC^{\NN_{\geq}}_{\ev}$ can also be described as the Dwyer--Kan localization of $\cC^{\NN_{\geq}}$ at the maps $\tau:X[1] \to X$. There is a functor $\cC^{\NN_{\geq}} \to \cC^{\NN_{\geq}}_{\ev}$, which is enriched in the sense that it maps the colimit of the filtered mapping space of the source (isomorphically) onto the mapping space of the target. We say a map in $\cC^{\NN_{\geq}}$ is a \textit{shifted equivalence} if it is an equivalence in $\cC^{\NN_{\geq}}_{\ev}$.

From this point of view, we can reinterpret rapid convergence at rate $r$ and constant $N$ as saying that $X[N]$ is shifted equivalent to a constant tower $W$, and the data of the equivalence (i.e the maps, composition, and homotopies) can be lifted to filtration $\leq r$.

The notion of rapid convergence is almost equivalent to the following weaker notion.

\begin{dfn}\label{dfn:weakrapidconv}
	We say that a tower $X$ in $\cC$ is \textit{weakly rapidly converging} at rate $r$ and constant $N$ if it is rapidly converging at rate $r$ and constant $N$ as a tower in the homotopy category $h\cC$.
	
	Equivalently, this means there is an object $W$, and maps $f:X_{N+r}\to W$, $g:W_{\const} \to X[N]$ such that $f \circ g_r$ is homotopic to the identity, and the composite $X_{i+r} \xrightarrow{f\tau^{i-N}} W \xrightarrow{g_{i-N}} X_{i}$ is homotopic to $\tau^r$ for each $i \geq N$.
\end{dfn}

It is clear that any rapidly converging tower is weakly rapidly converging at the same rate and constant. The converse is true after possibly changing the rate and constant.
\begin{prop}\label{prop:weakimpliesstrongrapidconv}
	Let $X_\bullet \in \cC^{\NN_{\geq}}$. Then if $X$ is weakly rapidly converging at rate $r\geq1$ and constant $N$, it is rapidly converging at rate $2r-1$ and constant $N'=N+r-1$.
\end{prop}

\begin{proof}
	
	%After replacing $X$ with $X[N]$, we may assume that $N=0$. 
	Note that the data of a weakly rapidly converging tower gives us the diagram below:\footnote{In case $r=1$, the parts of the diagram labelled $X_{N'}$ and $X_{N'+r-1}$ are identified.}
	
	\[\begin{tikzcd}[column sep=10]
		\cdots & {X_{N'+3r-1}} & \cdots & {X_{N'+2r-1}} & \cdots & {X_{N'+r-1}} & \cdots & {X_{N'}} \\
		\cdots & W && W && W
		\arrow[from=1-1, to=1-2]
		\arrow[from=1-2, to=1-3]
		\arrow[from=1-2, to=2-4]
		\arrow[from=1-3, to=1-4]
		\arrow[from=1-4, to=1-5]
		\arrow[from=1-4, to=2-6]
		\arrow[from=1-5, to=1-6]
		\arrow[from=1-6, to=1-7]
		\arrow[from=1-7, to=1-8]
		\arrow[from=2-1, to=2-2]
		\arrow["g", from=2-2, to=1-2]
		\arrow["\simeq"', from=2-2, to=2-4]
		\arrow["g", from=2-4, to=1-4]
		\arrow["\simeq"', from=2-4, to=2-6]
		\arrow[from=2-6, to=1-6]
	\end{tikzcd}\]
	
	Let $J$ be the shape of the diagram above where we don't include the inverses of the maps between the copies of $W$. Then $J$ is a poset. There are towers $x:\NN_{\geq} \to J$ sending $i$ to the spot of $X_{N'+i}$ and $y:\NN_{\geq} \to J$ sending $i$ to the spot of the copy of $W$ below $X_{(i-i\%r)+N'+r-1}$. Then there are maps $y \to x$ and $x[2r-1] \to y$ such that both composites are $\tau^{2r-1}$. Post-composing this with the functor $J \to \cC$ shows the desired result since $x$ is sent to $X[N']$ and $y$ is sent to $W_{\const}$.
\end{proof}

The following lemmas are immediate from the definition:

\begin{lem} \label{lem:rapconvfunctor}
	Let $X$ be a rapidly converging tower in $\cC$ and let $F:\cC \to \cD$ be a functor. Then $F(X)$ is rapidly converging in $\cD$ with the same rate and constant.
\end{lem}

\begin{lem}\label{lemma:rapidconvergeproduct}
	Let $\cC_i, i \in I$ be  a family of categories over some index set $I$. Then if $X_{n,i}$ are a family of (weakly) rapidly converging towers in $\cC_i$ with rate $r$ and constant $N$, then so is $(X_{n,i})$ in $\prod_{i}\cC_i$.
\end{lem}

The following lemma shows that rapid convergence glues well.

\begin{lem}\label{lem:rapconvorientedpullback}
	Let $\cC_1 \xrightarrow{F} \cD \xleftarrow{G} \cC_2$ be a cospan of $\infty$-categories. Let $(X,Y, F(X) \xrightarrow{f} G(Y))$ be a tower in the oriented pullback $\cC_1\vec{\times}_{\cD}\cC_2$. Suppose $X$ and $Y$ are rapidly converging at rates $r,r'$ and constants $N,N'$ respectively with $N'+r'\geq N$. Then $(X,Y,f)$ is rapidly converging with rate $r+r'$ and constant $N'$.
\end{lem}

\begin{proof}
	Let $N'' = N'+r'$. By considering the diagram
	
	\[\begin{tikzcd}
		{F(X)[N''+r]} & {F(W)_{\const}} & {F(X)[N'']} \\
		&& {G(Y[N''])} & {G(W'_{\const})} & {G(Y[N'])}
		\arrow[from=1-1, to=1-2]
		\arrow["{\tau^r}"{description}, bend left = 30, from=1-1, to=1-3]
		\arrow[from=1-2, to=1-3]
		\arrow[from=1-3, to=2-3]
		\arrow[from=2-3, to=2-4]
		\arrow["{\tau^{r'}}"{description}, bend right = 30, from=2-3, to=2-5]
		\arrow[from=2-4, to=2-5]
	\end{tikzcd}\]
	we see that in $(\cC_1\vec{\times}_{\cD}\cC_2)^{\NN_{\geq}}$, the map $$(X[N''+r],Y[N'+r'],\tau^{r}f) \to (X[N''],Y[N'],\tau^{r'}f)$$ given by $(\tau^r,\tau^{r'},\tau^r\circ \tau^{r'}\simeq \tau^{r'}\circ \tau^r)$ factors through the object $(W_{\const},W'_{\const},(\lim f)_{\const})$. After pre and post-composing this with appropriate powers of $\tau$ on the source and target, we learn that $\tau^{r+r'}$ on $(X[N''+r],Y[N''+r],f)$ factors through a constant tower. To see the other composite is the identity, this can be checked in $(\cC_1\vec{\times}_{\cD}\cC_2)^{\NN_{\geq}}_{\ev}$ since the tower is constant. However the  factorization consists of inverse shifted equivalences, so this follows.
	%	After replacing $X,Y$ with $X[N]$ and $Y[N]$, it suffices to assume that $N=0$. Note we get a map $F(W) \to W'$ by taking the limit of the map $F(X) \to Y$. From the assumptions, we can produce a sequence of three squares below.
	%	\[\begin{tikzcd}
		%		{F(X[2r])} & F(W) & {F(X[r])} & F(W) \\
		%		{Y[2r]} & {W'} & {Y[r]} & {W'}
		%		\arrow[from=1-1, to=1-2]
		%		\arrow[from=1-1, to=2-1]
		%		\arrow[from=1-2, to=1-3]
		%		\arrow[from=1-2, to=2-2]
		%		\arrow[from=1-3, to=1-4]
		%		\arrow[from=1-3, to=2-3]
		%		\arrow[from=1-4, to=2-4]
		%		\arrow[from=2-1, to=2-2]
		%		\arrow[from=2-2, to=2-3]
		%		\arrow[from=2-3, to=2-4]
		%	\end{tikzcd}\]
	%	The middle square commutes by naturality of limits, and we claim the left square also commutes. To see this, since the composite of the middle and right square is the identity, it suffices to check that the composite of all three squares commutes. But this follows from a diagram chase, using that the composite of the first two and last two squares as well as the middle square commute. Thus we have produced the desired data witnessing rapid convergence in the lax limit category.
\end{proof}

\begin{lem}\label{lem:rapidconvpullback}
	Let $X\rightarrow Y \leftarrow Z$ be a cospan of towers in $\cC$ such that $X,Z$ rapidly converge at rate $r$ with constant $N$ and $Y$ converges at rate $r'$ with constant $N'$ with $N'+r'\geq N$. Then the pullback $X\times_YZ$ rapidly converges at rate $r+r'$ and constant $N'$.
\end{lem}

\begin{proof}
  We note that the category $(\cC\times \cC) \vec{\times}_{\cC\times \cC}\cC$ is the category of cospan diagrams in $\cC$. Thus the result follows by applying \Cref{lem:rapconvorientedpullback} and \Cref{lem:rapconvfunctor}.
\end{proof}

\subsection{Larger powers of p}

\begin{lem}\label{lem:rapidconvergencefromcrystallinity}
	The tower of animated rings $\ZZ/p^{\bullet}\otimes \ZZ/p^r$ is rapidly converging at rate $r$ and constant $0$ to $\ZZ/p^r$.
\end{lem}

\begin{proof}
	The tower we are considering can be written out as the pushout of the span of towers, where the left and right side are constant, the left map sends $y$ to $0$, and the right map sends $y$ to $1$.
	
	\begin{center}
		\begin{tikzcd}
			\ZZ/p^{r} & \ZZ/p^{r}[p^\bullet y]\ar[l]\ar[r] &\ZZ/p^{r}
		\end{tikzcd}
	\end{center}
	
	The span is clearly rapidly converging at rate $r$ and constant $0$ to the constant span of $\ZZ/p^{r}$, so applying the colimit functor we are done.
\end{proof}

\begin{lem} \label{lem:v1lifting}
  The class $v_1 ^{p^n} \in H^0(\F_p (p^n(p-1)) (\Z_p))$ lifts to a class in $H^0 (\Z/p^{n+1} (p^n(p-1)) (\Z_p)$.

  Moreover, the class $(a^p \mu)^{p^n} \in H^0 ((\Zp)^{\Nyg}_{p=0}; \O\{p^n(p-1)\} \otimes a^* \O(p^n(p-1)))$ lifts to $H^0 ((\Zp)^{\Nyg}_{p^n=0}; \O\{p^n(p-1)\} \otimes a^* \O(p^n(p-1)))$.
\end{lem}

\begin{proof}
  This follows from the Bockstein spectral sequence, since the differentials are derivations and the $E_1$-page consists of $\F_p$-vector spaces.
\end{proof}

\begin{thm}\label{thm:stack}
	  Fix $c, k,n\geq0$. Then we have that the functors
	\begin{enumerate}
		\item \[X \mapsto (X^\syn)_{p^c=v_1^{p^n}=0} \]
		\item \[X \mapsto (X^\Nyg)_{p^c=v_1^{p^n}=0} \]
		\item \[X \mapsto (X^\prism)_{p^c=v_1^{p^{n+1}}=0}\]
		\item \[X \mapsto (X^\prism)_{p^{ck}=(F^{n+1-k})^* I=0}\]
		\item \[X \mapsto (X^\Nyg)_{p^{c}=(a^p\mu)^{p^n}=0}\]
	\end{enumerate}
  factor through the functor $X \mapsto X_{p^{c(n+2)}=0}$, where $X_{p^{c(n+2)}=0}$ is regarded as derived $\Z/p^{c(n+2)}$-scheme. In $(1),(2),(5)$ we require that $n\geq c-1$, and for $(3)$, we require $n+1\geq c-1$.
	Here, $X$ is a derived $p$-adic formal scheme, and all of the vanishing loci are taken in a derived sense.
\end{thm}

\begin{proof}
  We prove the result by induction on $c\geq1$,\footnote{The case $c=0$ is trivial.} with the base case being \Cref{thm:stackmodp}. We just indicate the proof for $(1)$ as the others are essentially identical. The hypothesis $n\geq c-1$ guarantees that $v_1^{p^n}$ exists modulo $p^c$ by \Cref{lem:v1lifting}. The inductive hypothesis gives that the functor
			\[X \mapsto (X^\syn)_{p^{i}=v_1^{p^n}=0} \]
      factors through $X \mapsto X_{p^{(n+2)i}=0}$ for $i<c$. Thus we learn that the tower (in the opposite category of stacks) $(\Spf(\ZZ/p^{\bullet})^\syn)_{p^{i}=v_1^{p^n}=0}$ is obtained by applying a functor to the tower $\ZZ/p^{\bullet}\otimes \ZZ/p^{i(n+2)}$, so that by applying \Cref{lem:rapidconvergencefromcrystallinity}, it is rapidly converging at rate $i(n+2)$ and constant $0$.
			
			Now the tower $$(\Spf(\ZZ/p^{\bullet})^\syn)_{p^c=v_1^{p^n}=0}$$ can be written as the pullback of the cospan of towers
      $$(\Spf(\ZZ/p^{\bullet})^\syn)_{p^{c-1}=v_1^{p^n}=0}\rightarrow(\Spf(\ZZ/p^{\bullet})^\syn)_{p=v_1^{p^n}=0, p^{c-1}=0}\leftarrow(\Spf(\ZZ/p^{\bullet})^\syn)_{p=v_1^{p^n}=0},$$
			and so by \Cref{lem:rapidconvpullback} we learn that it is rapidly converging with rate $c(n+2)$ and constant $0$ to $(\Spf(\ZZ_p)^\syn)_{p^c=v_1^{p^n}=0}$. In particular, we obtain a section $$\sigma:(\Spf(\ZZ_p)^\syn)_{p^c=v_1^{p^n}=0} \to (\Spf(\ZZ_p/(p^{c(n+2)}))^\syn)_{p^c=v_1^{p^n}=0}. $$
			
			Now we can construct the desired functor giving the factorization as the pullback of the cospan
			$$(\Spf(\ZZ_p)^\syn)_{p^c=v_1^{p^n}=0} \xrightarrow{\sigma} (\Spf(\ZZ_p/(p^{c(n+2)}))^\syn)_{p^c=v_1^{p^n}=0} \leftarrow (X^\syn)_{p^c=v_1^{p^n}=0}. \qedhere$$
\end{proof}

Exactly as in \Cref{thm:mainmodp}, we deduce crystallinity for syntomic and derived prismatic cohomology:

\begin{cor} \label{thm:main}
	Let $c,k,n\geq0$.
	The following functors on $\CAlg_p ^\Ani$ factor through the mod $p^{c(n+2)}$ reduction functor $\CAlg_p ^\Ani \to \CAlg_{\Z/p^{c(n+2)}} ^\Ani$:
	\begin{enumerate}
		\item mod $(p^c,v_1 ^{p^n})$ syntomic cohomology
		\[R \mapsto \ZZ/p^c (\ast) (R) /v_1 ^{p^n}\]
		\item mod $(p^c,v_1 ^{p^n})$ Nygaard-filtered derived prismatic cohomology
		\[R \mapsto N^{\geq \ast} \prism_{R} \{\ast\} /(p^c, v_1^{p^{n}})\]
		\item mod $(p^c,v_1 ^{p^{n+1}})$ derived prismatic cohomology
		\[R \mapsto \prism_{R} \{\ast\} /(p^c, v_1^{p^{n+1}})\]
		\item mod $(p^{kc}, (F^{n+1-k})^* I)$ derived prismatic cohomology, where $k \geq 1$, $F$ is the Frobenius, and $I$ is the Hodge-Tate ideal
		\[R \mapsto \prism_{R} \{\ast\} /(p^{kc}, (F^{n+1-k})^* I)\]
		\item mod $(p^{c}, (a^p\mu)^{p^n})$ Nygaard-filtered derived prismatic cohomology
		\[R \mapsto \N^{\geq \ast}\prism_{R} \{\ast\} /(p^{c}, (a^p\mu)^{p^n})\]
	\end{enumerate}
	
  In the cases $(1),(2),(5)$, we require that $n\geq c-1$, and for $(3)$ we require that $n+1\geq c-1$.
	%
	%  Given any $n \geq 0$, the mod $(p,v_1 ^{p^n})$ derived syntomic cohomology functor
	%  \[\F_p (\bullet) / v_1 ^{p^n} : \CAlg_p^\Ani \to \CAlg_p^{\Ani,\gr}\]
	%  factors through the functor which takes $R$ to its (derived) quotient $R/p^{n+2}$.
	%
	%  In other words, we construct a commutative diagram
	%  \begin{center}
		%    \begin{tikzcd}
			%      \CAlg_p^\Ani \ar[rr,"A \mapsto \F_p (\bullet) (A) / v_1 ^{p^n}"] \ar[dr, "A \mapsto A/p^{n+2}"] & & \CAlg_p ^{\Ani,\gr} \\
			%      & \Calg_{\Z/p^{n+2}} ^\Ani. \ar[ur] &
			%    \end{tikzcd}
		%  \end{center}
\end{cor}

\section{\'Etale cohomology of generic fibers}\label{sec:generic}
In this section, we explain how our results may be upgraded from mod $(p^c, v_1 ^{p^n})$ syntomic cohomology to mod $p^c$ syntomic cohomology if we make assumptions on the smoothness and dimension of our $p$-adic formal scheme.
By considering the \'etale realization of syntomic cohomology, we deduce conditions for the when the mod $p^c$ \'etale cohomology of the generic fiber of a $p$-adic formal schemes may be read off from its mod $p^d$ reduction.

\subsection{Recovering mod $p^c$ syntomic cohomology}
Our first result gives conditions on when the mod $p^c$ syntomic cohomology of an $F$-smooth $p$-adic formal scheme may be recovered from its mod $p^{n}$ reduction.
Via the \'etale comparison theorem, this also allows us to recover the mod $p^c$-\'etale cohomology of the generic fiber as a graded abelian group.
Before we state our theorem, we need a definition.

\begin{dfn}\label{dfn:affcohdim}
	Let $X$ be a $p$-adic formal scheme. Let $C_X$ be the category of Zariski sheaves of abelian groups $\mathfrak{F}$ such that $H^i(\mathfrak{F}|_{U})=0$ for $i>0$ on all affine opens $U \subset X$. Then we say the \textit{affine cohomological dimension} of $X$ is the smallest $d\in \NN\cup \{\infty\}$ such that $H^i(\mathfrak{F})=0$ for all $i>d, \mathfrak{F} \in C_X$.
\end{dfn}

\begin{exm}\label{exm:affcohdim}
	The affine cohomological dimension of $X$ is always at most the Zariski cohomological dimension, which is given by the Krull dimension of the underlying topological space\footnote{This is equal to the topological space associated to $X_{p=0}$.} by \cite[Theorem 4.5]{Scheid}. However, it can be significantly smaller--for example, it vanishes when $X$ is affine.
\end{exm}

We will also require some constants.

\begin{dfn}\label{dfn:functions}
	Given an integer $d$, we define:
  \begin{align*}
    &e_c(d) \coloneqq 1+\lceil \frac {d+2}{p^{c-1}(p-1)}\rceil. \\
    &e'_c(d) \coloneqq 2+\lceil \frac {2d+3}{p^{c-1}(p-1)}\rceil. \\
    &b_c(d) \coloneqq c(\lceil\log_p(e_c(d))\rceil + 2). \\
    &b'_c(d) \coloneqq c(\lceil\log_p(e'_c(d))\rceil + 2). 
  \end{align*}
\end{dfn}

\begin{thm} \label{thm:generic}
  Let $X$ denote a $p$-torsionfree $F$-smooth qcqs $p$-adic formal scheme of affine cohomological dimension $d$. %\footnote{Here, we mean the dimension of the underlying space of $X$. This is equal to the dimension of the scheme $X_{p=0}$.}
  Then the syntomic and \'etale mod $p^c$ cohomology groups\footnote{Here, $X_\eta$ is the adic generic fiber of $X$, i.e. $X \times_{\Spa(\Z_p, \Z_p)} \Spa(\Q_p, \Z_p)$ in the category of (pre-)adic spaces.}
  \[H^i_{\mathrm{syn}}(X;\ZZ/p^c(j)) \text{ and } H_{\et}^i(X_\eta; \ZZ/p^c(j))\]
  may be recovered functorially as graded $\ZZ/p^c[v_1^{p^{c-1}}]$-modules from $X_{p^{b_c(d)} = 0}$.

  More precisely, these functors factor through the essential image of the functor $X \mapsto X_{p^{b_c(d)}=0}$ from such $X$ into $\ZZ/p^{b_c(d)}$-schemes.
  Similarly, the cohomology rings can be recovered functorially from $X_{p^{b'_c(d)} = 0}$.
%  In particular, if $X_1$ and $X_2$ as above satisfy $(X_1)_{p^{b(d)} = 0} \cong (X_2)_{p^{b(d)} = 0}$, then $b_{i,j} (X_1;\F_p) = b_{i,j} (X_2;\F_p)$.
\end{thm}

\begin{rmk}
	In the case $X = \Spf A$ is affine and $(p,c) \neq (2,1)$, \Cref{thm:generic} state that the mod $p^c$ syntomic cohomology and the \'etale cohomology of the generic fiber only depend on $X_{p^{3c}=0}$. If also $(p,c) \neq (3,1),(2,2)$, then the ring structure can also be recovered.

	Analyzing the proof shows that the ranks of the mod $2$ \'etale cohomology groups can be recovered from $X_{2^3=0}$, but we do not know if the \'etale cohomology itself can be functorially recovered in this case.
\end{rmk}

\begin{rmk}
  For the technical defininition of $F$-smooth, we refer the reader to \cite{BM}.
  For example, $X$ could be a smooth $p$-adic formal scheme of relative dimension $d$ over $\Spf \O_K$ for a local or perfectoid field $K$.
\end{rmk}

Beyond crystallinity, the main input to the theorem of this section is a pair of vanishing lines for the reduced syntomic cohomology of $F$-smooth $p$-adic formal schemes.

\begin{lem} \label{lem:vanishing-lines}
  Let $X$ denote a $p$-torsion free $F$-smooth qcqs $p$-adic formal scheme of affine cohomological dimension $d$, and let $c \geq 1$. Then:
  \begin{enumerate}
    \item The complex $\Z/p^c (j) (X)$ is $(-j-d-1)$-connective.
    \item The complex $\Z/p^c (j) (X) / v_1 ^{p^{c-1}}$ is $(-j+p^{c-1}(p-1))$-coconnective.
        \item The graded complex $\Z/p^c (*) (X)\otimes_{\ZZ/p^c[v_1]} \Z/p^c (*) (X)$ is $(-j-2(d-1))$-connective in degree $j$.
    \item The graded complex $(\Z/p^c (*) (X)\otimes_{\ZZ/p^c[v_1]} \Z/p^c (*) (X))/ v_1 ^{p^{c-1}}$ is $(-j+2p^{c-1}(p-1))$-coconnective in degree $j$.
  \end{enumerate}
\end{lem}

\begin{proof}
  For $X = \Spf A$ affine, part (1) is an immediate consequence of \cite[Theorem 5.1(i)]{BeilFib}. Since syntomic cohomology is a Zariski sheaf, part (1) follows from the definition of affine cohomological dimension.

  On the other hand, we may reduce part (2) immediately to the case $X = \Spf (A)$ is affine.
  Then it is shown in \cite[Proof of Proposition 5.2]{BM} that the fiber of $\FF_p(j)(A) \xrightarrow{v_1} \FF_p(j+p-1)(A)$ is $(-j-1)$-coconnective.
  Using this, we see that $\fib \left(\Z/p^c(j)(A) \xrightarrow{v_1^{p^{c-1}}} \Z/p^c(j+p^{c-1}(p-1))(A) \right)$ is also $(-j-1)$-coconnective, from which we deduce that
  $\Z/p^c (A) / v_1 ^{p^{c-1}}$ is $(-j+p^{c-1}(p-1))$-coconnective, as desired.
  
  Part (3) is an immediate consequence of (1). To see (4), we first observe that it is true in the case $c=1$ by part (2), since $$(\FF_p (*) (X)\otimes_{\FF_p[v_1]}\FF_p (*) (X))/ v_1 \cong (\FF_p (j) (X) / v_1) \otimes_{\FF_p} (\FF_p (j) (X)/v_1)$$ and we are tensoring over a field. The general statement follows by realizing the complex as built from this via extensions. 
\end{proof}

\begin{qst}\label{rmk:tor-affine}
%	\Cref{lem:quotientrecoverymodv1} shows the mod $p$ syntomic cohomology of $A$ can only have simple $v_1$-torsion.
 If $X = \Spf A$ is affine, $p$-torsionfree and $F$-smooth, then \Cref{lem:vanishing-lines} implies the mod $p$ syntomic cohomology of $X$ has at most simple $v_1$-torsion.
	Is there such an $X$ such that the mod $p$ syntomic cohomology of $X$ has $v_1$-torsion?
\end{qst}

To apply this result to prove the theorem, we will use the following lemma in elementary algebra:

\begin{lem}\label{lem:recovermodout}
	Let $b\leq b'$ be integers.	Let $R$ be a discrete associative ring and let $R[x]$ be the graded polynomial algebra over $R$ with $x$ in grading $i$. Let $\cC\subset \Mod_{\Sp^{\ZZ}}(R[x])$ be the full subcategory of $x$-complete graded modules such that the cofiber $C/x$ lies in $[b-j,b'-j]$ in grading $j$.

  Then the functor taking an object $M \in \cC$ and recording $\pi_k(M)_j$ with its $x$-action for each $k$ and $j$ factors through $(-)/x^{e}:\cC \to \Mod_{\Sp^{\ZZ}}(R[x]) \to \Mod_{\Sp^{\ZZ}}(R[x]/x^e)$ if $e>\frac{b'-b}{i}$.
\end{lem}

\begin{proof}
	The long exact sequence on homotopy groups gives
	$$\pi_{k+1}(M/x)_{j}\to \pi_k(M)_{j-i}\xrightarrow{x} \pi_{k}(M)_{j} \to \pi_{k}(M/x)_{j}.$$
  So we learn that, when $k > b'-j$, multiplication by $x$ gives an isomorphism $\pi_k(M)_{j-i}\cong \pi_k(M)_j$.
  It therefore suffices to recover $\pi_k(M)_j$ for $k\leq b'-j$, along with its action of $x$.
	
	On the other hand, we also have the exact sequence
	$$\pi_k(M)_{j-ie}\xrightarrow{x^d} \pi_{k}(M)_{j} \to \pi_{k}(M/x^e)_j \to \pi_{k-1}(M)_{j-ie}$$
	whence we learn that $\pi_{k}(M)_{j} \to \pi_{k}(M/x^e)_j$ is an isomorphism if $k<b-j+ie$.
  Choosing $e$ such that $b'-j<b-j+ie$, i.e.\ such that $e > \frac{b'-b}{i}$, we obtain the desired result.
%	The assumptions imply that multiplication by $x$ is an isomorphism from $\pi_i(M[x^{-1}]_j)$ as soon as $j$
%	We have a short exact sequence
%	
%	$$ 0 \to \pi_i(M[x^{-1}])/x^d \to \pi_i(M_j)/x^d \to 0$$
\end{proof}

\begin{cor} \label{cor:recovery}
  Let $X$ denote a $p$-torsion free $F$-smooth qcqs $p$-adic formal scheme of affine cohomological dimension $d$, and let $c \geq 1$.
  Then the syntomic cohomology $H^i _{\mathrm{syn}} (X, \Z/p^c (j))$
  may be functorially recovered from the graded $\Z/p^c [v_1 ^{p^{c-1}}]/ v_1 ^{p^{c-1}e_c (d)}$-module $\Z/p^c (j) (X)/ v_1 ^{p^{c-1}e_c (d)}$
  as a graded $\Z/p^c [v_1 ^{p^{c-1}}]$-module.

  Morever, it may be functorially recovered from the graded $\Z/p^c [v_1 ^{p^{c-1}}]/ v_1 ^{p^{c-1}e'_c(d)}$-algebra $\Z/p^c (j) (X)/ v_1 ^{p^{c-1}e'_c(d)}$
  as a graded $\Z/p^c [v_1 ^{p^{c-1}}]$-algebra.
\end{cor}

\begin{proof}
  From \Cref{lem:vanishing-lines}, it follows that $\Z/p^c (j) (X)$ lies in $[-j-d-1, -j + p^{c-1} (p-1)]$.
  Taking $x$ to be $v_1 ^{p^{c-1}}$ in \Cref{lem:recovermodout}, we see that we obtain the first result if $e > \frac{p^{c-1} (p-1) + d+1}{p^{c-1}(p-1)} = 1 + \frac{d+1}{p^{c-1}(p-1)}$.
  The minimal such choice is $e = e_c (d)$, and yields the first statement of the corollary.

  For the second statement, let us use $R = \oplus_j \Z/p^c (j) (X)$ and $v = v_1 ^{p^{c-1}}$ to simplify notation. Then we have a multiplication map $(R/v^{e'_c(d)})\otimes_{\Z/p^c [v]/ v^{e'_c(d)}}(R/v^{e'_c(d)}) \to R/v^{e'_c(d)}$, which is the mod $v ^{e'_c(d)}$-reduction of the map $R\otimes_{\ZZ/p^c[v]} R \to R$. Moreover, both $(R\otimes_{\ZZ/p^c}R)/v$ and $R/v$ are both $(-j-2(d-1))$-connective and $(-j+2p^{c-1}(p-1))$-coconnective by \Cref{lem:vanishing-lines}, so we may apply \Cref{lem:recovermodout} to $(R/v^{e'_c(d)})\otimes_{\Z/p^c [v]/ v^{e'_c(d)}}(R/v^{e'_c(d)}) \to R/v^{e'_c(d)}$ to obtain the statement.
\end{proof}

\begin{proof}[Proof of \Cref{thm:generic}]
  Combining \Cref{cor:recovery} with \Cref{thm:mainintro}, we find that we may recover $H^i _{\mathrm{syn}} (X; \Z/p^c (j))$ as a graded $\Z/p^c [v_1 ^{p^{c-1}}]$-module from $X_{p^{b_c (d)} = 0}$.
  Similarly, we find that we may recover it as a graded $\Z/p^c [v_1 ^{p^{c-1}}]$-algebra from $X_{p^{b'_c (d)} = 0}$.

  To obtain the \'etale cohomology, we note that by the \'etale comparison theorem \cite[Theorem 8.3.1]{APC} \cite[Theorem 5.1]{BM}, the \'etale cohomology is obtained from the syntomic cohomology by inverting $v_1 ^{p^{c-1}}$.
\end{proof}

\subsection{Recovering the mod $p^c$ $F$-gauge}
The previous subsection explained how one can functorially recover \'etale cohomology groups.  Those results are strong in that they depend only on the affine cohomological dimension of $X$, and in particular hold for all affine schemes independent of dimension.  However, the results of the previous subsection are not sufficient to functorially recover \'etale cohomology \emph{complexes}.

In this subsection, we prove some more refined results in the case that the $p$-adic formal scheme $X$ is smooth over $\mathbb{Z}_p$ of relative dimension $d$.  Note now that $d$ refers to the dimension of $X$ over $\mathbb{Z}_p$, and not to its affine cohomological dimension.  We conjecture that similar results could be proved for $X$ smooth over more general $p$-adic rings of integers $\mathcal{O}_K$, but working over $\mathbb{Z}_p$ will allow us to quote useful theorems from \cite{Vologodsky}.

\begin{thm} \label{thm:FL}
  Let $\pi : X \to \Spf \Zp$ denote a smooth $p$-adic formal scheme of relative dimension $d$, and let $b_c ''(d) = \max\left(c+1, \lceil \log_p \left( \frac{d+1}{p-1} \right) + 2 \rceil \right)$.
  Then the associated mod $p^c$ $F$-gauge $R\pi ^\syn _* \mathcal{O}_{(X^\syn)_{p^c=0}} \in \mathcal{D} ((\Zp ^{\syn})_{p=0})$ may be functorially recovered from $X_{p^{cb''_c (d)} = 0}$.
\end{thm}

\begin{rmk}
  For example, this means that the mod $p$ $F$-gauge may be recovered from $X_{p^2=0}$ when $d \leq p-2$, and may be recovered from $X_{p^3 = 0}$ when $d \leq p^2-p-1$.
  The mod $p^2$ $F$-gauge may be recovered from $X_{p^6=0}$ when $d \leq p^2-p-1$.
\end{rmk}

Taking the \'etale realization, we obtain the following corollary:

\begin{cor}
  Given a smooth $p$-adic formal scheme $\pi : X \to \Spf \Zp$ of relative dimension $d$, the mod $p^c$ \'etale cohomology complex of the generic fiber
	\[R\Gamma_{\et} (X_{\eta, \mathbb{C}_p}, \mathbb{Z}/p^c)\]
  may be functorially recovered from $X_{p^{b''_c (d)} = 0}$ as a continuous $\Gal(\Q_p)$-representation.%  Functoriality in particular means that we recover the $\mathrm{Gal}(\mathbb{Q}_p)$ action on this complex. %the complex as a continuous representation of $\mathrm{Gal}(\mathbb{Q}_p)$.
\end{cor}

\begin{rmk}
  Examining the proof of \Cref{thm:FL}, we see that we have actually proven something stronger.
  Namely, given $d \geq 0$, there is a cohomology theory on smooth morphisms $\overline{X} \to \Spec (\Z/p^{cb''_c(d)})$ of relative dimension $\leq d$ taking values in $\mathcal{D} ((\Zp ^{\syn})_{p^c=0})$ that recovers the mod $p^c$ $F$-gauge of $X$ when $\overline{X} = X_{p^{cb''_c(d)}=0}$.

  In particular, taking the \'etale realization, we find that the ``mod $p^c$ \'etale cohomology of the generic fiber'' may be defined without the assumption that $\overline{X}$ lifts to a formal scheme over $\Zp$! %\todo{Just noting that there are indeed smooth projective surfaces over $\Z/p^n$ that cannot be lifted to $\Zp$. See Vakil's Murphy's Law paper. See also the MO question ``What is an example of a smooth variety over a finite field Fp which does not lift to Zp?''. I wonder what their etale cohomology is?}
\end{rmk}

To prove this theorem, we proceed in two steps.
First, we recall the notion of the \emph{Hodge--Tate weights} of an $F$-gauge, and note that in the context of \Cref{thm:FL}, the Hodge--Tate weights of $R\pi_* \mathcal{O}_{X^\syn}$ lie in the interval $[0,d]$.
Second, we prove that the categories of mod $p^c$ $F$-gauges on $\Zp$ with Hodge--Tate weights in the range $[0,i(p-1)-1]$ is unaffected by further modding out by $v_1 ^i$, extending an argument of Terentiuk, Vologodsky and Xu \cite{Vologodsky}.

\begin{dfn}[{\cite[Remark 5.3.14]{fgauge}}]
  Given a prismatic $F$-gauge $\mathcal{F} \in \mathcal{D} (\Z_p ^{\syn})$, we can consider its pullback $\mathcal{F}\vert_{B\G_m}$ along the composition
  \[B\G_m \hookrightarrow \A^1 / \G_m \xrightarrow{i_{dR}} \Z_p ^\Nyg \to \Z_p ^\syn\]
  Then $\mathcal{F}\vert_{B\G_m}$ may be regarded as a graded $\Z_p$-module $\bigoplus_i M_i$, and
  the \emph{Hodge--Tate weights} of $\mathcal{F}$ are the set $\{i \in \mathbb{Z} \vert M_i \neq 0\}$.\footnote{If $\mathcal{F} = R\pi_*\mathcal{O}_{X^\syn}$ for $X$ as in \Cref{thm:FL}, then it follows from \cite[Remark 5.3.14]{fgauge} that $\mathcal{F}\vert_{B\G_m}$ encodes the derived Hodge cohomology of $X$.}

  Let $\mathcal{D}_{[e,f]} (\Z_p ^\syn) \subset \mathcal{D} (\Z_p ^{\syn})$ denote the full subcategory of $F$-gauges with Hodge--Tate weights in the interval $[e,f]$.
\end{dfn}

\begin{exm}
  The $F$-gauge $\mathcal{O} \{i\}$ has Hodge--Tate weight $-i$.
\end{exm}

\begin{prop} \label{prop:where-weight}
  Let $\pi : X \to \Spf \Zp$ denote a smooth $p$-adic formal scheme of relative dimension $d$ over $\mathbb{Z}_p$. Then the Hodge--Tate weights of its $F$-gauge $R\pi ^\syn _* \mathcal{O}_{X^\syn}$ lie in $[0,d]$.
\end{prop}

\begin{proof}
  It suffices to show that the Hodge cohomology $H^*(X;\Omega_{X/\Z_p}^i)$ of $X$ vanishes for $i \not\in [0,d]$, which follows from the fact that $\Omega^i _{X / \Z_p}$ vanishes for $i$ outside of this range.
\end{proof}

Next we prove the following proposition, which is a mild generalization of \cite[(5.5)]{Vologodsky}.

\begin{prop} \label{prop:weight-equiv}
  Suppose $i$ is divisible by $p^{c-1}$, so that $v_1^i$ is well-defined modulo $p^c$.
  The pullback $\mathcal{D} _{[0,i(p-1)-1]} ((\Zp ^\syn)_{p^c = 0}) \to \mathcal{D} _{[0,i(p-1) -1]} ((\Zp ^\syn)_{p^c = v_1 ^i = 0})$ is an equivalence of categories.
\end{prop}

For the proof of this, we need the following result from \cite{Vologodsky}:

\begin{lem} \cite[Lemma 5.7]{Vologodsky} \label{lem:vol}
 Let $n \in \mathbb{Z}$. For any $E \in \mathcal{D}_{[n, \infty)} ((\Z_p ^\syn)_{p=v_1 = 0})$ and $E' \in \mathcal{D}_{(-\infty, n-1]} ((\Z_p ^\syn)_{p=v_1 = 0})$, we have $\RHom (E', E) = 0$.
\end{lem}

We will also use the following:
\begin{lem} (cf. \cite[(5.9) \& (5.10)]{Vologodsky}) \label{lem:vol2}
%\begin{enumerate}
%
%\item For each $j\ge 2$, the natural inclusion $\alpha_j:(\Z_p ^\syn)_{p=v_1 ^{j-1} = 0} \hookrightarrow (\Z_p ^\syn)_{p=v_1 ^j = 0}$ is a square-zero extension with ideal sheaf $\mathcal{O}_{(\Z_p ^\syn)_{p=v_1 = 0}} \{-(j-1)(p-1)\}$.
  Suppose $i$ is divisible by $p^{c-1}$, so that $v_1^i$ exists modulo $p^c$.  Then, for each $j\ge 2$, the natural inclusion $\alpha_j:(\Z_p ^\syn)_{p^c=v_1 ^{i(j-1)} = 0} \hookrightarrow (\Z_p ^\syn)_{p^c=v_1 ^{ij} = 0}$ is a square-zero extension with ideal sheaf $\mathcal{O}_{(\Z_p ^\syn)_{p^c=v_1^i= 0}} \{-(j-1)i(p-1)\}$.
%
%\item Suppose $i$ is such that $v_1^i$ exists modulo $p^c$, for some $c \ge 2$.  Then the natural inclusion $\beta_e:(\Z_p ^\syn)_{p^{c-1}=v_1 ^{i} = 0} \hookrightarrow (\Z_p ^\syn)_{p^c=v_1 ^i = 0}$ is a square-zero extension with ideal sheaf $\mathcal{O}_{(\Z_p ^\syn)_{p=v_1^i= 0}}$.
%\end{enumerate}
\end{lem}

\begin{cor}
Let $n \in \mathbb{Z}$, and suppose $v_1^i$ is well-defined modulo $p^c$. For any $E \in \mathcal{D}_{[n, \infty)} ((\Z_p ^\syn)_{p^c=v_1^i = 0})$ and $E' \in \mathcal{D}_{(-\infty, n-1]} ((\Z_p ^\syn)_{p^c=v_1^i= 0})$, we have $\RHom(E',E)=0$.
\end{cor}

\begin{proof}[Proof of \Cref{prop:weight-equiv}]
  Consider $\alpha_j : (\Zp ^\syn)_{p^c = v_1 ^{i(j-1)} = 0} \hookrightarrow (\Zp ^\syn)_{p^c = v_1^{ij} = 0}$.
  We prove that
  \[\alpha_j ^* : \mathcal{D} _{[0,i(p-1) -1]} ((\Zp ^\syn)_{p^c = v_1 ^{ij} = 0}) \to \mathcal{D} _{[0,i(p-1) -1]} ((\Zp ^\syn)_{p^c = v_1 ^{i(j-1)} = 0})\]
  is an equivalence for all $j > i$.  
  Taking the limit and using the fact that $(\Z_p ^\syn)_{p^c=0}$ is complete along the $v_1^i = 0$ locus, we obtain the desired statement.

  First, we prove full faithfulness. Let $\mathcal{F}, \mathcal{F}' \in \mathcal{D} _{[0,i(p-1) -1]} ((\Zp ^\syn)_{p^c = v_1 ^{ij} = 0})$.
  Then we want to show that
  \[\RHom (\mathcal{F}, \mathcal{F}') \to \RHom (\alpha_j^* \mathcal{F}, \alpha_j ^* \mathcal{F}') \simeq \RHom (\mathcal{F}, \alpha_{j*} \alpha_j ^* \mathcal{F}')\]
  is an equivalence.
  To start, we note that by \Cref{lem:vol2} the fiber of $\mathcal{F}' \to \alpha_{j*} \alpha_j ^* \mathcal{F}'$ may be identified with $\mathcal{F}' \otimes\mathcal{O}_{(\Z_p ^\syn)_{p^c=v_1^i = 0}} \{-(j-1)i(p-1)\}$, so it suffices to show that
  $\RHom (\mathcal{F}, \mathcal{F}' \otimes \mathcal{O}_{(\Z_p ^\syn)_{p^c=v_1^i = 0}} \{-(j-1)i(p-1)\}) = 0$.
  By $p$-adically filtering $\mathcal{O}_{(\Z_p ^\syn)_{p^c=v_1^i = 0}} \{-(j-1)i(p-1)\}$, we may reduce to the case that $c = 1$.
  Similarly, by $v_1$-adically filtering $\mathcal{O}_{(\Z_p ^\syn)_{p=v_1^i = 0}} \{-(j-1)i(p-1)\}$, we obtain an associated graded with pieces $\mathcal{O}_{(\Z_p ^\syn)_{p=v_1 = 0}} \{-(j-1)i(p-1)\}, \dots, \mathcal{O}_{(\Z_p ^\syn)_{p=v_1 = 0}} \{-(ij-1)(p-1)\}$.

  Now, $\mathcal{F}$ has Hodge--Tate weights $\leq i(p-1) -1$, while $\mathcal{F}' \otimes \mathcal{O}_{(\Z_p ^\syn)_{p=v_1 = 0}} \{-(j-1)i(p-1)\}), \dots, \mathcal{F}' \otimes \mathcal{O}_{(\Z_p ^\syn)_{p=v_1 = 0}} \{-(ij-1)(p-1)\})$ have Hodge--Tate weights $\geq (j-1)i(p-1) \geq i (p-1)$.
  It therefore follows from \Cref{lem:vol} that the mapping spectrum is $0$, as desired.

  Finally, we prove essential surjectivity. By deformation theory (see \cite[Lemma 5.8]{Vologodsky}), it suffices to show that
  for $\mathcal{F} \in \mathcal{D} _{[0,i(p-1) -1]} ((\Zp ^\syn)_{p^c = v_1 ^{i(j-1)} = 0})$,
  $\mathrm{Ext}^2 (\mathcal{F}, \mathcal{F}' \otimes \mathcal{O}_{(\Z_p ^\syn)_{p=v_1 = 0}} \{-(j-1)i(p-1)\}) = 0$.
  But this follows from applying $\pi_{-2}$ to the $\RHom$ statement above.
\end{proof}

Finally, we prove \Cref{thm:FL}.

\begin{proof}[Proof of \Cref{thm:FL}]
  By \Cref{prop:where-weight}, we have $R\pi^{\Syn} _* \mathcal{O}_{(X^\syn)_{p^c=0}} \in \mathcal{D} _{[0,d]} ((\Zp ^\syn)_{p = 0})$.
  It therefore follows from \Cref{prop:weight-equiv} that $R\pi^{\Syn} _* \mathcal{O}_{(X^\syn)_{p^c=0}}$ is functorially determined by its mod $v_1^i v_1^{kp^{c-1}}$ reduction whenever $d \leq kp^{c-1}(p-1)-1$, i.e. whenever $\frac{d+1}{p-1} \leq kp^{c-1}$.
  Now, it follows from \Cref{thm:stack}(1) that the mod $v_1 ^{p^{n-2}}$ reduction of $R\pi^{\Syn} _* \mathcal{O}_{(X^\syn)_{p^c=0}}$ is functorially determined by $X_{p^{cn} = 0}$.
  The smallest integer $n$ for which $\frac{d+1}{p-1} \leq p^{n-2}$ and $p^{n-2}$ is divisible by $p^{c-1}$ is $b_c ''(d) = \max\left(c+1, \lceil \log_p \left( \frac{d+1}{p-1} \right) + 2 \rceil \right)$, as desired.
\end{proof}

%\section{The tower $\mathrm{TR}(\mathbb{Z}/p^n)$}
%\input{almostTRzpn.tex}

%\section{Pro-constancy for $\THH$-bounded below rings}
%\input{ZtoS.tex}

\section{The multiplicative $p$-adic continuity of $\mathrm{TR}$}\label{sec:trcont}

Throughout this section, we fix an $m\geq 1$ and $c \geq 2$ so that $\mathbb{S}/p^c$ is an $\mathbb{E}_m$-algebra Moore spectrum and that $\mathbb{S}/p^{i}$ is a tower of $\mathbb{E}_{m+1}$-algebras for $i\geq b$. For any choice of $m$, such choices can be made by \cite{burklund2022multiplicative}. Our main result is the following:

\begin{thm} \label{thm:THHcrysmain}
For each $k \ge 0$, let $q = 2^k-1$. Then the tower of $\mathbb{E}_m$-$\tau_{\leq q}^{\cyc}\SP$-algebras in cyclotomic spectra
\[\{i \mapsto \THH(\mathbb{S}/p^{\max(i,b)})\otimes (\tau_{\leq q}^{\cyc}\mathbb{S})/p^c)\}\]
is rapidly converging to $\tau^{\mathrm{cyc}}_{\le q} (\mathbb{S})/p^c$ at rate $2^k(2c+1)$ and constant $b$.
\end{thm}

Tensoring with $\THH(R)$, applying $\TR$, and truncating (using \cite{antieau2021cartier}), we get:
\begin{cor}
	Let $R$ be an $\EE_{m+1}$-algebra with $\THH(R)$ bounded below.
	For each $k \ge 0$, let $q = 2^k-1$. Then the tower of $\mathbb{E}_m$-algebras in topological Cartier modules
	\[\{i \mapsto \tau_{\leq q}(\TR(R/p^{\max(i,b)})/p^c)\}\]
	is rapidly converging to $\tau_{\leq q}(\TR(R)/p^c)$ at rate $2^k(2c+1)$ and constant $b$.
\end{cor}

We note that we don't try to optimize the rate of convergence in our proof.

The key ingredient in the proof of \Cref{thm:THHcrysmain} is the following proposition.

%\begin{prop} \label{prop:THHcryspizero}
%The tower of cyclotomic spectra %for $n\geq 2c+1$
%\[ \{\THH(\mathbb{S}/p^n) 
%\otimes
%(\pi_0^{\cyc} \mathbb{S}) /p^c \}\]%_{n \ge ?}\]
%is rapidly converging at rate $2c+1$ and constant $0$ to $(\pi_0^{\cyc} \mathbb{S})/p^c$.
%\end{prop}
%
%We note that the natural map $\pi_0^{\cyc} \mathbb{S}_p \to \pi_0^{\cyc} \THH(\mathbb{Z}_p)$ is an equivalence, since both $\TR(\mathbb{Z}_p)$ and $\TR(\mathbb{S}_p)$ have bottom homotopy group $W(\mathbb{Z}_p)$.  Thus, we may rewrite the above proposition in the following more convenient form:

\begin{prop}\label{prop:THHbasecase}
 For any $m\geq 0$, the tower
  \[\{\THH(\mathbb{Z}/p^{n})\otimes_{\THH(\ZZ_p)} (\pi_0^{\cyc} \THH(\mathbb{Z}_p))/p^c \}\]
  is rapidly converging at rate $2c+1$ and constant $0$ to $(\pi_0^{\cyc} \THH(\mathbb{Z}_p)) / p^c$, as a tower of $\mathbb{E}_m$-$\pi_0^{\cyc}{\SP}$-algebras in cyclotomic spectra.
\end{prop}

%As a corollary, we obtain rapid convergence at the level of cyclotomic spectra.
%\begin{cor} \label{cor:main-unstructured-rapid-convergence}
%	For any fixed integer $q \ge 0$, as $n$ varies the tower of cyclotomic spectra
%	\[\{\tau^{\mathrm{cyc}}_{\le q} (\mathbb{S}) /p^c \otimes \THH(\mathbb{S}/p^n) \}\]
%	rapidly converges as a tower of $\tau_{\leq q}^{\cyc}\SP$-modules in cyclotomic spectra at rate $(q+1)(2c+1)$ and constant $0$.
%\end{cor}
%
%\begin{proof}
%	We prove this by induction on $q$ that, for $A=\tau^{\mathrm{cyc}}_{\le q} (\mathbb{S})$, the tower $\{A \otimes \THH(\mathbb{S}/p^n) / p^c\}$ is rapidly converging in the stably symmetric monoidal, idempotent-complete $\infty$-category of cyclotomic spectra.  
%	
%	The base case $q=0$ follows from \Cref{prop:THHbasecase}. The inductive argument follows from \Cref{lem:rapconvorientedpullback}, the fiber sequence 
%	$\Sigma^{q} \pi_q^{\heartsuit} \mathbb{S} 
%	\to \tau^{\mathrm{cyc}}_{\le q} (\mathbb{S}) 
%	\to \tau^{\mathrm{cyc}}_{\le q-1} (\mathbb{S})$,
%	and the fact that $\pi_q^{\heartsuit} \mathbb{S}$ is a $\pi_0^{\cyc} \mathbb{S}$-module in cyclotomic spectra.
%\end{proof}
We prove the main theorem of the section as a consequence of the above proposition:

\begin{proof}[Proof of \Cref{thm:THHcrysmain}] We prove by induction on $k\geq 0$ that if $q = 2^k-1$, the tower $$ \THH(\SP/p^n) \otimes R \otimes \SP/p^c$$ rapidly converges at rate $(q+1)(2c+1)$ and constant $b$ for every $\EE_m$-$\tau_{\leq q}\SP$-algebra $R$.
  Note that it suffices to prove this in the case $R = \tau_{\leq q} \Ss$, as for general $R$ we can base change along the unit map $\tau_{\leq q} \Ss \to R$.

  For $q=0$, the result follows from \Cref{prop:THHbasecase} and the equivalences of $\EE_{m}$-$\pi_0^{\cyc}\SP$-algebras in cyclotomic spectra for $n\geq b$:
  \begin{align*}
    \THH(\mathbb{Z}/p^{n}) \otimes_{\THH(\ZZ_p)} (\pi_0^{\cyc} \THH(\mathbb{Z}_p))/p^c
    &\simeq \THH(\Ss/p^{n}) \otimes (\pi_0^{\cyc} \THH(\mathbb{Z}_p))/p^c \\
    &\simeq \THH(\Ss/p^n) \otimes \pi_0 ^{\cyc} \Ss \otimes \Ss/p^c.
  \end{align*}
  In the second equivalence, we use the natural equivalence $\pi_0 ^{\cyc} \Ss_p \cong \pi_0 ^{\cyc} \THH(\Ss_p) \cong \pi_0 ^{\cyc} \THH(\Z_p)$.

	For the inductive step, consider the cospan
	$$\tau_{\leq \frac{q-1}2}^{\cyc}\SP \oplus \Sigma \tau_{\leq q}^{\cyc}\tau_{\geq \frac {q+1}2}^{\cyc}\SP \leftarrow \tau_{\leq {\frac {q-1}2}}^{\cyc}\SP \rightarrow \tau_{\leq \frac{q-1}2}^{\cyc}\SP \oplus \Sigma \tau_{\leq q}^{\cyc}\tau_{\geq \frac {q+1}2}^{\cyc}\SP$$ realizing $\tau_{\leq q}^{\cyc}\SP$ as a square-zero extension of $\tau_{\leq \frac{q-1}{2}}^{\cyc}\SP$. Note that each term in the pullback is individually an $\EE_{\infty}$-$\tau_{\leq {\frac {q-1}2}}^{\cyc}\SP$-algebra. After tensoring this cospan with $\THH(\SP/p^{\max(n,b)})\otimes \SP/p^c$, it follows from the induction hypothesis that each term in the cospan is rapidly converging of rate $2^{k-1}(2c+1)$ with constant $b$. It follows from \Cref{lem:rapconvorientedpullback} that the pullback is rapidly converging at rate $2^{k}(2c+1)$ with constant $b$, which is the desired statement.
\end{proof}

Before giving the proof of \Cref{prop:THHbasecase}, we establish a few lemmas. If $R$ is a discrete commutative ring, we denote by $W(R)$ the ring of $p$-typical Witt vectors on $R$, and by $\mathbb{W}(R)$ the ring of big Witt vectors on $R$.

\begin{lem}\label{lem:wittdivis}
Suppose $a,n$ are integers such that $n\ge 2c+1$.  Then any element in the kernel of the map
\[\mathbb{W}(\mathbb{Z}/p^{n+1}) \to \mathbb{W}(\mathbb{Z}/p^{n})\]
is divisible by $p^c$.  In particular, any element in the kernel of the map
\[W(\mathbb{Z}/p^{n+1}) \to W(\mathbb{Z}/p^{n})\]
is divisible by $p^c$.
%\todo{This is probably suboptimal, but explicit... do we care?  Anyone is welcome to try to make it more optimal if they care}
\end{lem}

\begin{proof}
For $\mathbb{Z}_{(p)}$-algebras $R$, $W(R)$ is a functorial summand of $\mathbb{W}(R)$ \cite[14.21]{HazewinkelWitt}, so the second assertion of the lemma follows from the first.  To prove the first sentence of the lemma, recall that as an abelian group $\mathbb{W}(R)$ is isomorphic to the multiplicative monoid of power series 
\[\mathbb{W}(R) \cong (1+R\llbracket t \rrbracket)^{\times}.\]
From this point of view, we need to check that any power series $1+r_1 t + r_2 t^2 + \cdots$, with each $r_i$ a multiple of $p^{n}$ in $\mathbb{Z}/p^{n+1}$, is the $(p^c)$th power of some other power series.

In fact, the power series $(1+\frac{r_1}{p^c} t + \frac{r_2}{p^c} t^2 + \cdots)$ will do.  The fact that $n \ge 2c+1$ ensures that any product $(\frac{r_i}{p^c})(\frac{r_j}{p^c})$ is trivial modulo $p^{n+1}$, so
\[\left(1+\frac{r_1}{p^c} t + \frac{r_2}{p^c} t^2 + \cdots\right)^{p^c} = (1+r_1 t + r_2 t^2 + \cdots)\]
as desired.
\end{proof}

\begin{rmk} \label{rmk:wittdivis}
As a corollary of the above lemma, as long as $n \ge 2c+1$ every element in the kernel of the surjective map
\[W(\mathbb{Z}_p) \to W(\mathbb{Z}/p^n)\]
is divisible by $p^c$.
\end{rmk}

\begin{lem}\label{lem:retbase}
For $n \ge 2c+1$, there exists a retraction of $\mathbb{E}_m$-$\pi_0^{\cyc}\SP/p^c$-algebras in cyclotomic spectra
\[
\begin{tikzcd}
  (\pi_0^{\cyc}\THH(\mathbb{Z}_p))/p^c \arrow{r} &   \THH(\mathbb{Z}/p^n) \otimes_{\THH(\mathbb{Z}_p)} (\pi_0^{\cyc}\THH(\mathbb{Z}_p))/p^c  \arrow[dashed, shift right=1ex]{l}
\end{tikzcd}
\]
\end{lem}

\begin{proof}
  Note that $\pi_0^{\cyc}\THH(\mathbb{Z}_p)$ is a cyclotomic spectrum with $\TR$ the discrete algebra in $p$-typical Cartier modules $W(\mathbb{Z}_p)$.  Since $W(\mathbb{Z}_p)$ is $p$-torsion free, $(\pi_0^{\cyc}\THH(\mathbb{Z}_p))/p^c$ is also discrete.

To produce the desired retraction it therefore suffices to produce one of the form
\[
\begin{tikzcd}
  \pi_0^{\cyc}\THH(\mathbb{Z}_p)/p^c \arrow{r} & \tau^{\mathrm{cyc}}_{\le 0} \left(\THH(\mathbb{Z}/p^n) \otimes_{\THH(\mathbb{Z}_p)} (\pi_0^{\cyc}\THH(\mathbb{Z}_p))/p^c \right)  \arrow[dashed, shift right=1ex]{l}
\end{tikzcd}
\] 
Translating to $p$-typical Cartier modules \cite{antieau2021cartier}, we must prove that the map $W(\mathbb{Z}_p) \to W(\mathbb{Z}/p^n)$ admits a section after taking the underived quotient by $p^c$.  By \Cref{lem:wittdivis}, after taking the underived quotient by $p^c$ this map becomes an isomorphism.
\end{proof}

\begin{cor} \label{cor:Zalgrapidconv}
Suppose $n \ge 2c+1$ and let $m\geq 0$. Then the functor 
  \[R \mapsto  \THH(R) \otimes_{\THH(\mathbb{Z}_p)} (\pi_0^{\cyc} \THH(\mathbb{Z}_p))/p^c,\]
from $\mathbb{E}_{m+1}$-$\mathbb{Z}_p$-algebras to $\mathbb{E}_{m}$-algebras in $\pi_0^{\cyc}\SP$-modules in cyclotomic spectra,
factors through the functor 
\[R \mapsto R/p^n\]
from $\mathbb{E}_{m+1}$-$\mathbb{Z}_p$-algebras to $\mathbb{E}_{m+1}$-$\mathbb{Z}/p^n$-algebras.
\end{cor}

\begin{proof}
%This functor is given by tensoring the $\mathbb{E}_m$-$\THH(\mathbb{Z}_p)$-algebra $\THH(R)$ along the $\mathbb{E}_m$-algebra map $\THH(\mathbb{Z}_p) \to \THH(\mathbb{Z}_p) / p^c$. 
  We may consider the functor on $\EE_m$-$\ZZ/p^n$-algebras $$R \mapsto \THH(R) \otimes_{\THH(\ZZ/p^n)} (\pi_0^{\cyc}\THH(\ZZ_p))/p^c$$
  where the map ${\THH(\ZZ/p^n)}\to (\pi_0^{\cyc}\THH(\ZZ_p))/p^c$ comes from the retraction in \Cref{lem:retbase}. Because $\THH(R/p^n) \cong \THH(R)\otimes_{\THH(\ZZ_p)}\THH(\ZZ/p^n)$, the fact that the map in \Cref{lem:retbase} is a retraction gives the desired factorization of functors.
\end{proof}

\begin{proof}[Proof of \Cref{prop:THHbasecase}]
  Applying \Cref{cor:Zalgrapidconv}, we see that it suffices to prove that the tower of $\mathbb{E}_{m+1}$-$\mathbb{Z}_p/p^{2c+1}$-algebras
$(\mathbb{Z}/p^{\bullet}) / p^{2c+1}$ is rapidly converging at rate $2c+1$ and constant $0$.
  This is a consequence of \Cref{lem:rapidconvergencefromcrystallinity}.
\end{proof}

\section{Crystallinity for topological Hochschild homology and algebraic $K$-theory}\label{sec:crysthhk}
In this section, we use \Cref{thm:THHcrysmain} to prove the generalized crystallinity result \Cref{thm:general-crys}, here proven as \Cref{thm:genmooregeneralcrystallinity} and \Cref{cor:ktheorycrys}.

\begin{prop}\label{prop:genmooreretract}
  Suppose that we are given $q = 2^k-1$, a $\mathbb{E}_{m}$-algebra structure on $\Ss/p^c$, a tower of $\mathbb{E}_{m+1}$-algebras
  \[\dots \to \Ss/p^{b+2} \to \Ss/p^{b+1} \to \Ss/p^{b},\]
  and an $\mathbb{E}_{m+1}$-map
  \[\Ss/p^{b+i_0} \to W = \Ss / (p^{b+i_0}, \dots, v_n ^{i_n}).\]
  If
  $i_0 \geq (q+1)(2c+1)$
  and
  $i_\ell \geq \frac{q+2}{2p^\ell-2},$
  for all $1 \leq \ell \leq n$,
  then the map
  \[(\tau^{\mathrm{cyc}}_{\le q}\mathbb{S})/p^c \to \THH(\mathbb{S}/(p^{b+i_0}, v_1 ^{i_1}, \dots, v_n ^{i_n})) \otimes (\tau^{\mathrm{cyc}}_{\le q}\mathbb{S})/p^c \]
  of $\mathbb{E}_{m}$-rings in cyclotomic spectra admits a retraction.
\end{prop}

\begin{proof}
  By \Cref{thm:THHcrysmain}, we know that the tower of $\mathbb{E}_{m}$-rings in cyclotomic spectra
    \[\{i \mapsto \THH(\mathbb{S}/p^{\mathrm{max}(i,b)}) \otimes (\tau^{\mathrm{cyc}}_{\le q}\mathbb{S})/p^c \}_i\]
  is rapidly converging at rate $(q+1)(2c+1)$ and constant $b$.
  Thus the map
  \[\tau^{\mathrm{cyc}}_{\le q}\mathbb{S} / p^c \to \THH(\mathbb{S}/p^{b+i_0}) \otimes (\tau^{\mathrm{cyc}}_{\le q}\mathbb{S})/p^c\]
  admits a cyclotomic $\mathbb{E}_{m}$-retraction for any $i_0 \geq (q+1)(2c+1)$ by definition.

  To show that 
  \[(\tau^{\mathrm{cyc}}_{\le q}\mathbb{S})/p^c \to \THH(\mathbb{S}/(p^{b+i_0}, v_1 ^{i_1}, \dots, v_n ^{i_n})) \otimes (\tau^{\mathrm{cyc}}_{\le q}\mathbb{S})/p^c \]
  admits an $\mathbb{E}_{m}$-retraction in cyclotomic spectra, we first note that because $(\tau^{\mathrm{cyc}}_{\le q}\mathbb{S})/p^c$
  is $(q+1)$-truncated in the cyclotomic $t$-structure, we may replace $\THH(\mathbb{S}/(p^{b+i_0}, v_1 ^{i_1}, \dots, v_n ^{i_n}))$ by $\tau_{\leq q+1} ^{\mathrm{cyc}} \THH(\mathbb{S}/(p^{b+i_0}, v_1 ^{i_1}, \dots, v_n ^{i_n}))$.
  Now, the fiber of the map
  \[\Ss/p^{b+i_0} \to \Ss / (p^{b+i_0}, \dots, v_n ^{i_n})\]
  is $\min_{1 \leq \ell \leq n} (2i_\ell (p^{\ell}-1))$-connective, which implies that the fiber of
  \[\THH(\Ss/p^{b+i_0}) \to \THH(\Ss / (p^{b+i_0}, \dots, v_n ^{i_n}))\]
  is $\min_{1 \leq \ell \leq n} (2i_\ell (p^{\ell}-1))$-connective in the cyclotomic $t$-structure.
  In particular, this map induces an isomorphism after applying $\tau^{\mathrm{cyc}}_{\leq q+1}$ if we assume that $\min_{1 \leq \ell \leq n} (2i_\ell (p^{\ell}-1)) \geq q+2$.
  In conclusion, we require the bounds
  $i_0 \geq (q+1)(2c+1)$
  and
  $2i_\ell (p^{\ell}-1) \geq q+2$
  for $1 \leq \ell \leq n$,
  and these are exactly what appear in the statement of the proposition.
\end{proof}

%	In the proof of the theorem below, it is convenient to use a standard $t$-structure on module categories in cyclotomic spectra, which we now define.
%	For $B$ an $\EE_1$-ring in cyclotomic spectra, the category $\Mod_{\CycSp}(B)$ has a standard $t$-structure where the connective objects are those of the form $B\otimes X$ for $X$ a connective cyclotomic spectrum. An object is coconnective if it is coconnective on underlying, which by \cite[Theorem 3.21]{antieau2021cartier} can be checked after applying $\TR$.

\begin{thm}\label{thm:genmooregeneralcrystallinity}
	Let $c,b,q=2^k-1, m\geq 1,\SP/p^c, W=\Ss / (p^{b+i_0}, \dots, v_n ^{i_n})$ be as in \Cref{prop:genmooreretract}.
	Suppose that $R$ is an $\mathbb{E}_{m+1}$-algebra for $m\geq1$, $n \ge 0$ an integer, and let $\Ss/p^c \to V$ denote an $\EE_{m}$-map where $V$ is a type $n+2$-complex such that $\TR(R) \otimes  V$ is $q$-truncated.
	
	Then the functor $\Alg_{\EE_m}(\Mod(R))\to \Alg_{\EE_{m-1}}(\Mod_{\CycSp}((\tau^{\cyc}_{\leq q}\SP)/p^c))$ 
	\[A \mapsto \THH(A) \otimes V\]
	factors through the functor $\Alg_{\EE_m}(\Mod(R))\to \Alg_{\EE_m}(\Mod(R\otimes W))$
	\[A \mapsto A \otimes W.\]
\end{thm}

\begin{proof}
	The functor $A \mapsto \THH(A)\otimes V$ in the theorem statement factors through $\Alg_{\EE_{m-1}}(\Mod_{\CycSp}(\THH(R)\otimes V))$.
	Since $\TR(R) \otimes V$ is $q$-truncated as a spectrum, there is an $\EE_m$-algebra map given as the composite $(\tau_{\le q}^{\cyc}\SP)/p^c \to \tau_{\leq q}^{\cyc} (\SP/p^c) \to \tau_{\geq0}(\THH(R)\otimes V) \to \THH(R)\otimes V$ and hence a forgetful functor $$\Alg_{\EE_{m-1}}(\Mod_{\CycSp}(\THH(R)\otimes V)) \to \Alg_{\EE_{m-1}}(\Mod_{\CycSp}((\tau_{\leq q}^{\mathrm{cyc}} \Ss)/p^c).$$
	
	Thus using the retraction in \Cref{prop:genmooreretract}, we learn that there is a natural equivalence of $\EE_{m-1}$-$(\tau_{\le q}^{\cyc}\SP)/p^c$-algebras
  $$(\THH(A\otimes W)\otimes V)\otimes_{\THH(W)\otimes (\tau_{\le q}^{\cyc}\SP)/p^c}(\tau_{\le q}^{\cyc}\SP)/p^c \cong \THH(A)\otimes V$$
	Since the left hand side only depends on $A\otimes W$ in $\Alg_{\EE_m}(\Mod(R\otimes W))$, we are done.
\end{proof}

\begin{rmk}
	In the above theorem, given an $\EE_{m+1}$-algebra $R$ satisfying the height $n$ Lichtenbaum--Quillen property \cite[Definition 4.2]{telescope}, it is always possible to find $V,W,c,b,q$ to make the above theorem apply using \cite{burklund2022multiplicative}.
	
	As some examples, \Cref{thm:genmooregeneralcrystallinity} applies in case of the $\EE_3$-ring $R=\BP\langle n-1 \rangle$ by \cite{hahn2022redshift}.
\end{rmk}

We are also above to prove a version of crystallinity for algebraic $K$-theory. In this case, we need to restrict our input to connective algebras.

\begin{cor} \label{cor:ktheorycrys}
	Let $c,b,q=2^k-1, m\geq 1,\SP/p^c, W=\Ss / (p^{b+i_0}, \dots, v_n ^{i_n})$ be as in \Cref{prop:genmooreretract}.
	Suppose that $R$ is a connective $\mathbb{E}_{m+1}$-algebra for $m\geq2$, $n \ge 0$ an integer, and let $\Ss/p^c \to V$ denote an $\EE_{m}$-map where $V$ is a type $n+2$-complex such that $\TR(R) \otimes  V$ is $q$-truncated. Let $\Alg_{\EE_m}'(\Mod(R)_{\geq 0})$ be the category of $p$-complete connective $\EE_m$-$R$-algebras such that either $p$ is nilpotent, or $\pi_0$ is commutative\footnote{This is automatic for $m\geq2$}.
	
  Then the functor $\Alg_{\EE_m}' (\Mod(R)_{\geq 0})\to \Alg_{\EE_{m-1}} (\mathrm{Sp}_{\geq 0})$ 
	\[A \mapsto K(A) \otimes V\]
  factors through the functor $\Alg_{\EE_m}(\Mod(R)_{\geq 0})\to \Alg_{\EE_m}(\Mod(R\otimes W)_{\geq 0})$
	\[A \mapsto A \otimes W.\]
\end{cor}

\begin{proof}
  Consider the diagram:
  \begin{center}
    \begin{tikzcd}
      K(A) \otimes V \ar[r] \ar[d] & \TC(A) \otimes V \ar[d] \\
      K(A \otimes W) \otimes V \ar[r] \ar[d] & \TC(A \otimes W) \otimes V \ar[d] \\
      K(\pi_0 (A) / p) \otimes V \ar[r] & \TC(\pi_0 (A) / p) \otimes V
    \end{tikzcd}
  \end{center}
  Here, we have used the fact that $\pi_0 (A) / p \to \pi_0 (A \otimes W) / p$ is an isomorphism.
  Now, the bottom square is Cartesian by \cite{DGM} and the outer square is Cartesian by \cite{DGM} and in the case $\pi_0$ is commutative also by \cite[Theorem 1.5]{CMM}, since $A$ is $p$-complete, hence $\pi_0 A$ is derived $p$-complete, hence $(p) \subset \pi_0A$ is Henselian.
  As a consequence, the upper square is also Cartesian.

  Now, the middle horizontal arrow clearly factors through $A \mapsto A \otimes W$.
  It therefore suffices to show that the top right vertical arrow factors through $A \mapsto A \otimes W$.
  It follows from \Cref{thm:genmooregeneralcrystallinity} that this arrow only depends on the map $A \otimes W \to A \otimes W \otimes W$ of $\E_m$-algebras over $R \otimes W$.
  We conclude by noting that, up to transposing the two factors of $W$ in the target, this may be identified with the map $B \to B \otimes W$ that exists for any $\E_m$-$(R \otimes W)$-algebra $B$.
\end{proof}

\section{Invariants of $R \langle \epsilon \rangle$}\label{sec:epsilon}
By \Cref{thm:mainintro}, we can compute the mod $(p,v_1^{p^n})$ syntomic cohomology of any animated ring $R$ as a functor of the (derived) quotient $R/p^{n+2}$.  If $R=\mathbb{Z}/p^m$ for $m \ge n+2$, then $R/p^{n+2} \cong \mathbb{Z} \langle \epsilon \rangle / p^{n+2}$, where $\mathbb{Z} \langle \epsilon \rangle$ is defined as follows:

\begin{dfn}
We let $\mathbb{Z} \langle \epsilon \rangle$ denote $\mathbb{Z} \otimes_{\mathbb{Z}[x]} \mathbb{Z}$, where $x \mapsto 0$ along both augmentations.  In other words, $\mathbb{Z} \langle \epsilon \rangle$ is the free animated ring on a degree $1$ class.
\end{dfn}

In particular, for $m \ge n+2$, the mod $(p,v_1^{p^n})$ syntomic cohomology of $\mathbb{Z}/p^m$ is the same as the mod $(p,v_1^{p^n})$ syntomic cohomology of $\mathbb{Z} \langle \epsilon \rangle$.  We will spend the rest of the paper computing the mod $(p,v_1^{p^n})$ syntomic cohomology of $\mathbb{Z} \langle \epsilon \rangle$ and drawing consequences for algebraic $K$-theory.

To calculate the syntomic cohomology of $\mathbb{Z} \langle \epsilon \rangle$ by quasi-syntomic descent, in this section we will understand the syntomic cohomology of $R \langle \epsilon \rangle = R \otimes_{\mathbb{Z}} \mathbb{Z}\langle \epsilon \rangle$ for quasiregular semiperfectoid rings $R$. We will first recall a classical decomposition of $\THH(R \langle \epsilon \rangle)$ in the category of cyclotomic spectra, which straightforwardly implies the following theorem:

\begin{thm} \label{thm:classical-tc-epsilon}
For any animated ring $R$, there is an equivalence:
\[\displaystyle \TC(R \langle \epsilon \rangle)_p^{\wedge} \simeq \TC(R)_p^{\wedge} \oplus \bigoplus_{p \nmid \ell} \TR (R;\Sigma^{2\ell} R)_p^{\wedge}.\]
Here, the sum ranges over all positive integers $\ell$ not divisible by $p$, and $\TR (R;\Sigma^{2 \ell} R)$ refers to $\langle p \rangle$-polygonic $\TR$ with coefficients in the bimodule $\Sigma^{2\ell} R$.
\end{thm}

\begin{rmk}
Since any animated ring $R$ is connective, for each $\ell > 0$ the homotopy groups of $\TR (R;\Sigma^{2\ell} R)_p^{\wedge}$ are concentrated in degrees $2\ell-1$ and above. In particular, the infinite sum appearing in \Cref{thm:classical-tc-epsilon} is also an infinite product, and therefore $p$-complete.
\end{rmk}

Next, we observe that the above decomposition interacts well with the Bhatt--Morrow--Scholze motivic filtration.  In particular:

\begin{thm} \label{thm:mot-fil-epsilon-main}
Suppose that $R$ is a qrsp ring.  Then, functorially in $R$, $\fil^*_{\mot} \TC(R \langle \epsilon \rangle)_p^{\wedge}$ is the equalizer of the diagram of filtered $\mathbb{E}_{\infty}$-rings 
\[
\begin{tikzcd}[column sep=huge]
\tau_{\ge 2*} \TC^{-}(R\langle \epsilon \rangle)_p^{\wedge} \arrow[shift left]{r}{\tau_{\ge 2*} \varphi^{hS^1}} \arrow[shift right, "\tau_{\ge 2*} \mathrm{can}" below]{r} & \tau_{\ge 2*} \TP(R\langle \epsilon \rangle)_p^{\wedge}.
\end{tikzcd}
\]
Explicitly, defining
\[\fil^* _{\mot} \TR(R;\Sigma^{2\ell} R)_p^{\wedge}\]
to be the limit of the diagram
\[
\begin{tikzcd}[column sep=-14ex]
& \tau_{\ge 2*} \left(\THH(R;(\Sigma^{2\ell}R)^{\otimes p})^{tC_p}\right)_p^{\wedge} && \tau_{\ge 2*} \left(\THH(R;(\Sigma^{2\ell}R)^{\otimes p^2})^{tC_{p^2}}\right)_p^{\wedge} && \cdots \\
\tau_{\ge 2*}\THH(R;\Sigma^{2\ell} R)_p^{\wedge} \arrow{ur}{\varphi} & & \tau_{\ge 2*} \left(\THH(R;(\Sigma^{2\ell}R)^{\otimes p})^{hC_p}\right)_p^{\wedge} \arrow{ur}{\varphi^{hC_p}} \arrow{ul}{\mathrm{can}} && \tau_{\ge 2*} \left(\THH(R;(\Sigma^{2\ell}R)^{\otimes p^2})^{hC_{p^2}}\right)_p^{\wedge} \arrow{ur}{\varphi^{hC_{p^2}}} \arrow{ul}{\mathrm{can}}
\end{tikzcd}
\]
of nonunital filtered $\mathbb{E}_{\infty}$-rings, there is a decomposition
\[\fil^*_{\mot} \TC(R \langle \epsilon \rangle)_p^{\wedge} \cong \fil^*_{\mot} \TC(R)_p^{\wedge} \oplus \bigoplus_{p \nmid \ell} \fil^* _{\mot} \TR(R;\Sigma^{2\ell} R)_p^{\wedge}.\]
\end{thm}

The content of the above theorem is the identification of the motivic filtration on
$\TC^{-}(R\langle \epsilon \rangle)_p^{\wedge}$
and
$\TP (R\langle \epsilon \rangle)_p^{\wedge}$
with the double-speed Postnikov filtration when $R$ is qrsp.
This is not formal: while $R$ is qrsp, $R\langle \epsilon \rangle$ is certainly not!
This allows us to relate the motivic filtration on $\TC(R\langle \epsilon \rangle)_p ^{\wedge}$, which is defined in terms of $R\langle \epsilon \rangle$, to a filtration on $\TR(R; \Sigma ^{2\ell} R)_p ^{\wedge}$ that is defined in terms of $R$.
Descending from qrsp rings to $p$-quasisyntomic rings, and then left Kan extending from $p$-completed polynomial rings to $p$-complete animated rings, we obtain the following corollary:

\begin{cor} \label{cor:mot-fil-epsilon-main}
  Let $R$ denote a $p$-complete animated ring. Define
  \[\fil^* _{\mot} \TR(R;\Sigma^{2\ell} R)_p^{\wedge}\]
  to be the limit of the diagram
  \[
\begin{tikzcd}[column sep=-14ex]
  & \fil^*_{\mot} \left(\THH(R;(\Sigma^{2\ell}R)^{\otimes p})^{tC_p}\right)_p^{\wedge} && \fil^*_{\mot} \left(\THH(R;(\Sigma^{2\ell}R)^{\otimes p^2})^{tC_{p^2}}\right)_p^{\wedge} && \cdots \\
\fil^*_{\mot} \THH(R;\Sigma^{2\ell} R)_p^{\wedge} \arrow{ur}{\varphi} & & \fil^*_{\mot} \left(\THH(R;(\Sigma^{2\ell}R)^{\otimes p})^{hC_p}\right)_p^{\wedge} \arrow{ur}{\varphi^{hC_p}} \arrow{ul}{\mathrm{can}} && \fil^*_{\mot} \left(\THH(R;(\Sigma^{2\ell}R)^{\otimes p^2})^{hC_{p^2}}\right)_p^{\wedge} \arrow{ur}{\varphi^{hC_{p^2}}} \arrow{ul}{\mathrm{can}}
\end{tikzcd}
  \]
  Then there is a decomposition
  \[\fil^*_{\mot} \TC(R \langle \epsilon \rangle)_p^{\wedge} \cong \fil^*_{\mot} \TC(R)_p^{\wedge} \oplus \bigoplus_{p \nmid \ell} \fil^* _{\mot} \TR(R;\Sigma^{2\ell} R)_p^{\wedge}.\]
\end{cor}
%
%\begin{rmk}
%Since \Cref{thm:mot-fil-epsilon-main} is functorial in qrsp rings, it determines the motivic filtration on $\TC(R \langle \epsilon \rangle)_p^{\wedge}$ for any $p$-quasisyntomic $R$, such as $R=\mathbb{Z}$.
%\end{rmk}

\subsection{The cyclotomic spectrum $\THH(R \langle \epsilon \rangle)$}

In order to identify $\THH(R \langle \epsilon \rangle)$ as a cyclotomic spectrum, it will be helpful to note the following alternative definition of the ring $R \langle \epsilon \rangle$:

\begin{lem} \label{lem:square-zero}
For any animated ring $R$, $R \langle \epsilon \rangle$ is equivalent to the split square-zero extension of $R$ by the bimodule $\Sigma R$.
\end{lem}

\begin{proof}
The homotopy groups of $R \langle \epsilon \rangle = R \otimes_{R[x]} R$ may be calculated by a $\mathrm{Tor}$ spectral sequences, which immediately collapses to give the formula $\pi_*R \langle \epsilon \rangle \cong \Lambda_{\pi_*R}(\epsilon)$, where $|\epsilon|=1$.  This homotopy ring is the same as that of the split square-zero extension $R \oplus \Sigma R$, so it remains only to find a map inducing an isomorphism.  To produce such a map, we simply note that $R \langle \epsilon \rangle$ is the free animated $R$-algebra on a degree $1$ class, so there is an $R$-algebra map that sends $\epsilon$ to the corresponding degree $1$ class in the split square-zero extension.
\end{proof}

In light of \Cref{lem:square-zero}, we may understand $\THH(R\langle \epsilon \rangle)$ in terms of the literature on $\THH$ of square-zero extensions, which is a cornerstone of the proof of the Dundas--Goodwillie--McCarthy theorem.  In particular, we recall from \cite[Proposition 4.5.1]{SamDGM} (see also \cite{LindenstraussMcCarthy}) that there is a natural equivalence of $S^1$-equivariant spectra: 
\[\THH(R \langle \epsilon \rangle) \simeq \THH(R) \oplus \bigoplus_{k \ge 1}\mathrm{Ind}_{C_n} ^{S^1} \THH(R; (\Sigma^{2}R)^{\otimes k}),\]
%\[\THH(R \langle \epsilon \rangle) \simeq \THH(R) \oplus \bigoplus_{k \ge 1} \THH(R) \otimes \Hom((S^1/C_k)_+,(\Sigma^{2}R)^{\otimes k})\]
where $\mathrm{Ind}_{C_n} ^{S^1}$ is the \emph{right} adjoint to the forgetful functor.

\begin{rmk}
Since $R$ is connective, for each $k \ge 1$ the spectrum underlying 
\[\mathrm{Ind}_{C_n} ^{S^1} \THH(R; (\Sigma^{2}R)^{\otimes k})\]
%\[\THH(R) \otimes \Hom((S^1/C_k)_+,(\Sigma^{2} R)^{\otimes k})\]
is $(2k-1)$-connective.  In particular, the infinite direct sum above is also an infinite product.
\end{rmk}

By the previous remark,
\[\THH(R \langle \epsilon \rangle)^{tC_p} \simeq \THH(R)^{tC_p} \times \prod_{k \ge 1} \left( \mathrm{Ind}_{C_n} ^{S^1} \THH(R; (\Sigma^{2}R)^{\otimes k}) \right)^{tC_p}.\]

In terms of the above direct product decomposition above, the cyclotomic Frobenius is given by a direct product of the following maps \cite[Lemma 4.8.1]{SamDGM}:

\begin{enumerate}
\item The cyclotomic Frobenius 
\[\THH(R) \to \THH(R)^{tC_p}\]
\item For each integer $k \ge 1$, the Frobenius
  \[\mathrm{Ind}_{C_n} ^{S^1} \THH(R; (\Sigma^{2}R)^{\otimes k}) \to \left( \mathrm{Ind}_{C_n} ^{S^1} \THH(R; (\Sigma^{2}R)^{\otimes k}) \right)^{tC_p},\]
    %\THH(R) \otimes \Hom((S^1/C_k)_+,(\Sigma^{2} R)^{\otimes k}) \to \left( \THH(R) \otimes \Hom((S^1/C_{pk})_+,(\Sigma^{2}R)^{\otimes pk}) \right)^{tC_p},\]
obtained by setting $M=\Sigma^2R$ in the construction of \cite[Theorem 6.31]{polygonicspectra}.
\end{enumerate}

%\begin{rmk}
%Recall from \cite[Theorem 6.31]{polygonicspectra} that, if $R$ is a ring and $M$ a bimodule, there is a Frobenius map
%\[\THH(R;M^{\otimes k}) \to \THH(R;M^{\otimes pk})^{tC_p}.\]
%\end{rmk}

We also have the following definitions:
\begin{dfn}
The $\langle p \rangle$-polygonic $\mathrm{TR}$ of $R$ with coefficients in $M$, denoted by $\TR(R;M)_p^{\wedge}$, is the limit of the diagram
\[
\begin{tikzcd}[column sep=-10ex]
& \left(\THH(R;(M)^{\otimes p})^{tC_p}\right)_p^{\wedge} && \left(\THH(R;(M)^{\otimes p^2})^{tC_{p^2}}\right)_p^{\wedge} && \cdots \\
\THH(R;M)_p^{\wedge} \arrow{ur}{\varphi} & & \left(\THH(R;(M)^{\otimes p})^{hC_p}\right)_p^{\wedge} \arrow{ur}{\varphi^{hC_p}} \arrow{ul}{\mathrm{can}} &&  \left(\THH(R;(M)^{\otimes p^2})^{hC_{p^2}}\right)_p^{\wedge} \arrow{ur}{\varphi^{hC_{p^2}}} \arrow{ul}{\mathrm{can}}
\end{tikzcd}
\]
\end{dfn}

Based on this definition, it is clear that
\[\mathrm{TC}(R \langle \epsilon \rangle)_p^{\wedge} \simeq  \mathrm{TC}(R)_p^{\wedge} \times \prod_{p \nmid \ell} \TR(R;\Sigma^{2\ell} R)_p^{\wedge} \simeq  \mathrm{TC}(R)_p^{\wedge} \oplus \bigoplus_{p \nmid \ell} \TR(R;\Sigma^{2\ell} R)_p^{\wedge}.\]

We also define truncated variants of $\TR$ (which we use in \Cref{sec:compute}) $\TR^{[n]}(R;M)_p^{\wedge}$, as the fiber of the map %\todo{is this indexing standard though?}

$$\prod_0^{n}(\THH(R;(M)^{\otimes p^i}))^{hC_{p^i}} \xrightarrow{\can-\varphi^{hC_{p^i}}} \prod_1^{n}(\THH(R;(M)^{\otimes p^i})^{tC_{p^i}}$$

Similarly, we can define motivic filtrations $\fil_{\mot}^*\TR^{[n]}(R;M)_p^{\wedge}$ and associated gradeds $\gr_{\mot}^*\TR^{[n]}(R;M)_p^{\wedge}$ as in \Cref{cor:mot-fil-epsilon-main} by truncating the diagram.

Note that $\lim_n\fil_{\mot}^*\TR^{[n]}(R;M)_{p}^{\wedge} = \fil_{\mot}^*\TR(R;M)_p^{\wedge}$.

\subsection{The motivic filtration}

If $R$ is a qrsp ring, then the Bhatt--Morrow--Scholze motivic filtration on $\THH(R)_p^{\wedge}$ is given by the double speed Postnikov filtration $\tau_{\ge 2*} \THH(R)_p^{\wedge}$ \cite{BMS}.  In this section, we observe that the same is true of $\THH(R \langle \epsilon \rangle)_p ^{\wedge}$. Namely,
\[\mathrm{fil}^*_{\mot} \THH(R\langle \epsilon \rangle)_p^{\wedge} = \tau_{\ge 2*} \THH(R \langle \epsilon \rangle)_p^{\wedge}.\]
In this sense $R\langle \epsilon \rangle$ behaves much like a discrete quasi-regular semi-perfectoid ring, despite the fact that $R \langle \epsilon \rangle$ is not discrete. The main theorem is as follows:

\begin{thm} \label{thm:mot-doublespeed}
Let $R$ be a qrsp ring.  Then
\[\mathrm{fil}^*_{\mot} \THH(R\langle \epsilon \rangle)_p^{\wedge} \simeq \tau_{\ge 2*} \THH(R \langle \epsilon \rangle)_p^{\wedge},\]
\[\mathrm{fil}^*_{\mot} \TC^{-}(R \langle \epsilon \rangle)_p^{\wedge} \simeq \tau_{\ge 2*} \TC^{-}(R \langle \epsilon \rangle)_p^{\wedge},\]
and
\[\mathrm{fil}^*_{\mot} \TP(R \langle \epsilon \rangle) \simeq \tau_{\ge 2*} \TP(R \langle \epsilon \rangle).\]

\end{thm}

\begin{rmk}
The above theorem is stated for a single qrsp ring $R$, but in fact the identification of motivic filtrations with double-speed Postnikov filtrations is functorial in $R$.  Indeed, if $A \to B$ is any map of $\mathbb{E}_{\infty}$-rings there is a unique (up to contractible choice) lift of such a map to a filtered $\mathbb{E}_{\infty}$-ring map $\tau_{\ge 2*} A \to \tau_{\ge 2*} B$.  Thus, the above theorem identifies the motivic filtrations on $\THH(R \langle \epsilon \rangle)_p^{\wedge}$, $\TC^{-}(R \langle \epsilon \rangle)_p^{\wedge},$ and $\TP(R \langle \epsilon \rangle)_p^{\wedge}$ as sheaves on the quasisyntomic site. Similarly, this implies that the motivic filtrations on the maps
\[\varphi,\mathrm{can}:\TC^{-}(R)_p^{\wedge} \to \TP(R)_p^{\wedge}\]
are $\tau_{\ge 2*} \varphi$ and $\tau_{\ge 2*} \mathrm{can}$, respectively.
\end{rmk}

Our proof of \Cref{thm:mot-doublespeed} begins with a description of the motivic associated graded on $\THH(R \langle \epsilon \rangle)$:

\begin{prop} \label{prop:THHmodzero}
Suppose that $R$ is a qrsp ring.  Then
\[\pi_*\gr^*_{\mot} \THH(R \langle \epsilon \rangle) = (\pi_*\gr^*_{\mot} \THH(R)) \langle \epsilon \rangle \otimes \Gamma(\sigma \epsilon),\]
where $\lVert \epsilon \rVert = (1,-1)$ and $\lVert \sigma\epsilon^{(k)}  \rVert = (2k,0)$.  

As a consequence,
\[\pi_*\gr^*_{\mot} \THH(R \langle \epsilon \rangle)_p^{\wedge} = \left(\pi_*\gr^*_{\mot} \THH(R)_p^{\wedge}\right) \langle \epsilon \rangle \otimes \Gamma(\sigma \epsilon),\]
where $\lVert \epsilon \rVert = (1,-1)$, $\lVert \sigma\epsilon^{(k)}  \rVert = (2k,0)$, and all classes in $\pi_*\gr^*_{\mot} \THH(R)_p^{\wedge}$ have Adams filtration $0$.
\end{prop}

\begin{proof}
As an $\mathbb{E}_{\infty}$-ring spectrum,
\[R\langle \epsilon \rangle \simeq R \otimes_{R[x]} R.\]
which may be calculated as the geometric realization of the simplicial commutative ring

\[
\begin{tikzcd}
R \arrow{r} &  R[x_1] \arrow[shift left=2]{l} \arrow[shift right=2]{l} \arrow[shift left=2]{r} \arrow[shift right=2]{r} & R[x_1,x_2] \arrow[shift left=4]{l} \arrow[shift right=4]{l} \arrow{l} \arrow[shift left=4]{r} \arrow{r} \arrow[shift right=4]{r}& \cdots \arrow[shift right=6]{l} \arrow[shift right=2]{l} \arrow[shift left=2]{l} \arrow[shift left=6]{l}.
\end{tikzcd}
\]

Applying $\THH$, we see that $\THH(R \langle \epsilon \rangle)$ may be calculated as the geometric realization of the simplicial ring spectrum

\[
\begin{tikzcd}
\THH(R) \arrow{r} &  \THH(R[x_1]) \arrow[shift left=2]{l} \arrow[shift right=2]{l} \arrow[shift left=2]{r} \arrow[shift right=2]{r} & \THH(R[x_1,x_2]) \arrow[shift left=4]{l} \arrow[shift right=4]{l} \arrow{l} \arrow[shift left=4]{r} \arrow{r} \arrow[shift right=4]{r}& \cdots \arrow[shift right=6]{l} \arrow[shift right=2]{l} \arrow[shift left=2]{l} \arrow[shift left=6]{l}.
\end{tikzcd}
\]

The motivic filtration on $\THH(R \langle \epsilon \rangle)$ is then computed as the geometric realization of
\[
\begin{tikzcd}
\fil^*_{\mot} \THH(R)_p \arrow{r} & \fil_{\mot}^* \THH(R[x_1]) \arrow[shift left=2]{l} \arrow[shift right=2]{l} \arrow[shift left=2]{r} \arrow[shift right=2]{r} & \fil^*_{\mot} \THH(R[x_1,x_2]) \arrow[shift left=4]{l} \arrow[shift right=4]{l} \arrow{l} \arrow[shift left=4]{r} \arrow{r} \arrow[shift right=4]{r}& \cdots \arrow[shift right=6]{l} \arrow[shift right=2]{l} \arrow[shift left=2]{l} \arrow[shift left=6]{l}.
\end{tikzcd}
\]
In particular, we may calculate
$\pi_*\gr^*_{\mot} \THH(R)$ using the spectral sequence beginning with the cohomology of the simplicial ring
\[
\begin{tikzcd}
\pi_*\gr^*_{\mot} \THH(R) \arrow{r} & \pi_*\gr_{\mot}^* \THH(R[x_1]) \arrow[shift left=2]{l} \arrow[shift right=2]{l} \arrow[shift left=2]{r} \arrow[shift right=2]{r} & \pi_*\gr^*_{\mot} \THH(R[x_1,x_2]) \arrow[shift left=4]{l} \arrow[shift right=4]{l} \arrow{l} \arrow[shift left=4]{r} \arrow{r} \arrow[shift right=4]{r}& \cdots \arrow[shift right=6]{l} \arrow[shift right=2]{l} \arrow[shift left=2]{l} \arrow[shift left=6]{l}.
\end{tikzcd}
\]
We now note that, by the HKR theorem, 
\[\pi_*\gr^*_{\mot} \THH(R[x_1,x_2,\cdots,x_k]) \cong \pi_*\gr^*_{\mot}\THH_*(R)[x_1,x_2,\cdots,x_k]\langle \sigma x_1,\sigma x_2,\cdots,\sigma x_k \rangle.\]
The above cobar complex computes
\[\mathrm{Tor}_{\pi_*\gr^*_{\mot}\THH(R)[x] \langle \sigma x \rangle}(\pi_*\gr^*_{\mot}\THH(R),\pi_*\gr^*_{\mot}\THH(R)),\]
which is $\pi_*\gr^*_{\mot}\THH(R) \langle \epsilon \rangle \otimes \Gamma(\sigma \epsilon)$.  The class $\epsilon$ is represented by $x_1$, and therefore has Adams filtration $-1$.  The class $\sigma \epsilon$ is represented by $\sigma x_1$, and therefore has Adams filtration $0$.
\end{proof}

%\begin{cor}
%Suppose that $A$ is any quasisyntomic discrete ring.  It follows from the above that
%\[\pi_*\gr^*_{\mot} \THH(A \langle \epsilon \rangle) = (\pi_*\gr^*_{\mot} \THH(A)) \langle \epsilon \rangle \otimes \Gamma(\sigma \epsilon),\]
%where $\lVert \epsilon \rVert = (1,-1)$ and $\lVert \sigma\epsilon^{(k)}  \rVert = (2k,0)$.
%\end{cor}

To deduce \Cref{thm:mot-doublespeed}, we will apply the following lemma:

\begin{lem} \label{lem:identify-doublespeed}
Let $\mathrm{fil}^* R$ be a complete filtered spectrum, with colimit spectrum $R$, such that the bigraded homotopy groups $\pi_* \gr^* R$ are concentrated in Adams filtrations $0$ and $-1$.  Then there is a canonical identification $\tau_{\ge 2*} R \simeq \mathrm{fil}^*R$.
\end{lem}

\begin{proof}
By assumption, for each integer $k$, $\gr^k R$ has homotopy groups concentrated in degrees $2k$ and $2k+1$. Since the filtered spectrum is complete, it follows that the $k$-th filtered piece is $2k$-connective. The canonical map from $\mathrm{fil}^*R$ to $R$ with the constant filtration factors through $\tau_{\geq2*}R$, and the resulting map is easily seen to be an equivalence.
\end{proof}

\begin{proof}[Proof of \Cref{thm:mot-doublespeed}]
By combining \Cref{prop:THHmodzero} with \Cref{lem:identify-doublespeed}, we deduce immediately that 
  \[\fil^*_{\mot} \THH(R \langle \epsilon \rangle)_p ^{\wedge} \simeq \tau_{\ge 2*} \THH(R \langle \epsilon \rangle)_p ^{\wedge}.\]
There is an algebraic homotopy fixed point spectral sequence
  \[\pi_*(\gr^*_{\mot} \THH(R \langle \epsilon \rangle))_p ^{\wedge} [t] \implies \pi_* \gr^*_{\mot} \TC^{-}(R)_p ^{\wedge},\]
  $\lVert t \rVert = (-2,0)$, from which one learns that the bigraded homotopy groups $\pi_* \gr^*_{\mot} \TC^{-}(R)_p ^{\wedge}$ are concentrated in Adams filtration $0$ and $-1$.  Similarly, the algebraic Tate spectral sequence
  \[\pi_*(\gr^*_{\mot} \THH(R \langle \epsilon \rangle))_p ^{\wedge} [t^{\pm 1}] \implies \pi_* \gr^*_{\mot} \TP(R)_p ^{\wedge}\]
  implies that $\pi_* \gr^*_{\mot} \TP(R)_p ^{\wedge}$ is concentrated in Adams filtrations $0$ and $-1$.  Applying \Cref{lem:identify-doublespeed}, we conclude the result.
\end{proof}

\section{The mod $p$ syntomic cohomology of $\mathbb{Z}_p\langle \epsilon \rangle$}\label{sec:compute}

In this section, we compute the mod $p$ syntomic cohomology of $\Zp \langle \epsilon \rangle$.
Our strategy will be to use the associated graded of \Cref{cor:mot-fil-epsilon-main}
to reduce this to the computation of the homotopy groups of 
\[\gr^* _{\mot} \left(\THH(\Zp;(\Sigma^{2\ell}\Zp)^{\otimes_{\Zp} p^n})^{hC_{p^n}}\right)/p\]
for $\ell \nmid p$ and $n \geq 0$,
and the Frobenius and canonical maps between them.
%Before we start, we note that by tracing through the definitions, there is a natural equivalence of $C_{p^n}$-spectra:
%\[\THH(\Zp;(\Sigma^{2\ell}\Zp)^{\otimes_{\Zp} p^n}) \simeq \THH(\Zp) \otimes \Ss^{\ell \rho_n},\]
%where $\rho_n$ is the complex regular representation of $C_{p^n}$.

In \Cref{sec:twisted}, we use an algebraic version of homotopy fixed points spectral sequence to study the groups
\[\pi_* \gr^* _{\mot} \left(\THH(\Zp;(\Sigma^{2\ell} \Zp)^{\otimes_{\Zp} p^n})^{hC_{p^n}}\right)/p.\]
As a preliminary step, we show that there is a natural equivalence of $C_{p^n}$-spectra:
\[\THH(\Zp;(\Sigma^{2\ell}\Zp)^{\otimes_{\Zp} p^n}) \simeq \THH(\Zp) \otimes \Ss^{\ell \rho_n},\]
where $\rho_n$ is the complex regular representation of $C_{p^n}$.
%If we remove the motivic associated graded and take $\ell =0$, the form of this spectral sequence was the subject of a conjecture of B\"okstedt and Madsen \cite[Conjecture 4.3]{BokMad} that was proven by Tsalidis \cite[Theorem 6.1]{Tsalidis}.
%We carry out the computation for general $\ell$ by studying how the the twisting imparted by $\ell$ affects the computation.

Then, in \Cref{sec:v1}, we compute the Frobenius and canonical maps up to $v_1$-adic filtration, where they admit simple formulas.

Finally, in \Cref{sec:computation}, we use these two ingredients together with \Cref{cor:mot-fil-epsilon-main} to compute the mod $p$ syntomic cohomology of $\Zp \langle \epsilon \rangle$.
Along the way, we will see that the up-to-$v_1$-adic-filtration computations of \Cref{sec:v1} are sufficient to compute the full answer.

\begin{ntn}
	As notation, we will use a dot as in $\dot{+},\dot{=},\dot{\mapsto}$ to mean ``up to a unit in $\FF_p^{\times}$.'' For example, $a\dot{=}b$ means $a=cb$ for $c \in \FF_p^\times$.
\end{ntn}

\subsection{Twisted spectral sequences}\label{sec:twisted}

As above, we let $\rho_n$ denote the complex regular representation of $C_{p^n}$.

Before we start, we establish an equivalence of $C_{p^n}$-spectra:
\[\THH(\Zp;(\Sigma^{2\ell}\Zp)^{\otimes_{\Zp} p^n}) \simeq \THH(\Zp) \otimes \Ss^{\ell \rho_n}.\]
This will be an immediate consequence of \Cref{cor:polygonic-R} below.

\begin{lem} \label{lem:polygonic-sphere}
  Given any spectrum $V$, there is a natural equivalence of spectra with $C_n$-action
  \[\THH(\Ss; V^{\otimes n}) \simeq V^{\otimes n},\]
  where $V^{\otimes n}$ is equipped with the action of $C_n$ by cyclic permutations.
\end{lem}

\begin{proof}
  By definition, the $C_n$-action on $\THH(\Ss; V^{\otimes n})$ is inherited from the equivalence between $\THH(\Ss; V^{\otimes n})$ and the simplicial diagram of the form $\{(\Ss^{\otimes n})^{\otimes \bullet} \otimes V^{\otimes n}\}$, which is equipped with the cyclic permutation action of $C_n$.
  This simplicial diagram is equivalent to the constant simplicial diagram at the spectrum with $C_n$-action $V^{\otimes n}$, from which the result follows.
\end{proof}

\begin{cor} \label{cor:polygonic-R}
  Given any $\mathbb{E}_\infty$-ring $R$ and spectrum $V$, there is a natural equivalence of spectra with $C_n$-action
  \[\THH(R; (R \otimes V)^{\otimes_R n}) \simeq \THH(R) \otimes V^{\otimes n}.\]
\end{cor}

\begin{proof}
  This follows from combining \Cref{lem:polygonic-sphere} with the symmetric monoidality of polygonic THH:
  \[\THH(R \otimes_{\Ss} \Ss; (R \otimes_{\Ss} V)^{\otimes n}) \simeq \THH(R) \otimes \THH(\Ss; V^{\otimes n}) \simeq \THH(R) \otimes V^{\otimes n}.\]
\end{proof}

In the remainder of this section, we compute the homotopy groups of
\[\gr^* _{\mot} \left((\THH(\Zp) \otimes \Ss^{\ell \rho_n})^{hC_{p^n}}\right)/p.\]
using the spectral sequence coming from the Nygaard filtration.
To do this, let us first rewrite these groups as the cohomology of a quasicoherent sheaf on $\Zp^{\Nyg}$.

\begin{prop}
  There are natural equivalences
  \[\gr^i _{\mot} (\THH(\Zp) \otimes \Ss^{\ell\rho_n})^{h C_{p^n}} \simeq R\Gamma(\Zp^{\Nyg}; \left(\O/(a^* \O(-1) \otimes \mathcal{I}_n)\right) \{i\} \otimes a^* \O(\ell p^n) \otimes \bigotimes_{i=0} ^{p^n-1} \mathcal{I}_{v_p(i)}^{-\ell}) [2i].\]
  Here, we let $\mathcal{I} \to \O$ denote the Hodge--Tate divisor on $\Zp^\prism$, let $a^* \O(-1) \to \O$ denote the Nygaard divisor on $\Zp^{\Nyg}$, and let $\mathcal{I}_n \to \O$ on $\Zp^{\Nyg}$ denote $(F')^* (\mathcal{I} \otimes \cdots \otimes  (F^{n-1})^* \mathcal{I}) \to \O$.
\end{prop}

\begin{proof}
%To begin, let $\Zp^{\Nyg, \wedge}$ denote the Nygaard completion of $\Zp^{\Nyg}$, i.e. the completion at the vanishing locus of $a : \Zp^{\Nyg} \to \A^1 / \G_m$.
To begin, we recall that we have
\[\gr^i _{\mot} \THH(\Zp)^{hS^1} \simeq R\Gamma(\Zp^{\Nyg}; \O\{i\}) [2i].\]
A priori, the left hand side is the Nygaard completion of the right hand side, but the absolute prismatic cohomology of $\Zp$ is already Nygaard-complete.

  Recall that the divisors $(F')^* \mathcal{I} \to \O$ and $a^* \O(-1) \to \O$ on $\Zp^{\Nyg}$, correspond, after pulling back to $R^{\Nyg}$ for $R$ a qrsp ring, to the ideals generated by the prismatic element $d$ and $t$, respectively.
It therefore follows from \cite[Lemma 3.1]{Noah} and quasisyntomic descent that there is an isomorphism:
\[\gr^i _{\mot} \THH(\Zp)^{hC_{p^n}} \simeq R\Gamma(\Zp^{\Nyg}; \left(\O/a^* \O(-1) \otimes (F')^* \mathcal{I} \otimes \cdots \otimes (F')^* (F^{n-1})^* \mathcal{I}\right)\{i\}) [2i],\]
i.e. an isomorphism
\[\gr^i _{\mot} \THH(\Zp)^{hC_{p^n}} \simeq R\Gamma(\Zp^{\Nyg}; \left(\O/a^* \O(-1) \otimes \mathcal{I}_n\right)\{i\}) [2i].\]

To deal with the twisting, we note that if we let $\lambda$ denote the standard complex representation of $S^1$, then there are fiber sequences
\[\Ss^{-\lambda^k} \to \Ss^0 \to \mathbb{D}(\Sigma^\infty_+ S^1 / C_k),\]
which give rise to fiber sequences
\[(\THH(R) \otimes \Ss^{-\lambda^k})^{h S^1} \to \THH(R)^{h S^1} \to \THH(R)^{h C_k} \simeq \THH(R)^{h C_{p^{v_p(k)}}}.\]

It follows that we have
  \[\gr^i _{\mot} (\THH(\Zp) \otimes \Ss^{-\lambda^k})^{hS^1} \simeq R\Gamma(\Zp^{\Nyg}; \O\{i\} \otimes a^* \O(-1) \otimes \mathcal{I}_{v_p (k)}) [2i],\]
from which we may deduce that (cf. \cite[Lemma 4.3]{Noah})
\[\gr^i _{\mot} (\THH(\Zp) \otimes \Ss^{-\lambda^k})^{h C_{p^n}} \simeq R\Gamma(\Zp^{\Nyg}; \left(\O/a^* \O(-1) \otimes \mathcal{I}_n\right) \{i\} \otimes a^* \O(-1) \otimes \mathcal{I}_{v_p(k)}) [2i].\]

Finally, writing $\rho_n$ as the restriction of $\bigoplus_{i=0}^{p^n-1} \lambda^i$ from $S^1$ to $C_{p^n}$, we conclude that 
\[\gr^i _{\mot} (\THH(\Zp) \otimes \Ss^{\ell\rho_n})^{h C_{p^n}} \simeq R\Gamma(\Zp^{\Nyg}; \left(\O/a^* \O(-1) \otimes \mathcal{I}_n\right) \{i\} \otimes a^* \O(\ell p^n) \otimes \bigotimes_{i=0} ^{p^n-1} \mathcal{I}_{v_p(i)}^{-\ell}) [2i].\]
\end{proof}

Now let us discuss spectral sequences.
\begin{cnstr}
Given a complex of quasicoherent sheaves $\mathcal{F}$ on $\Zp ^{\Nyg}$, there are spectral sequences associated to the Nygaard filtration of the form:
%\pi_* R\Gamma((\Zp ^{\Nyg}) _{a=0} ; F_{a=0}) [t] \cong 
\[\pi_* R\Gamma((\Zp ^{\Nyg}) _{a=0} ; \mathcal{F}_{a=0} \otimes (\bigoplus_{i=0} ^\infty a^* \O(-i))) \Rightarrow \pi_* R\Gamma(\Zp ^{\Nyg} ; \mathcal{F})^{\wedge} _{a}\]
and
%\pi_* R\Gamma((\Zp ^{\Nyg}) _{a=0} ; F_{a=0}) [t^{\pm 1}] \cong 
\[\pi_* R\Gamma((\Zp ^{\Nyg}) _{a=0} ; \mathcal{F}_{a=0} \otimes (\bigoplus_{i=-\infty} ^\infty a^* \O(-i))) \Rightarrow \pi_* R\Gamma(\Zp ^{\Nyg} ; \mathcal{F})^{\wedge} _{a}[a^{-1}].\]
(Indeed, to see this we push forward to $\mathbb{A}^1 / \G_m$ and use the correspondence between complexes on this stack and filtered complexes.)
These are algebraic versions of the $S^1$-homotopy fixed points and $S^1$-Tate spectral sequences on the level of $\THH$.
It follows from the definitions that the former is a truncation of the latter.
We refer to both of these spectral sequences as the \emph{Nygaard spectral sequence}.
\end{cnstr}

Next, we reduce modulo $p$ and compute these spectral sequences in the cases relevant to
$\gr^i _{\mot} (\THH(\Zp) \otimes \Ss^{\ell\rho_n})^{h C_{p^n}}.$

We start with the untwisted case, when $\ell = 0$.
By \cite[Construction 2.4]{BM}, we have the following starting point.

\begin{ntn}
  In the remainder of this section, given a complex of quasicoherent sheaves $\mathcal{F}$ on $\Z_p ^{\Nyg}$ (or $(\Z_p ^{\Nyg})_{p=0}$), we will say that an element of $H^s (X; F \otimes \mathcal{O} \{t\})$ lies in bidegree $(2t-s,s)$.
  Moreover, in the context of the Nygaard spectral sequence
  \[\pi_* R\Gamma((\Zp ^{\Nyg}) _{a=0} ; \mathcal{F}_{a=0} \otimes \mathcal{O} \{t\}_{a=0} \otimes (\bigoplus_{i=0} ^\infty a^* \O(-i))) \Rightarrow \pi_* R\Gamma(\Zp ^{\Nyg} ; \mathcal{F} \otimes \mathcal{O}\{t\})^{\wedge} _{a},\]
  we say that an element of 
  \[\pi_{-s} R\Gamma((\Zp ^{\Nyg}) _{a=0} ; \mathcal{F}_{a=0} \otimes \mathcal{O} \{t\}_{a=0} \otimes  a^* \O(-i))\]
  has tridegree $(2t-s, s, i)$.
\end{ntn}

\begin{prop} \label{prop:e2-untwisted}
  There is an isomorphism
  \[H^s ((\Zp^{\Nyg}) _{p=a=0} ; \O\{t\}_{p=a=0}) \cong \F_p [\mu] \otimes \Lambda(\lambda_1),\]
  where $\abs{\mu} = (2p,0)$ and $\abs{\lambda_1} = (2p-1, 1)$.
  %$\mu \in H^0 ((\Zp^{\Nyg}) _{p=a=0} ; \O\{p\}_{p=a=0})$ and $\lambda_1 \in H^{1} ((\Zp^{\Nyg}) _{p=a=0} ; \O\{p\}_{p=a=0})$.
\end{prop}

We will also need the following lemma.

\begin{lem} \label{lem:pull-trivial}
  Let $\mathcal{L}$ denote a line bundle over $\Zp^\prism$.
  Then the pullback of $\mathcal{L}$ along the composite
  \[(\Zp ^{\Nyg}) _{a=0} \hookrightarrow \Zp ^{\Nyg} \xrightarrow{F'} \Zp ^\prism \]
  is trivial.
  In particular, we have
  \[\O\{1\} \cong a^* \O(-1)\]
  after pullback to $(\Zp ^{\Nyg}) _{a=0}$.
\end{lem}

\begin{proof}
  By the proof of \cite[Proposition 5.3.7]{fgauge}, we have a commutative square
  \begin{center}
    \begin{tikzcd}
      (\Zp ^{\Nyg}) _{a=0} \ar[r] \ar[d] & \Zp^{\Nyg} \ar[d, "F'"] \\
      \mathrm{Spf} \Zp \ar[r,"dR"] & \Zp^{\prism},
    \end{tikzcd}
  \end{center}
  where the bottom horizontal map corresponds to the Cartier--Witt divisor $p : W(\Zp) \to W(\Zp)$.
  The lemma therefore follows from the fact that any line bundle over $\mathrm{Spf} \Zp$ is trivializable.

  The final statement follows from the first part of the lemma and the definition of $\O_{(\Z_p)^{\Nyg}}\{1\}$ in \Cref{dfn:NygBundles}.
\end{proof}

\begin{prop} \label{prop:sseq-untwisted}
  The Nygaard spectral sequence for the cofiber of $\mathcal{F} = \left(\O/a^* \O(-1) \otimes \mathcal{I}_n\right)\{*\}$ by $p$ takes the form
 \[\F_p [\mu, t^{\pm 1}] \otimes E(\lambda_1, u_n) \Rightarrow \pi_* \gr_{\mot}^* \THH(\Zp)^{tC_{p^n}}/p,\]
	with differentials generated under the Leibniz rule by:
	\begin{align*}
		d_{2(p+ \dots + p^{k+1})} (t^{p^k}) \doteq v_1 ^{p + \dots + p^k} t^{p^{k+1} + p^k} \lambda_1 \\
		d_{2(1+p+ \dots + p^n)} (u_n) \doteq v_1 ^{1+ \dots + p^{n-1}} t^{p^n},
	\end{align*}	
  for $0 \le k <n$, and the fact that $v_1=t\mu$ and $t^{p^{n}}$ are permanent cycles.
  In particular, the spectral sequence collapses at the $E_{2(1+\dots+p^n)+1}$-page.
  Here, we have $\abs{\mu} = (2p,0,0)$, $\abs{\lambda_1} = (2p-1, 1,0)$, $\abs{t} = (-2,0,1)$ and $\abs{u_n} = (-1,-1,0)$
\end{prop}

\begin{proof}
  To start, we note that by computation on the level of qrsp's, the mod $p$ reduction of the map
  $a^* \O(-1) \otimes \mathcal{I}_n \to \O$
  is detected in Nygaard filtration $1+p+\dots+p^n$ by a unit multiple of the element $\mu^{1+p+\dots+p^{n-1}}$.
  Indeed, mod $p$ we have $\mathcal{I}_n \cong (F')^* \mathcal{I}^{1+ \dots + p^{n-1}}$ as divisors, so it suffices to note that $\mu$ detects the prismatic element $d$ mod $p$ in Nygaard filtration $p$.

  In particular, this divisor produces the zero map upon pullback to $(\Zp ^{\Nyg})_{p=a=0}$,
  so that on this locus we have 
  \[\O/a^* \O(-1) \otimes \mathcal{I}_n \cong \O \oplus (a^* \O(-1) \otimes \mathcal{I}_n) [1]\]

  Next, we note that the pullback of $(F')^* \mathcal{I}$ to $(\Zp ^{\Nyg})_{a=0}$ is trivial by \Cref{lem:pull-trivial}, so that on $(\Zp ^{\Nyg})_{p=a=0}$ we further have
  \[\O/a^* \O(-1) \otimes \mathcal{I}_n \cong \O \oplus a^* \O(-1) [1] \cong \O \oplus \O\{1\} [1].\]

  The computation of the $E_2$-page now follows from \Cref{prop:e2-untwisted}, where $u_n$ corresponds to the $\O\{1\} [1]$ factor.
  The differential on $u_n$ therefore follows from our discussion above of the attaching map corresponding to this factor.

  The differentials on $t$ are mapped forward from the Nygaard spectral sequence for the mod $p$ reduction of $\mathcal{F} = \O\{*\}$, which was done in \cite[Propositions 6.32 \& 6.34]{LiuWang}.
\end{proof}

Next we study how twisting by a representation sphere affects this spectral sequence.
We start by noting that on $(\Zp ^{\Nyg}) _{p=0}$, there is an isomorphism $(F')^* \mathcal{I} \simeq \O\{1-p\} \otimes a^* \O(1-p)$.
Furthermore using the fact that $F^* \mathcal{L} \cong \mathcal{L}^p$ mod $p$, we find that $\mathcal{I}_n \cong (F')^* \mathcal{I} ^{1 + \dots + p^{n-1}} \cong \O\{1-p^n\} \otimes a^* \O(1-p^n)$.
Finally, we conclude that
\begin{align*}
  &a^* \O(\ell p^n) \otimes \bigotimes_{i=0} ^{p^n-1} \mathcal{I}_{v_p(i)}^{-\ell} \\
  &\cong \O\{\ell \sum_{i=0} ^{p^n-1} (p^{v_p(i)} - 1) \} \otimes a^* \O( \ell \sum_{i=0} ^{p^n-1} p^{v_p(i)} ) \\
  &\cong \O\{\ell n(p^n-p^{n-1}) - \ell p^n \} \otimes a^* \O( \ell n(p^n-p^{n-1}) ),
\end{align*}
since $\sum_{i=0} ^{p^n-1} p^{v_p (i)} = n(p^n-p^{n-1})$.

Using this, we will compute the twisted spectral sequence.

\begin{thm} \label{thm:sseq-twisted}
  The Nygaard spectral sequence for $\gr _{\mot} ^* (\THH(\Zp) \otimes \Ss^{\ell \rho_n})^{hC_{p^n}}$ is of the form:
  \[\sigma \epsilon^{(\ell p^n)} (\F_p [\mu, t ^{\pm 1}] \otimes E(\lambda_1, u_n)) \Rightarrow \pi_* \gr_{\mot}^* \THH(\Zp)^{tC_{p^n}}/p,\]
  where $\abs{\sigma \epsilon ^{(\ell p^n)}} = (2\ell p^n, 0, 0)$.
  The differentials are determined by the following two facts:
  \begin{itemize}
    \item It is a module over the untwisted Nygaard spectral sequence, which was computed in \Cref{prop:sseq-untwisted}.
    \item The differentials on $\sigma \epsilon^{(\ell p^n)}$ are determined by the differentials on $t^{-\ell n(p^n-p^{n-1})}$ in the spectral sequence of \Cref{prop:sseq-untwisted} via the formula
      \[d_{m} (\sigma \epsilon^{(\ell p^n)}) = \frac{d_{m} (t^{-\ell n(p^n-p^{n-1})})}{t^{-\ell n(p^n-p^{n-1})}} \sigma \epsilon^{(\ell p^n)}.\]
  \end{itemize}
\end{thm}

\begin{proof}
  This follows from combining the following two facts:
  \begin{itemize}
    \item Twisting a sheaf by $\O\{k\}$ simply shifts the Nygaard spectral sequence by $(-2k,0,0)$.
    \item Twisting a sheaf by $a^* \O (k)$ changes where in the Nygaard filtration we truncate: we truncate in filtrations $\geq -k$ instead of $\geq 0$. This has the effect of taking the truncation of the $t$-inverted spectral sequence spanned by $t^i$ where $i \geq -k$.
      This is isomorphic to the spectral sequence obtained by shifting the spectral sequence by $(2k, 0, 0)$ and twisting the differentials by the differential of $t^{-k}$.
  \end{itemize}
\end{proof}

Now, we use \Cref{thm:sseq-twisted} to explicitly compute the Nygaard spectral sequence for $\pi_* \gr_{\mot}^* (\THH(\Zp) \otimes \Ss^{\ell\rho_n})^{hC_{p^n}}/p$.
Our results are summarized in the following two propositions.

\begin{ntn}
  Given two nonnegative integers $a$ and $b$, we let $a \% b$ denote the unique integer between $0$ and $b-1$ which is equivalent to $a$ modulo $b$.
\end{ntn}

\begin{prop}\label{prop:hfppages}
	Let $n,\ell\geq0$. The $E_{2(p+\dots+p^{n})+1}$-page of the Nygaard spectral sequence
\[\sigma \epsilon^{(\ell p^n)} \F_p [\mu, t] \otimes E(\lambda_1, u_n) \Rightarrow \pi_* \gr_{\mot}^* (\THH(\Zp) \otimes \Ss^{\ell\rho_n})^{hC_{p^n}}/p\] is given as a $\FF_p[v_1]$-module as follows: 

\begin{align*}
  &\sigma \epsilon^{(\ell p^n)} E(u_n)\otimes \big( (t^{p^{n-1}(p-\ell\%p)},u^{p^{n-1}(\ell\%p)})\FF_p[t^{p^n},\mu^{p^n},v_1]/(v_1^{p^n}=t^{p^n}\mu^{p^n})\otimes E(\lambda_1)\\
    &\oplus \bigoplus_{k=1} ^{n-2} 
	\bigoplus_{i > p^{k+1} \vert v_p (i) = k}  \F_p [v_1] / (v_1^{p+ \dots+ p^k}) \{t^i \lambda_1\} \\
  &\oplus \bigoplus_{k=0} ^{n-2} (\bigoplus_{i=1} ^{p-1} \F_p [v_1] / (v_1 ^{p + \dots + p^k + p^k(p-i)}) \{t^{p^k i} \lambda_1\} \oplus \bigoplus_{j > 0 \vert v_p (j) = k} \F_p[v_1] /(v_1^{p+ \dots + p^{k+1}}) \{\mu^{j} \lambda_1\})
	\\
  &\oplus  \F_p [v_1] / (v_1^{p+ \dots+ p^{n-1}}) \{t^i \lambda_1 \vert i \geq p^{n}, v_p (i + \ell n p^{n-1}) = n-1 \} \\
  &\oplus \F_p [v_1] / (v_1 ^{p + \dots + p^{n-1} + p^{n-1}(p-i)}) \{t^{p^{n-1} i} \lambda_1 \vert 1 \leq i < p, v_p (p^{n-1} i +n \ell p^{n-1}) = n-1 \} \\
  &\oplus \F_p[v_1] /(v_1^{p+ \dots + p^{n}}) \{\mu^{j} \lambda_1 \vert j \geq 0, v_p ( j - \ell n p^{n-1}) = n-1\}
	\big)
\end{align*}

Here, $(t^{p^{n-1}(p-\ell\%p)},u^{p^{n-1}(\ell\%p)})\FF_p[t^{p^n},\mu^{p^n},v_1]/(v_1^{p^n}=t^{p^n}\mu^{p^n})$ refers to the sub-$\FF_p[t^{p^n},\mu^{p^n},v_1]/(v_1^{p^n}=t^{p^n}\mu^{p^n})$-module of $\FF_p[t,\mu]$ generated by $t^{p^{n-1}(p-\ell\%p)}$ and $u^{p^{n-1}(\ell\%p)}$.
\end{prop}

\begin{proof}
  The differentials of this spectral sequence are determined by \Cref{thm:sseq-twisted}.
	
  Note that we may ignore the element $u_n$ since it doesn't play a role in the differentials of the spectral sequence up to this page, so we omit it from the discussion. Next, up to the $E_{2(p+\dots+p^{n+1})}$-page, the spectral sequence agrees with that converging to $\pi_* \gr_{\mot} ^* (\THH(\Zp))^{hC_{p^{n-1}}}/p$, so by induction on $n$, we may assume that the $E_{2(p+\dots+p^{n+1})}$-page is given by $\sigma \epsilon^{(\ell p^n)}\FF_p[t^{p^{n-1}},\mu^{p^{n-1}},v_1]/(v_1^{p^{n-1}}=t^{p^{n-1}}\mu^{p^{n-1}})\otimes E(\lambda_1)$ plus the second and third lines of the answer stated in the proposition. This latter part plays no role in this page of the spectral sequence since it is $t$-torsion.

  So what remains is to understand homology of the $d_{2(p+ \dots + p^{n})}$-differential on
  \newline
  $\sigma \epsilon^{(\ell p^n)}\FF_p[t^{p^{n-1}},\mu^{p^{n-1}},v_1]/(v_1^{p^{n-1}}=t^{p^{n-1}}\mu^{p^{n-1}})\otimes E(\lambda_1)$.
  This differential goes from non-multiples of $\lambda_1$ to multiples of $\lambda_1$.
  The kernel is exactly the submodule
  \newline
  $\sigma \epsilon^{(\ell p^n)} (t^{p^{n-1}(p-\ell\%p)},u^{p^{n-1}(\ell\%p)})\FF_p[t^{p^n},\mu^{p^n},v_1]/(v_1^{p^n}=t^{p^n}\mu^{p^n})$, and the image is given by $t^{p^n}v_1^{p + \dots + p^{n-1}}\lambda_1$ times the complementary summand in $\sigma \epsilon^{(\ell p^n)}\FF_p[t^{p^{n-1}},\mu^{p^{n-1}},v_1]/(v_1^{p^{n-1}}=t^{p^{n-1}}\mu^{p^{n-1}})$.
  Note that $\sigma \epsilon^{(\ell p^n)} \lambda_1 (t^{p^{n-1}(p-\ell\%p)},u^{p^{n-1}(\ell\%p)})\FF_p[t^{p^n},\mu^{p^n},v_1]/(v_1^{p^n}=t^{p^n}\mu^{p^n})$ is not in the image, so together we find that $(t^{p^{n-1}(p-\ell\%p)},u^{p^{n-1}(\ell\%p)})\FF_p[t^{p^n},\mu^{p^n},v_1]/(v_1^{p^n}=t^{p^n}\mu^{p^n})\otimes E(\lambda_1)$ lies in the homology.
	
  The $\lambda_1$-multiples in the complementary summand are isomorphic to the free $\FF_p[v_1]$-module generated by terms of the form $\lambda_1 t^{p^{n-1}i}$ for $p \nmid i+\ell , i>0$, and $\lambda_1 \mu^{p^{n-1}i}$ for $p\nmid i-k,i\geq0$. When $i\geq p$, the terms of the form $\lambda_1 t^{p^{n-1}i}$ are $v_1^{p+ \dots +p^{n-1}}$-torsion because of differentials from $t^{p^{n-1}(i-p)}$. It follows that we get a summand $ \F_p [v_1] / (v_1^{p+ \dots+ p^{n-1}}) \{t^i \lambda_1 \vert i \geq p^{n}, v_p (i + \ell n p^{n-1}) = n-1 \}$ on the $E_{2(p+\dots+p^{n+1})+1}$-page from these terms.
	
	The terms $\lambda_1 t^{p^{n-1}i}$ for $p \nmid i+\ell ,0<i<p$ are similarly $v_1 ^{p + \dots + p^{n-1} + p^{n-1}(p-i)}$-torsion due to differentials from $\mu^{p^{n-1}(p-i)}$, so we obtain the following summand of the $E_{2(p+\dots+p^{n+1})+1}$-page:
	$$ \F_p [v_1] / (v_1 ^{p + \dots + p^{n-1} + p^{n-1}(p-i)}) \{t^{p^{n-1} i} \lambda_1 \vert 1 \leq i < p, v_p (p^{n-1} i +n \ell p^{n-1}) = n-1 \}. $$%\todo{see if it is good to simplify this fomula}
	
	Finally, the terms of the form  $\mu^{p^{n-1}i}\lambda_1$ for $p \nmid i-\ell$ are $v_1^{p+ \dots + p^n}$-torsion, because of differentials from $\mu^{p^{n-1}(i+p)}$. This contributes the last summand to the $E_{2(p+\dots+p^{n+1})+1}$-page, written below, since we have accounted for all elements of the spectral sequence:
	
	$$ \oplus \F_p[v_1] /(v_1^{p+ \dots + p^{n}}) \{\mu^{j} \lambda_1 \vert j \geq 0, v_p ( j - \ell n p^{n-1}) = n-1\}).$$
\end{proof}

Finally, we run the last nontrivial differential of the Nygaard spectral sequence to compute $\pi_* \gr_{\mot} ^* (\THH(\Zp) \otimes \Ss^{\ell \rho_n})^{hC_{p^n}}/p$:

\begin{prop}\label{prop:hfpsstwisted}
  For $n,\ell\geq 0$, the $E_\infty$-page of the Nygaard spectral sequence for $\pi_* \gr_{\mot}^* (\THH(\Zp) \otimes \Ss^{\ell \rho_n})^{hC_{p^n}}/p$ is given by:
\begin{align*}
  &\sigma\epsilon^{(\ell p^n)} E(\lambda_1) \otimes (\F_p [v_1] / (v_1 ^{1+ \dots + p^{n-1}}) \{t^{i} \vert i \geq p^n,  i \equiv -n \ell p^{n-1} \mod p^n\} \\
  \oplus &\F_p [v_1] / (v_1 ^{1 + \dots + p^{n-1} + (p^n-i)}) \{t^i \vert 0 < i < p^n, i \equiv -n \ell p^{n-1} \mod p^n \} \\
  \oplus &\F_p [v_1] / (v_1 ^{1+ \dots + p^{n}}) \{\mu^{j} \vert j \geq 0, j \equiv n \ell p^{n-1} \mod p^n \}) \\
  \oplus &E(u_n) \otimes ( \bigoplus_{k=1} ^{n-2} 
   \bigoplus_{i > p^{k+1} \vert v_p (i) = k}  \F_p [v_1] / (v_1^{p+ \dots+ p^k}) \{t^i \lambda_1\} \\
   \oplus &\bigoplus_{k=0} ^{n-2} (\bigoplus_{i=1} ^{p-1} \F_p [v_1] / (v_1 ^{p + \dots + p^k + p^k(p-i)}) \{t^{p^k i} \lambda_1\} \oplus \bigoplus_{j > 0 \vert v_p (j) = k} \F_p[v_1] /(v_1^{p+ \dots + p^{k+1}}) \{\mu^{j} \lambda_1\})
  ) \\
  \oplus &E(u_n) \otimes ( \F_p [v_1] / (v_1^{p+ \dots+ p^{n-1}}) \{t^i \lambda_1 \vert i \geq p^{n}, v_p (i + \ell n p^{n-1}) = n-1 \} \\
   \oplus &\F_p [v_1] / (v_1 ^{p + \dots + p^{n-1} + p^{n-1}(p-i)}) \{t^{p^{n-1} i} \lambda_1 \vert 1 \leq i < p, v_p (p^{n-1} i +n \ell p^{n-1}) = n-1 \} \\
   \oplus &\F_p[v_1] /(v_1^{p+ \dots + p^{n}}) \{\mu^{j} \lambda_1 \vert j \geq 0, v_p ( j - \ell n p^{n-1}) = n-1\})
\end{align*}

  Furthermore there are no extension problems when computing $\pi_* \gr_{\mot}^* (\THH(\Zp) \otimes \Ss^{\ell \rho_n})^{hC_{p^n}}/p$ as an $\F_p[v_1]$-module.
\end{prop}

\begin{proof}
	By \Cref{thm:sseq-twisted}, to understand the $E_{\infty}$-page of the spectral sequence, it suffices to compute the cohomology of the $d_{2(1+\dots+p^n)}$-differential on the $E_{2(1+\dots+p^n)}$-page, which is described by \Cref{prop:hfppages}.
  By \Cref{thm:sseq-twisted}, this differential is determined by the following formulae: $d_{2(1+\dots+p^n)} (\sigma \epsilon^{(\ell p^n)} u_n) = v_1 ^{1+ \dots + p^{n-1}} t^{p^n}$
  and the fact that it is zero on any term not containing $u_n$.
  Since $v_1^{1+\dots + p^{n-1}}t^{p^n}$ acts by $0$ on all but the first summand of the result of \Cref{prop:hfppages}, this differential is $0$ on all of these summands, which therefore survive to contribute all but the first three claimed summands. 
	
	On the remaining summand $$(t^{p^{n-1}(p-\ell\%p)},u^{p^{n-1}(\ell\%p)})\FF_p[t^{p^n},\mu^{p^n},v_1]/(v_1^{p^n}=t^{p^n}\mu^{p^n})\otimes E(\lambda_1)$$, multiplication by $v_1^{1+\dots + p^{n-1}}t^{p^n}$ is injective, so we get a summand $$(t^{p^{n-1}(p-\ell\%p)},u^{p^{n-1}(\ell\%p)})\FF_p[t^{p^n},\mu^{p^n},v_1]/(v_1^{p^n}=t^{p^n}\mu^{p^n}, v_1^{1+\dots + p^{n-1}}t^{p^n})\otimes E(\lambda_1).$$ 
	We can separate this into $\FF_p[v_1]$-module summands as the submodule generated by $t^i$ for $i\geq p^n$, the submodule generated by $t^i$ for $0<i<p^n$, and ths submodule generated by $\mu^j$ for $j\geq0$. A generator of the form $t^i$ for $i\geq p^n$ is $v_1^{1+ \dots + p^{n-1}}$-torsion because of the fact that $(v_1^{1+\dots + p^{n-1}}t^{p^n})t^{p^n-i}=0$, giving the first claimed summand of the proposition. A generator of the form $t^i$ for $0<i< p^n, i \equiv -n \ell p^{n-1} \mod p^n$ is $v_1 ^{1 + \dots + p^{n-1} + (p^n-i)}$-torsion since $v_1 ^{1 + \dots + p^{n-1} + (p^n-i)}t^{i} = (v_1^{1+\dots + p^{n-1}}t^{p^n})\mu^{p^n-i}$, giving the second claimed summand. A generator of the form $\mu^i$ is $v_1^{1+\dots + p^n}$-torsion since $\mu^iv_1^{1+\dots + p^n}= \mu^i(v_1^{1+\dots + p^{n-1}}t^{p^n})\mu^{p^n}$, giving the third summand.
	
  Next, we show that there are no $v_1$-extensions.
  In tridegree of the form $(x, \pm 1, b)$, every class is of the form $t^i\mu^j\lambda_1$ or $t^i\mu^ju_n$ respectively, and we will argue identically in either case, so we focus on the former. In particular, there is at most one class detected in each topological degree and filtration. If there is a potential $v_1=t\mu$ extension problem, it is because for a class $x$ detected by some $(t^i\mu^j\lambda_1)$, some power of $(t\mu)^k(t^i\mu^j\lambda_1)$ is zero on the $E_{\infty}$-page, but $v_1^kx$ is detected in higher filtration by $(t^p\mu)^l(t\mu)^k(t^i\mu^j\lambda_1)$ for some $l$. But we can see that on the $E_2$-page this is divisible by at least as many $t\mu$s as $t^i\mu^j\lambda_1$. It follows that by changing $x$ by a class detected in the associated graded by $(t^p\mu)^l(t^i\mu^j\lambda_1)$, $v_1^kx$ will be detected in higher filtration. Continuing this way, and using that the filtration is complete, there are no extension problems in these weights.
	
  A similar argument shows that in tridegrees of the form $(x, 0, a)$, where the classes are of the form $t^i\mu^j$ or $t^i\mu^j\lambda_1u_n$, there cannot be extension problems within each type of class, and from the latter to the former. 
	
  Let's rule out the case of extensions from the former to the latter.
  Suppose there exists a class $x$ detected by $t^i\mu^j$, such that $(t\mu)^k(t^i\mu^j)$ is zero, yet $v_1 ^k x$ is nonzero and detected by a class of the form $(t^p\mu)^l(t\mu)^{k-1}t^i\mu^j\lambda_1u_n$.
  Note that $l > 0$, since an extension is otherwise impossible for filtration reasons.
  Then, by changing $x$ by an element detected by $(t^p\mu)^{l-1}t^{p-1}t^i\mu^j\lambda_1u_n$ to get an $x'$, $v_1^kx$ is detected in higher filtration.
  Note that $(t^p\mu)^{l-1}t^{p-1}t^i\mu^j\lambda_1u_n$ must be a permanent cycle, because by inspection we see that if $v_1 y$ is a permanent cycle in this spectral sequence, so is $y$.

\end{proof}

\begin{prop}\label{prop:twistedtatess}
	Let $n,\ell\geq0$. 
  The $E_\infty$-page of the Nygaard spectral sequence for $$\pi_* \gr_{\mot} ^*(\THH(\Zp) \otimes \Ss^{\ell \rho_n})^{tC_{p^n}}/p$$ is given by:
\begin{align*}
  &\sigma \epsilon^{(\ell p^n)} \F_p [v_1, t^{\pm p^n}]/(v_1 ^{1+ \dots + p^{n-1}}) \{t^{-n \ell p^{n-1}}\} \otimes E(\lambda_1) \\
  &\oplus \left( E(u_n) \otimes \left( \bigoplus_{k=1} ^{n-2} \F_p [v_1] /(v_1 ^{p+ \dots + p^{k}}) \{t^i \lambda_1 \vert v_p(i) = k \} \right) \right) \\
  &\oplus E(u_n) \otimes \F_p [v_1] / (v_1 ^{p+\dots+p^{n-1}}) \{t^i \lambda_1 \vert v_p(i+\ell n p^{n-1}) = n-1\}.
\end{align*}

  There are no extension problems when computing $\pi_*\gr^*_{\mot} (\THH(\Zp) \otimes \Ss^{\ell \rho_n})^{tC_{p^n}}/p$ as an $\F_p[v_1]$-module.
\end{prop}

\begin{proof}
	The $E_{\infty}$-page of the spectral sequence can be read off from inverting $t$ in \Cref{prop:hfpsstwisted}. To see that no extension problems happen for a particular class, it suffices to prove this after multiplying by a sufficiently large power of $t$, for which it follows again by \Cref{prop:hfpsstwisted}, since the two spectral sequences agree in high enough filtrations.
\end{proof}

Now we want to get a grip on the Frobenius maps. To understand these, we begin by noting that
\[\varphi : R\Gamma((\Zp\langle \epsilon \rangle^{\Nyg})_{p=0}, \O\{*\}) \to R\Gamma((\Zp \langle \epsilon \rangle^{\prism})_{p=0}, \O\{*\})\]
is equivalent to the map inverting $\mu$ on the source
because $j_{HT}$, which induces $\varphi$, is the inclusion of the nonvanishing locus of $\mu$.

It therefore follows that 
\[\varphi: \gr^* _{\mot} \TC^{-} (\Zp \langle \epsilon \rangle) / p \to \gr^* _{\mot} \TP (\Zp \langle \epsilon \rangle) / p\]
may be identified with the map that inverts $\mu$ and Nygaard completes.
The same is true of 
\[\varphi : \gr_{\mot}^* (\THH(\Zp) \otimes \Ss^{k\rho_{n-1}})^{hC_{p^{n-1}}}/p \to \gr_{\mot}^* (\THH(\Zp) \otimes \Ss^{k\rho_{n}})^{tC_{p^{n}}}/p\]
as it is a graded summand of the former map.
Since the corresponding prismatic complexes are already Nygaard complete, this last Frobenius map may be identified with $\mu$-inversion.

%Note that we have, for $\mathcal{F} = \O\{i+k n(p^n-p^{n-1}) - \ell p^n \} \otimes a^* \O( k n(p^n-p^{n-1}) )$,
%\begin{align*}
  %\gr_{\mot}^i (\THH(\Zp) \otimes \Ss^{k\rho_n})^{tC_{p^n}}/p
  %&\simeq R\Gamma((\Zp ^{\Nyg})_{p=0} ; \mathcal{F})^{\wedge} _{a}[a^{-1}] \\
  %&\simeq R\Gamma((\Zp ^{\Nyg})_{p=0} ; \mathcal{F})^{\wedge}[a^{-1}] \\
  %&\simeq R\Gamma((\Zp ^{\prism})_{p=0} ; \mathcal{F})^{\wedge} \\
  %&\simeq R\Gamma((\Zp ^{\Nyg})_{p=0} ; \mathcal{F})^{\wedge}[\mu^{-1}],
%\end{align*}
%sice $j_{dR} : \Zp ^{\prism} \hookrightarrow \Zp ^{\Nyg}$ is the nonvanishing locus of $a$, and $j_{HT} : (\Zp ^{\prism})_{p=0} \hookrightarrow (\Zp ^{\Nyg})_{p=0}$ is the nonvanishing locus of $\mu$ after modding out by $p$.

As a consequence, we can also compute
$\pi_* \gr_{\mot}^* (\THH(\Zp) \otimes \Ss^{\ell \rho_n})^{tC_{p^n}}/p$
using the $\mu$-inverted Nygaard spectral sequence. The result is as follows:
%
%It is also possible to compute $\pi_* \gr_{\mot}(\THH(\Zp) \otimes \Ss^{\ell \rho_n})^{tC_{p^n}}/p$ by pulling back along $j_{HT}$ instead of $j_{dR}$.....using $j_{HT} ^* a^* \O(1) \cong I$.
%
%It is also possible to compute the Tate construction using the $\mu$-inverted homotopy fixed points spectral sequence, since the $S^1$-equivariant frobenius map $\phi_p:\THH(\ZZ_p)/p \to (\THH(\ZZ_p)/p)^{tC_p}$ realizes the target as $\THH(\ZZ_p)/p[\mu^{-1}]$ \cite[Corollary 1.5]{mathew2021k}, and $((\THH(\ZZ_p)/p)^{tC_p})^{hC_{p^{n-1}}} \cong (\THH(\ZZ_p)/p)^{tC_{p^n}}$ \cite[Lemma II.4.1]{nikolaus2018topological}. For degree reasons, this must also hold at the level of the associated graded of the motivic filtration.

\begin{prop}
  Let $n,\ell \geq0$. The $E_\infty$-page of the $\mu$-inverted Nygaard spectral sequence for $\pi_*\gr_{\mot}^* (\THH(\Zp) \otimes \Ss^{\ell\rho_n})^{tC_{p^n}}/p$ is given by:
\begin{align*}
  &\sigma \epsilon^{(\ell p^{n-1})} \F_p [v_1, \mu^{\pm p^{n-1}}]/(v_1 ^{1+ \dots + p^{n-1}}) \{\mu^{(n-1) \ell p^{n-2}}\} \otimes E(\lambda_1) \\
  &\oplus \left( E(u_{n-1}) \otimes \left( \bigoplus_{k=0} ^{n-3} \F_p [v_1] /(v_1 ^{p+ \dots + p^{k} + p^{k+1}}) \{\mu^j \lambda_1 \vert v_p(j) = k \} \right) \right) \\
  &\oplus E(u_{n-1}) \otimes \F_p [v_1] / (v_1 ^{p+\dots+p^{n-1}}) \{\mu^j \lambda_1 \vert v_p(j-\ell(n-1)p^{n-2}) = n-2\}.
\end{align*}

  There are no extension problems when computing $\pi_* \gr_{\mot}(\THH(\Zp) \otimes \Ss^{\ell \rho_n})^{tC_{p^n}}/p$ as an $\F_p[v_1]$-module.
\end{prop}

\begin{proof}
	The $E_{\infty}$-page of the spectral sequence can be read off from inverting $\mu$ in \Cref{prop:hfpsstwisted}. To see that no extension problems happen for a particular class, it suffices to prove this after multiplying by a sufficiently large power of $\mu$, for which it follows again by \Cref{prop:hfpsstwisted}, since the two spectral sequences agree up to higher and higher filtrations as the power of $\mu$ gets larger.
\end{proof}

\subsection{Working up to $v_1$-adic filtration}\label{sec:v1}

We begin with the following observation.

\begin{lem}\label{lemma:1d}
  Let $n,k\geq0$. The graded $\F_p$-vector space
  \[(\pi_* \gr_{\mot}^* (\THH(\Zp) \otimes \Ss^{\ell  \rho_n})^{tC_{p^n}} /p )/v_1\]
  is at most one-dimensional in each degree.

  Furthermore, the graded $\F_p$-vector space
  \[(\pi_* \gr_{\mot}^* (\THH(\Zp) \otimes \Ss^{\ell  \rho_n})^{hC_{p^n}} /p )/v_1\]
  is at most one-dimensional in each degree with the following exception:
  if $n$ is divisible by $p$, then degree $2\ell p^n$ is two-dimensional
  with given by classes represented in the Nygaard spectral sequence by $\sigma \epsilon^{(\ell  p^n)}$ and $u_n \lambda_1 t^{p-1} \sigma \epsilon^{(\ell  p^n)}$.
\end{lem}

\begin{proof}
  The lemma is easy for $n = 0$, so we assume $n \geq 1$.

  For the Tate fixed points, we use the $\mu$-inverted Nygaard spectral sequence.
  Then modulo $v_1$ all basis elements are of the form $u_{n-1} ^{e_1} \lambda_1 ^{e_2} \mu^j \sigma \epsilon^{(\ell p^n)}$, where $j$ is an integer and $e_1, e_2 \in \{0,1\}$.
  Moreover, we have $e_1 =0$ if $e_2 = 0$.

  Now, we have $\abs{u_{n-1} ^{e_1} \lambda_1 ^{e_2} \mu^j \sigma \epsilon ^{(\ell p^n)}} \equiv - e_1 - e_2 \mod 2p$, so that we may recover $e_1$ and $e_2$ from the degree.
  But, given a fixed $(e_1, e_2)$, it is clear that there is at most one class of the form $u_{n-1} ^{e_1} \lambda_1 ^{e_2} \mu^j \sigma \epsilon^{(\ell p^n)}$ in each degree, so we are done.

  Now we do the homotopy fixed points. We start by assuming that $n > 1$.
  In this case, modulo $v_1$ in odd degrees we have basis elements of the form
  \begin{align*}
    &\lambda_1 t^i \sigma \epsilon^{(\ell p^n)} \\
    &\lambda_1 \mu^j \sigma \epsilon^{(\ell p^n)},
  \end{align*}
  and it is clear that there is at most one of these in each odd degree.
  In even degrees, we have basis elements of the form
  \begin{align*}
    &t^{pi} \sigma \epsilon^{(\ell p^n)} \\
    &\mu^{j} \sigma \epsilon^{(\ell p^n)} \\
    &u_n \lambda_1 t^{pi} \sigma \epsilon^{(\ell p^n)} \\
    &u_n \lambda_1 \mu^{j} \sigma \epsilon^{(\ell p^n)} \\
    &u_n \lambda_1 t^{i} \sigma \epsilon^{(\ell p^n)},  i \in \{1,\dots,p-1\}.
  \end{align*}
  (That we can put the $p$ in $t^{pi}$ is where we use $n > 1$.)
  Considering degrees modulo $2p$, we see that the degrees of the first two lines do not intersect those of the second two lines.
  It is also clear that the first two and last three lines do not contain more than one elements of the same degree.
  We therefore only need to worry about the first two and last line. Since the degrees of the last line are $2\ell p^n +2p-4, \dots, 2\ell p^n$, we see that the only possible point of intersection occurs in degree $2\ell p^n$, where we can have both
  $\sigma \epsilon^{(\ell  p^n)}$ and $u_n \lambda_1 t^{p-1} \sigma \epsilon^{(\ell  p^n)}$
  if $p$ divides $n$.

  Finally, we do the case $n=1$.
  Again, the case of odd degrees is easy.
  Our even degree basis elements are:
  \begin{align*}
    &t^{i} \sigma \epsilon^{(\ell p)} \\
    &\mu^{j} \sigma \epsilon^{(\ell p)} \\
    &u_1 \lambda_1 \mu^{j} \sigma \epsilon^{(\ell p)} \\
    &u_1 \lambda_1 t^{i} \sigma \epsilon^{(\ell p)}, i \in \{1,\dots,p-1\}.
  \end{align*}
  Again, it is clear that the first two and last two lines do not have any repeated degrees.
  The first and third line don't because $\abs{t^i}$ is non-positive and $\abs{u_1 \lambda_1 \mu^j}$ is positive.
  The second and third line don't by considering the reduction modulo $2p$.
  The final line again lies in degrees $2\ell p +2p-4, \dots, 2\ell p$, which could only intersect the first two in $2\ell p$ in the same way as the $n > 1$ case.
  But this is not possible because $p$ does not divide $n=1$.
\end{proof}

\begin{ntn}
	Let $n,k\geq0$. We use $\Fil^*_{v_1}$ to denote the $v_1$-adic filtration on $\pi_*$, and $\gr^*_{v_1}$ to denote the associated graded of this filtration. Then by \Cref{lemma:1d}, the generators (as an $\F_p[v_1]$-module) of the $E_\infty$-page of the Nygaard spectral sequence lift uniquely to elements in $\gr ^*_{v_1} (\pi_*\gr_{\mot} ^* (\THH(\Zp) \otimes \Ss^{\ell  \rho_n})^{h C_{p^n}}/p)$, with the exception of $\sigma \epsilon^{(\ell p^n)}$ when $n$ is divisible by $p$, in which case the indeterminacy is generated by $u_n \lambda_1 t^{p-1} \sigma \epsilon^{(\ell p^n)}$.
  We use this to unambiguously regard these symbols as elements of $\gr^*_{v_1}( \pi_* \gr_{\mot} ^* (\THH(\Zp) \otimes \Ss^{\ell  \rho_n})^{h C_{p^n}}/p)$, with the exception of $\sigma \epsilon^{(\ell p^n)}$, for which we fix once and for all a choice of lift.

  We do similarly for the Tate construction and its Nygaard spectral sequence, without any exceptions.
\end{ntn}

Next we compute the maps $\mathrm{can}$ and $\varphi$ on the level of the $v_1$-adic associated graded of their homotopy groups. Recall that $\mathrm{can}$ is induced by $j_{dR}$ and $\varphi$ is induced by $j_{HT}$.

\begin{prop}\label{prop:grv1canphi}
  Let $n,k\geq0$. Equip $\pi_*\gr_{\mot} (\THH(\Zp) \otimes \Ss^{2\ell \rho_n})^{h C_{p^n}}/p$ and $\pi_*\gr_{\mot} (\THH(\Zp) \otimes \Ss^{2\ell \rho_n})^{t C_{p^n}}/p$ with the $v_1$-adic filtration.
  Then $can$ and $\varphi$ induce the maps
  \begin{align*}
    \can : \gr^* _{v_1} (\pi_* \gr_{\mot}(\THH(\Zp) \otimes \Ss^{2\ell \rho_n})^{h C_{p^n}}/p) &\to \gr^* _{v_1} (\pi_* \gr_{\mot}(\THH(\Zp) \otimes \Ss^{2\ell \rho_n})^{t C_{p^n}}/p) \\
    v_1^a \sigma \epsilon^{(p^n \ell)} t^i \lambda_1 ^{\epsilon_1} u_n ^{\epsilon_2} &\mapsto v_1^a \sigma \epsilon^{(p^n \ell)} t^i \lambda_1 ^{\epsilon_1} u_n ^{\epsilon_2} \\
    v_1^a \sigma \epsilon^{(p^n \ell)} \mu^j \lambda_1 ^{\epsilon_1} u_n ^{\epsilon_2} &\mapsto 0, j>0
  \end{align*}
  and
  \begin{align*}
    \varphi : \gr^* _{v_1}( \pi_* \gr_{\mot}(\THH(\Zp) \otimes \Ss^{2\ell \rho_n})^{h C_{p^n}}/p) &\to \gr^* _{v_1}( \pi_* \gr_{\mot}(\THH(\Zp) \otimes \Ss^{2k \rho_{n+1}})^{t C_{p^{n+1}}}/p) \\
    v_1^a \sigma \epsilon^{(p^n \ell)} t^i \lambda_1 ^{\epsilon_1} u_n ^{\epsilon_2} &\mapsto 0, i>0 \\
    v_1^a \sigma \epsilon^{(p^n \ell)} \mu^j \lambda_1 ^{\epsilon_1} u_n ^{\epsilon_2} &\stackrel{\cdot}{\mapsto} v_1^a \sigma \epsilon^{(p^{n+1} \ell)} t^{p^n \ell (p-1) - pj} \lambda_1 ^{\epsilon_1} u_{n+1} ^{\epsilon_2}.
  \end{align*}
\end{prop}

\begin{proof}
	Using \Cref{lemma:1d}, it is enough to check that the maps are as claimed at the level of an associated graded, with the exception of the class $\sigma \epsilon^{(\ell p^n)}$ in the case $p|n$, for which there is a potential indeterminacy of $u_n\lambda_1t^{p-1}\sigma\epsilon^{(\ell p^n)}$. On this class, both $\can$ and $\varphi$ are zero, so $\varphi$ and $\can$ of $\epsilon^{(\ell p^n)}$ are determined by the map at the level of an associated graded.
	
	For $\can$, the map of spectral sequences is as claimed since the map of spectral sequence is given by inverting $t$, so the map is given as claimed on the associated graded of the Nygaard spectral sequence.

  	If the classes in the Tate construction are named using the $\mu$-inverted Nygaard spectral sequence, then the map $\varphi$ is simply given by inverting $\mu$. To obtain the proposition, we need to transition from $\mu$-inverted names to Tate names, which can be done up to a unit in $\FF_p^{\times}$ using the fact that there is at most one class in each degree.
\end{proof}

\subsection{The computation}\label{sec:computation}

In this section, we tackle the computation.

Consider
$\varphi - \can : \Fil^* _{v_1} \prod_{i=0} ^n \pi_* (\gr_{\mot}^* (\THH(\Zp) \otimes \Ss^{\ell \rho_i})^{hC_{p^i}}/p) \to \Fil^* _{v_1} \prod_{i=1} ^n\pi_*( \gr_{\mot}^* (\THH(\Zp) \otimes \Ss^{\ell \rho_i})^{tC_{p^i}}/p)$.

First, we note that the $v_1$-adic filtrations here are finite, since each of the spectra involved has bounded $v_1$-torsion.

\begin{lem}\label{lemma:surjphican}
	Let $n,\ell\geq0$. The map $$\Fil^* _{v_1} \prod_{i=0} ^n \pi_* (\gr_{\mot}^* (\THH(\Zp) \otimes \Ss^{\ell \rho_i})^{hC_{p^i}}/p) \to \Fil^* _{v_1} \prod_{i=1} ^n\pi_*( \gr_{\mot}^* (\THH(\Zp) \otimes \Ss^{\ell \rho_i})^{tC_{p^i}}/p)$$ is surjective.
\end{lem}

\begin{proof}
	Since the $v_1$-adic filtration is finite, it will suffice to prove surjectivity on the $v_1$-adic associated graded.
	We use the formulas in \Cref{prop:grv1canphi}. The image of $\can$ in the ($v_1$-associated graded of the) $C_{p^i}$-Tate construction is spanned by those basis elements which have a nonnegative power of $t$. The image of $\varphi$ in the $C_{p^i}$-Tate construction in particular contains every basis element with a negative power of $t$. Note also that $\varphi-\can$ is either $\varphi$ or $-\can$ unless the power of $\mu$ and $t$ is $0$. Let us write the source as the sum of three terms $\oplus_{0}^n(S_{+,i}\oplus S_{0,i} \oplus S_{-,i})$ by having $S_{+,i}$ be the summand of $\gr^*_{v_1}\pi_* \gr_{\mot}(\THH(\Zp) \otimes \Ss^{\ell \rho_i})^{hC_{p^i}}/p$ where the power of $\mu$ is positive, $S_{-,i}$ be the summand where the power of $t$ is positive, and $S_{0,i}$ be the summand where both powers are $0$. Similarly, let $T_{+,i}$ be the summand of the target where the power of $t$ is negative, $T_{0,i}$ be the summand where the power of $t$ is zero, and $T_{-,i}$ the summand where it ie positive. 
	
	Then we can put a filtration on the source and target where we filter the target so that $T_{+,i},T_{-,i},T_{0,i}$ are in filtration $i$, and filter the source so that $S_{-,i} , S_{+,i-1}, S_{0,i-1}$ are in filtration $i$. Then $\varphi-\can$ preserves this filtration, and on the associated graded, $S_{-,i}$ surjects onto $T_{-,i}$ via $-\can$, and $S_{+,i-1}\oplus S_{0,i-1}$ surjects onto $(\pi_*\gr_{\mot}(\THH(\Zp) \otimes \Ss^{\ell \rho_i})^{tC_{p^i}}/p)/T_{-,i}$ via $\varphi$, so this shows surjectivity.
\end{proof}

%We claim that from the formulas in \Cref{prop:grv1canphi} that on $\gr^* _{v_1}$, this map is surjective. Indeed, we can see from those formulas that 
%

As a consequence of \Cref{lemma:surjphican}, only the kernel of $\varphi-\can$ contributes to $\pi_* \gr^* _{\mot} \TR^{[n]} (\Zp; \Sigma^{2\ell} \Zp) / p$.

\begin{cnstr}
	We define a filtration on $\pi_* \gr_{\mot} ^* \TR^{[n]} (\Zp; \Sigma^{2\ell} \Zp) / p$ via
	
	$$\Fil^i \pi_* \gr_{\mot}^* \TR^{[n]} (\Zp; \Sigma^{2\ell} \Zp) / p \coloneqq \ker(\Fil^i _{v_1} (\varphi-\can)).$$ and similarly for $\pi_* \gr_{\mot} ^* \TR (\Zp; \Sigma^{2\ell} \Zp) / p$.
	
	We use $\gr^i \pi_* \gr_{\mot}^* \TR^{[n]} (\Zp; \Sigma^{2\ell} \Zp) / p$ to denote the associated graded of this filtration, and similarly for $\TR$.
\end{cnstr}

%Then a further consequence of \Cref{lemma:surjphican} is that
%$\gr^* \pi_* \gr_{\mot}\TR^{[n]} (\Zp; \Sigma^{2\ell} \Zp) / p \cong \ker(\gr^* _{v_1} (\varphi-\can))$.

%Finally, we use the fact that $\gr_{\mot}^* \TR (\Zp; \Sigma^{2\ell} \Zp) / p \to \gr_{\mot}^* \TR^{[n]} (\Zp; \Sigma^{2\ell} \Zp) / p$ is an equivalence in increasing degrees to transfer these facts to $\gr_{\mot}^* \TR (\Zp; \Sigma^{2\ell} \Zp) / p$.

Now, we do some legwork necessary to state the output of our main computation.

\begin{dfn} \label{dfn:tr-family-generators}
  Below, we list some mutually linearly independent families of $\gr_{v_1} ^* \prod_{i=0} ^\infty \pi_* \gr_{\mot}(\THH(\Zp) \otimes \Ss^{\ell \rho_i})^{hC_{p^i}}/p$ that lie in the kernel of $\varphi-\can$. Eventually, we will show that they form a basis for $\gr^* \pi_* \gr_{\mot}\TR (\Zp; \Sigma^{2\ell} \Zp) / p$ in a certain sense.

\begin{align}
  &\lambda_1 ^e (\sigma \epsilon^{(p^n \ell)} \mu^j \stackrel{\cdot}{+} \sigma \epsilon^{(p^{n+1} \ell)} t^{p^n\ell(p-1) -pj}) \tag{$1_{n,\ell}$} \\
  &\text{where }
  e \in \{0,1\}, 0 < j \leq p^{n-1} \ell (p-1) - p^n, j \equiv n \ell p^{n-1} \mod p^n \notag
\end{align}

\begin{align}
  &\lambda_1 ^e (\sigma \epsilon^{(p^n \ell)} \mu^j \stackrel{\cdot}{+} \sigma \epsilon^{(p^{n+1} \ell)} t^{p^n\ell(p-1) -pj} \stackrel{\cdot}{+} \delta^{j}_{p^{n-1}\ell(p-1)} \sigma \epsilon^{(p^{n+2} \ell)} t^{p^{n+1} \ell (p-1)}) \tag{$2_{n,\ell}$} \\
  &\text{where }
  e \in \{0,1\}, p^{n-1} \ell (p-1) - p^n < j \leq p^{n-1} \ell (p-1), j > 0, j \equiv n \ell p^{n-1} \mod p^n \notag
\end{align}
Here $\delta^j_l$ denotes the Kronecker delta function.

\begin{align}
  &\sigma \epsilon^{(p^n \ell)} t^i \lambda_1 u_n ^e \tag{$3_{n,\ell}$} \\
  &\text{where }
  e \in \{0,1\}, 0 < i < p, v_p (i + \ell n p^{n-1}) = 0 \notag
\end{align}

In the next two families, we let $1 \leq r \leq n$.

\begin{align}
  &\sigma \epsilon^{(p^n \ell)} \mu^j \lambda_1 u_n^e \stackrel{\cdot}{+} \sigma \epsilon^{(p^{n+1} \ell)} t^{p^n\ell(p-1) -pj} \lambda_1 u_{n+1} ^e \tag{$4_{n,\ell,r}$} \\
  &\text{where }
  e \in \{0,1\}, j > 0, v_p (j-\ell n p^{n-1}) = r-1, j \leq p^{n-1} \ell (p-1) - p^r \notag
\end{align}

\begin{align}
  &\sigma \epsilon^{(p^n \ell)} \mu^j \lambda_1 u_n^e \stackrel{\cdot}{+} \sigma \epsilon^{(p^{n+1} \ell)} t^{p^n\ell(p-1) -pj} \lambda_1 u_{n+1} ^e \stackrel{\cdot}{+} \delta^j_{p^{n-1}\ell(p-1)} \sigma \epsilon^{(p^{n+2} \ell)} t^{p^{n+1} k(p-1)} \lambda_1 u_{n+2} ^e \tag{$5_{n,\ell,r}$} \\
  &\text{where }
  e \in \{0,1\}, j > 0, v_p (j-\ell n p^{n-1}) = r-1, p^{n-1} \ell (p-1) - p^r < j \leq p^{n-1} \ell (p-1) \notag
\end{align}

The constants in the $\dot{+}$s are uniquely determined by the condition that these classes are in the kernel of $\varphi-\can$ (see \Cref{lemma:trfamilies}).
\end{dfn}

\begin{dfn} \label{dfn:trfamilies}
We let $A_{n,\ell}$, $B_{n,\ell}$, $C_{n,\ell}$, $D_{n,\ell,r}$ and $E_{n,\ell,r}$ denote the sub-$\F_p[v_1]$-modules of $\prod_{i=0} ^\infty \gr_* ^{v_1} \pi_*\gr_{\mot} (\THH(\Zp) \otimes \Ss^{\ell \rho_i})^{hC_{p^i}} /p$ generated by the families $(1_{n,\ell}), (2_{n,\ell}), (3_{n,\ell}), (4_{n,\ell,r})$ and $(5_{n,\ell,r})$, respectively.
\end{dfn}

\begin{lem}\label{lemma:trfamilies}
	$A_{n,\ell}$, $B_{n,\ell}$, $C_{n,\ell}$, $D_{n,\ell,r}$ and $E_{n,\ell,r}$ are in the kernel of $\varphi-\can$, and their elements are linearly independent in the sense that $A_{n,\ell}+ B_{n,\ell}+ C_{n,\ell}+ D_{n,\ell,r}+ E_{n,\ell,r}$ is the direct sum of the cyclic $\FF_p[v_1]$-submodules generated by each element of $(1_{n,\ell}), (2_{n,\ell}), (3_{n,\ell}), (4_{n,\ell,r})$ and $(5_{n,\ell,r})$. Moreover the $v_1$-torsion order of the generators is listed below:
	
	\begin{enumerate}
		\item[$(1_{n,\ell})$:]
		$\lambda_1 ^e (\sigma \epsilon^{(p^n \ell)} \mu^j \stackrel{\cdot}{+} \sigma \epsilon^{(p^{n+1} \ell)} t^{p^n\ell(p-1) -pj})$ is $v_1^{1+ \dots + p^n}$-torsion.
		\item[$(2_{n,\ell})$:]
		$\lambda_1 ^e (\sigma \epsilon^{(p^n \ell)} \mu^j \stackrel{\cdot}{+} \sigma \epsilon^{(p^{n+1} \ell)} t^{p^n\ell(p-1) -pj} \stackrel{\cdot}{+} \delta^{j}_{p^{n-1}\ell(p-1)} \sigma \epsilon^{(p^{n+2} \ell)} t^{p^{n+1} \ell (p-1)})$ is
    \newline
      $v_1^{1+\dots+p^n+(p^{n+1} - (p^n\ell(p-1)-pj))}$-torsion unless $j=p^n\ell(p-1)$ and $k= 1$, in which case it is $v_1^{1+\dots+p^{n+1} + p^{n+1}}$-torsion.
		\item[$(3_{n,\ell})$:]
		$\sigma \epsilon^{(p^n \ell)} t^i \lambda_1 u_n ^e$ is $v_1^{p-i}$-torsion.
		\item[$(4_{n,\ell,r})$:]
		$\sigma \epsilon^{(p^n \ell)} \mu^j \lambda_1 u_n^e \stackrel{\cdot}{+} \sigma \epsilon^{(p^{n+1} \ell)} t^{p^n\ell(p-1) -pj} \lambda_1 u_{n+1} ^e$ is $v_1^{1+ \dots + p^{r}}$-torsion.
		\item[$(5_{n,\ell,r})$:] 
		$\sigma \epsilon^{(p^n \ell)} \mu^j \lambda_1 u_n^e \stackrel{\cdot}{+} \sigma \epsilon^{(p^{n+1} \ell)} t^{p^n\ell(p-1) -pj} \lambda_1 u_{n+1} ^e \stackrel{\cdot}{+} \delta^j_{p^{n-1}\ell(p-1)} \sigma \epsilon^{(p^{n+2} \ell)} t^{p^{n+1} k(p-1)} \lambda_1 u_{n+2} ^e$ is $v_1^{1+ \dots + p^{r} + (p^{r+1} - (p^n\ell(p-1)-pj))}$-torsion unless $n=r, j=p^{n-1}\ell(p-1),k=1$ in which case it is $v_1^{1+ \dots + p^{n+1} + p^{n+1}}$-torsion.
	\end{enumerate}
\end{lem}

\begin{proof}
	We first prove that the classes are in $\gr_{\mot}^* \TR (\Zp; \Sigma^{2\ell} \Zp) / p$ and check their $v_1$-torsion orders. In general, doing this is a straightforward application of \Cref{prop:grv1canphi} and \Cref{prop:hfpsstwisted}, where the first term in each sum has $\can$ vanishing, the last term has $\varphi$ vanishing, and $\varphi$ of each term is $\can$ of the next one. The $v_1$-torsion order can be read off from the maximum of the $v_1$-torsion order of each term in \Cref{prop:hfpsstwisted}. Therefore, we will just indicate where the each term of each class appears in the description of the $E_{\infty}$-page of the Nygaard spectral sequence of \Cref{prop:hfpsstwisted}.
	
	For family $(1_{n,\ell})$, the first term in the family appears in the third line of \Cref{prop:hfpsstwisted}, and the second term appears in the first line.
	
	For family $(2_{n,\ell})$, the first term in the family appears in the first line of \Cref{prop:hfpsstwisted}, and the second term appears in the second line unless $p^{n-1}\ell(p-1)=j$. In this case, the second term appears on the third line, and the third term appears on the first line if $k\geq 2$ and the second line otherwise. Note that the case $k=0$ doesn't happen because $j>0$.
	
	For family $(3_{n,\ell})$, the term appears on the fifth line of \Cref{prop:hfpsstwisted} in the first of the two summands.
	
	For family $(4_{n,\ell,r})$, the first term appears in the second part of the fifth line of \Cref{prop:hfpsstwisted} if $n>r$ and the last line if $n=r$. The second term appears in the fourth line if $n>r$ and the sixth line if $n=r$.
	
	For family $(5_{n,\ell,r})$, the first term appears in the second part of the fifth line of \Cref{prop:hfpsstwisted} if $n>r$ and the last line if $n=r$. The second term appears in the first part of the fifth line if $n>r$ and the second to last line if $n=r$. The third term appears in the third to last line if $k>1$, and in the second to last line if $k=1$ (this term appears only when $n=r$).
	
	For the claim about linear independence, we note that it is straightforward to verify that each element of each term of each family $(1_{n,\ell})\dots (5_{n,\ell,r})$ only appears once and is a basis element of the answer of \Cref{prop:hfpsstwisted} in terms of cyclic $\FF_p[v_1]$-modules, so there can be no linear relations. Note that this holds despite the fact that elements may appear to be of the same type: for example elements of the form $v_1^{1+\dots + p^{n+1}+ p^{n+1}-1}\lambda_1 \sigma\epsilon^{(p^{n+1})}t^{p^{n+1}(p-1)}$ appear to come from both family $(2_{n,\ell})$ and $(5_{n,\ell,r})$, but the former happens when $p|n+1$ and the latter when $p\nmid n+1$.
%	To see that family $(1_{n,\ell})$ and $(2_{n,\ell})$ are in the kernel of $\varphi-\can$, $\can$ vanishes on the first summand, and $\varphi$ vanishes on the last summand that is nonzero in the sum. It is then easy to see that $\varphi$ of each term in the sum is $\can$ of the next term in the sum.
%	We first explain the torsion order of the generators. The torsion order of the $(1_{n,\ell})$ comes from the third summand of \Cref{prop:hfpsstwisted}, and the torsion order for the second
\end{proof}
%It is not too hard to show that the families above determine a basis for these modules. (Need to expand on what this means.)
%Using this, we determine their isomorphism types (as $\F_p[v_1]$-modules) by noting the $v_1$-torsion levels of these elements:
%
%\begin{enumerate}
%  \item 
%    $\lambda_1 ^e (\sigma \epsilon^{(p^n \ell)} \mu^j \stackrel{\cdot}{+} \sigma \epsilon^{(p^{n+1} \ell)} t^{p^n\ell(p-1) -pj})$ is $v_1^{1+ \dots + p^n}$-torsion.
%  \item
%    $\lambda_1 ^e (\sigma \epsilon^{(p^n \ell)} \mu^j \stackrel{\cdot}{+} \sigma \epsilon^{p^{n+1} k} t^{p^n\ell(p-1) -pj} \stackrel{\cdot}{+} \delta^{j}_{p^{n-1}\ell(p-1)} \sigma \epsilon^{(p^{n+2} \ell)} t^{p^{n+1} \ell (p-1)})$ is $v_1^{1+\dots+p^n+(p^{n+1} - (p^n\ell(p-1)-pj))}$-torsion.
%  \item
%  $\sigma \epsilon^{(p^n \ell)} t^i \lambda_1 u_n ^e$ is $v_1^{p-i}$-torsion.
%  \item
%    $\sigma \epsilon^{(p^n \ell)} \mu^j \lambda_1 u_n^e \stackrel{\cdot}{+} \sigma \epsilon^{(p^{n+1} \ell)} t^{p^n\ell(p-1) -pj} \lambda_1 u_{n+1} ^e$ is $v_1^{1+ \dots + p^{r}}$-torsion.
%  \item 
%    $\sigma \epsilon^{(p^n \ell)} \mu^j \lambda_1 u_n^e \stackrel{\cdot}{+} \sigma \epsilon^{(p^{n+1} \ell)} t^{p^n\ell(p-1) -pj} \lambda_1 u_{n+1} ^e \stackrel{\cdot}{+} \delta^j_{p^{n-1}\ell(p-1)} \sigma \epsilon^{(p^{n+2} \ell)} t^{p^{n+1} k(p-1)} \lambda_1 u_{n+2} ^e$ is $v_1^{1+ \dots + p^{r} + (p^{r+1} - (p^n\ell(p-1)-pj))}$-torsion.
%\end{enumerate}

Now, we need to deal with $\gr^* (\pi_* \gr^* _{\mot} \TR^{[m]} (\Zp; \Sigma^{2\ell} \Zp) / p)$ instead of $\gr^* (\pi_* \gr^* _{\mot} \TR (\Zp; \Sigma^{2\ell} \Zp) / p)$. First, some of these families get truncated.

\begin{dfn}
We let $A_{n,\ell} ^{[m]}$, $B_{n,\ell} ^{[m]}$, $C_{n,\ell} ^{[m]}$, $D_{n,\ell,r} ^{[m]}$ and $E_{n,\ell,r} ^{[m]}$ denote the images of these in $\prod_{i=0} ^m \gr_* ^{v_1} \pi_* \gr_{\mot}^* (\THH(\Zp) \otimes \Ss^{\ell \rho_i})^{hC_{p^i}} /p$.
\end{dfn}

\begin{lem}\label{lemma:v1torsquot}
  \begin{itemize}
    \item If $m > n$, we have $A_{n,\ell} ^{[m]} = B_{n,\ell} ^{[m]} = C_{n,\ell} ^{[m]} = D_{n,\ell,r} ^{[m]} = E_{n,\ell,r} ^{[m]} = 0$.
    \item If $n \geq m$, we have $A_{n,\ell} \cong A_{n,\ell} ^{[m]}$, $C_{n,\ell} \cong C_{n,\ell} ^{[m]}$and $D_{n,\ell,r} \cong D_{n,\ell,r} ^{[m]}$.
    \item If $n > m+1$ or $n=m+1$ and either $r<n$ or $k\neq 1$, we have $B_{n,\ell} \cong B_{n,\ell} ^{[m]}$ and $E_{n,\ell,r} \cong E_{n,\ell,r} ^{[m]}$.
    \item The maps $B_{n,1} \to B_{n,1} ^{[n+1]}$ and $E_{n,1,n} \to E_{n,1,n} ^{[n+1]}$ are isomorphic to $B_{n,1} \to B_{n,1} / (v_1^{1+\dots+p^n+(p^{n+1} - (p^n (p-1)-pj))}(\lambda_1 ^e (\sigma \epsilon^{(p^n )} \mu^j \stackrel{\cdot}{+} \sigma \epsilon^{(p^{n+1} )} t^{p^n (p-1) -pj} \stackrel{\cdot}{+} \delta^{j}_{p^{n-1}(p-1)} \sigma \epsilon^{(p^{n+2})} t^{p^{n+1} (p-1)}))$ where $j=p^n(p-1)$, and 
    
    $E_{n,1,n} \to E_{n,1,n} / (v_1^{1+ \dots + p^{n} + (p^{n+1} - (p^n (p-1)-pj))}\sigma \epsilon^{(p^n)} \mu^j \lambda_1 u_n^e \stackrel{\cdot}{+} \sigma \epsilon^{(p^{n+1})} t^{p^n(p-1) -pj} \lambda_1 u_{n+1} ^e \stackrel{\cdot}{+} \delta^j_{p^{n-1}(p-1)} \sigma \epsilon^{(p^{n+2} )} t^{p^{n+1}(p-1)} \lambda_1 u_{n+2} ^e)$.
    \item The maps $B_{n,\ell} \to B_{n,\ell} ^{[n]}$ and $E_{n,\ell,r} \to E_{n,\ell,r} ^{[n]}$ are isomorphic to $B_{n,\ell} \to B_{n,\ell} / (v_1 ^{1 + \dots + p^n})$ and $E_{n,\ell,r} \to E_{n,\ell,r} / (v_1 ^{p + \dots + p^r})$, repsectively.
  \end{itemize}
\end{lem}

\begin{proof}
	This is a straightforward consequence of \Cref{lemma:trfamilies}, by looking at the $v_1$-torsion orders of the elements in the sums making up the families $(1_{n,\ell})\dots (5_{n,\ell,r})$, and seeing for which $i$ they are detected in $\gr_* ^{v_1} \pi_* \gr_{\mot}^* (\THH(\Zp) \otimes \Ss^{\ell \rho_i})^{hC_{p^i}} /p$.
\end{proof}

There are also two additional families in $\gr^* (\pi_* \gr^* _{\mot} \TR^{[n]} (\Zp; \Sigma^{2\ell} \Zp) / p)$:

\begin{dfn}
\begin{align}
  &\lambda_1 ^e \sigma \epsilon^{(p^n \ell)} \mu^j \tag{$6_{n,\ell}$} \\
  &\text{where }
  e \in \{0,1\}, j > p^{n-1} \ell (p-1), j \equiv n \ell p^{n-1} \mod p^n \notag
\end{align}

\begin{align}
  &\sigma \epsilon^{(p^n \ell)} \mu^j \lambda_1 u_n^e \tag{$7_{n,\ell,r}$} \\
  &\text{where }
  e \in \{0,1\}, j > p^{n-1} \ell (p-1), v_p (j-\ell n p^{n-1}) = r-1 \notag
\end{align}
\end{dfn}

\begin{dfn}
We let $F_{n,\ell} ^{[n]}$ and $G_{n,\ell,r} ^{[n]}$ denote the $\F_p [v_1]$-modules generated by $(6_{n,\ell})$ and $(7_{n,\ell,r})$ in $\prod_{i=0} ^m \gr_* ^{v_1} \pi_* \gr_{\mot}^* (\THH(\Zp) \otimes \Ss^{\ell \rho_i})^{hC_{p^i}} /p$.
\end{dfn}

\begin{lem}\label{lemma:FG}
  $F_{n,\ell}^{[n]}$ and $G_{n,\ell,r}^{[n]}$ lie in $\gr^{*} (\pi_*\gr_{\mot} ^* \TR^{[n]} (\Zp; \Sigma^{2\ell} \Zp) / p)$, and their elements are linearly independent in the sense that $A_{n,\ell}^{[n]}+ B_{n,\ell}^{[n]}+ C_{n,\ell}^{[n]}+ D_{n,\ell,r}^{[n]}+ E_{n,\ell,r}^{[n]}+ F_{n,\ell}^{[n]} + G_{n,\ell,r}^{[n]}$ is the direct sum of the cyclic $\FF_p[v_1]$-submodules generated by each element of $(1_{n,\ell}), (2_{n,\ell}), (3_{n,\ell}), (4_{n,\ell,r}), (5_{n,\ell}), (6_{n,\ell})$ and $(7_{n,\ell,r})$. Moreover the $v_1$-torsion order of the classes are listed in \Cref{lemma:v1torsquot} and below:
	\begin{enumerate}
		\item[$(6_{n,\ell})$] $\lambda_1 ^e \sigma \epsilon^{(p^n \ell)} \mu^j$ is $v_1 ^{1+ \dots + p^n}$-torsion
		\item[$(7_{n,\ell,r})$] $\sigma \epsilon^{(p^n \ell)} \mu^j \lambda_1 u_n^e$ is $v_1 ^{p + \dots + p^r}$-torsion.
	\end{enumerate}
\end{lem}

\begin{proof}
	To see that the classes in $(6_{n,\ell})$ and $(7_{n,\ell,r})$ are in $\TR^{[n]}$, we just need to see they are in the kernel of the map $\can$, which is true because they are a multiple of $\mu$. The $v_1$-torsion orders can be read off from the fact that $(6_{n,\ell})$ appears in the family of \Cref{prop:hfpsstwisted} on the third line and $(7_{n,\ell,r})$ appears on the second part of the fifth line. Linear independence of the classes follows as in \Cref{lemma:trfamilies}, because basis elements are not repeated among elements.
\end{proof}

Our main theorem computes $\gr^{*} (\pi_*\gr_{\mot} ^* \TR^{[m]} (\Zp; \Sigma^{2\ell} \Zp) / p)$ and $\gr^{*} (\pi_*\gr_{\mot} ^* \TR (\Zp; \Sigma^{2\ell} \Zp) / p)$ as $\F_p [v_1]$-modules.

\begin{thm} \label{thm:computation}
  There are isomorphisms:
  \begin{align*}
    \gr^* \pi_*\gr_{\mot}^* \TR^{[m]} (\Zp; \Sigma^{2\ell} \Zp)/p \cong \bigoplus_{n=0} ^m \left( A_{n,\ell} ^{[m]} \oplus B_{n,\ell} ^{[m]} \oplus C_{n,\ell} ^{[m]} \oplus \bigoplus_{r=1} ^n \left(D_{n,\ell,r} ^{[m]} \oplus E_{n,\ell,r} ^{[m]} \right) \right) \oplus F_{m,\ell} ^{[m]} \oplus \bigoplus_{r=1} ^m \left( G_{m,\ell,r} ^{[m]} \right)
  \end{align*}
  for all $m \geq 0$ and
  \begin{align*}
    \gr^* \pi_*\gr_{\mot}^* \TR (\Zp; \Sigma^{2\ell} \Zp)/p \cong \bigoplus_{n=0} ^\infty \left( A_{n,\ell} \oplus B_{n,\ell} \oplus C_{n,\ell} \oplus \bigoplus_{r=1} ^n \left(D_{n,\ell,r} \oplus E_{n,\ell,r} \right) \right)
  \end{align*}
\end{thm}

Before we prove this theorem, we use it to show that we can get rid of the associated graded.
%When combined with the following proposition, this determines $\pi_* \TR (\Zp; \Sigma^{2\ell} \Zp)/p$ as a graded $\F_p [v_1]$-module.

\begin{prop}
  There is a (noncanonical) isomorphism of graded $\F_p [v_1]$-modules
  \[\pi_* \gr_{\mot}^* \TR (\Zp; \Sigma^{2\ell} \Zp)/p \cong \gr^* \pi_*\gr_{\mot}^* \TR (\Zp; \Sigma^{2\ell} \Zp)/p.\]
\end{prop}

\begin{proof}
  In fact, we show that $\Fil^i \pi_* \TR (\Zp; \Sigma^{2\ell} \Zp)/p = \Fil^i _{v_1} \pi_* \TR (\Zp; \Sigma^{2\ell} \Zp)/p$, from which the claim follows immediately.
  In the remainder of this proof, we will write $\Fil^i$ and $\Fil^i _{v_1}$ for these graded modules.

  By definition, $\Fil^i _{v_1} \subseteq \Fil^i$.
  %We need to show that $\Fil^i \subseteq \Fil^i_{v_1}$.
  %In other words, that if $v_1x$ in $\gr_{v_1} ^* \prod_{l=0} ^\infty \pi_* \gr_{\mot}(\THH(\Zp) \otimes \Ss^{\ell \rho_l})^{hC_{p^l}}/p$ is in the kernel of $\varphi-\can$, then there is a $y$ such that $v_1y=0$, and $x-y$ is in the kernel of $\varphi-\can$. Write $x$ as a sum of basis elements. Using  \Cref{prop:grv1canphi}, we see that with the exception of the classes in \Cref{prop:hfpsstwisted} in the second line, the first part of the fifth line, and the seventh line, all basis elements for $\varphi$ and $\can$ are nonzero, $\varphi$ and $\can$ preserve the $v_1$-torsion order of the basis element (and linear independence in the sense of \Cref{lemma:trfamilies}). For the exceptional family of classes, we note that they appear as $v_1$ multiples of elements in families $(2_{n,\ell})$ and $(5_{n,\ell,r})$ respectively, so that by dividing the appropriate basis element by $v_1$ for each such term, and adding on the other term when these basis elements appearing, we obtain the desired $y$.\todo{this line is subtle, make sure I didn't mess up.}
  It therefore suffices to show that $\Fil^i \subset \Fil^i _{v_1}$.
  Inspecting the description of  of $\gr^* \pi_*\gr_{\mot}^* \TR (\Zp; \Sigma^{2\ell} \Zp)/p$ given in \Cref{thm:computation}, we see that $v_1 : \gr^* \to \gr^{*+1}$ is surjective.
  Since $\Fil^*$ is a finite filtration in each degree\footnote{To prove this, we note that $\pi_* \gr_{\mot}^* (\THH(\Zp) \otimes \Ss^{\ell \rho_m})^{h C_{p^m}}/p$ has bounded $v_1$-torsion and $\pi_* \gr_{\mot}^* \TR (\Zp; \Sigma^{2\ell} \Zp)/p \to \pi_* \gr_{\mot}^* \TR^{[m]} (\Zp; \Sigma^{2\ell} \Zp)/p$ is an isomorphism in an increasing range of degrees.}, this implies that $v_1 : \Fil^* \to \Fil^{*+1}$ is surjective, so that $\Fil^i \subset \Fil^i _{v_1}$ by induction on $i$, as desired.
\end{proof}

\begin{proof}[Proof of \Cref{thm:computation}]
  Our proof is by induction on $m$; the base case $m=0$ is easy since $\gr^* _{\mot} \TR^{[0]} (\Zp; \Sigma^{2\ell} \Zp) \simeq \gr^* _{\mot} \THH(\Zp; \Sigma^{2\ell} \Zp)$.

  By definition, there are pullback squares:
  \begin{center}
    \begin{tikzcd}
      \gr_{\mot} ^* \TR^{[m]} (\Zp; \Sigma^{2\ell} \Zp) \ar[r] \ar[d] & \gr_{\mot} ^* (\THH(\Zp) \otimes \Ss^{\ell \rho_m})^{h C_{p^m}} \ar[d, "\can"] \\
      \gr_{\mot} ^* \TR^{[m-1]} (\Zp; \Sigma^{2\ell} \Zp) \ar[r, "\varphi"] & \gr_{\mot} ^* (\THH(\Zp) \otimes \Ss^{\ell \rho_m})^{t C_{p^m}},
    \end{tikzcd}
  \end{center}
  where $\varphi: \gr_{\mot} ^* \TR^{[m-1]} (\Zp; \Sigma^{2\ell} \Zp) \to \gr_{\mot} ^* (\THH(\Zp) \otimes \Ss^{\ell \rho_m})^{t C_{p^m}}$ is the composite
  \[\gr_{\mot} ^* \TR^{[m-1]} (\Zp; \Sigma^{2\ell} \Zp) \to \gr_{\mot} ^* (\THH(\Zp) \otimes \Ss^{\ell \rho_{m-1}})^{h C_{p^{m-1}}} \xrightarrow{\varphi} \gr_{\mot} ^* (\THH(\Zp) \otimes \Ss^{\ell \rho_{m}})^{t C_{p^{m}}}.\]
  Modding out by $p$ and using \Cref{lemma:surjphican}, we see that
  \begin{center}
    \begin{tikzcd}
      \gr^* \pi_* \gr_{\mot} ^* \TR^{[m]} (\Zp; \Sigma^{2\ell} \Zp)/p \ar[r] \ar[d] & \gr^* _{v_1} \pi_* \gr_{\mot} ^* (\THH(\Zp) \otimes \Ss^{\ell \rho_m})^{h C_{p^m}}/p \ar[d, "\can"] \\
      \gr^* \pi_* \gr_{\mot} ^* \TR^{[m-1]} (\Zp; \Sigma^{2\ell} \Zp)/p \ar[r, "\varphi"] & \gr^* _{v_1} \pi_* \gr_{\mot} ^* (\THH(\Zp) \otimes \Ss^{\ell \rho_m})^{t C_{p^m}}/p
    \end{tikzcd}
  \end{center}
  is a pullback of graded $\F_p [v_1]$-modules.

  First, we note that $F^{[m-1]} _{m-1,\ell} \oplus \bigoplus_{r=1} ^{m-1} \left( G^{[m-1]} _{m-1,\ell,r} \right)$ maps isomorphically onto the basis elements of $\gr_{v_1} ^* \pi_* \gr_{\mot}^* (\THH(\Zp) \otimes \Ss^{\ell \rho_m})^{tC_{p^m}}/p$ which contain a negative power of $t$. More specifically, $F^{[m-1]} _{m-1,\ell}$ hits elements in the first summand listed in \Cref{prop:twistedtatess}, and $G^{[m-1]} _{m-1,\ell,r}$ hits elements in the second line if $r<m-1$ and the third if $r=m-1$.

  On the other hand, the map
  $$\can : \gr_{v_1} ^* \pi_* \gr_{\mot}^* (\THH(\Zp) \otimes \Ss^{\ell \rho_m})^{h C_{p^m}} \to \gr_{v_1} ^* \pi_* \gr_{\mot}^* (\THH(\Zp) \otimes \Ss^{\ell \rho_m})^{t C_{p^m}}$$
  surjects onto the basis elements which contain a nonnegative power of $t$.
  Moreover, we can easily compute the kernel of this map, to be generated by all elements with a positive power of $\mu$ (i.e in the third summand, the second part of the fifth line, and the last line of \Cref{prop:hfpsstwisted}), and the $v_1$-multiples of elements from the second, first part of the fifth, and the second to last lines which are sent to $0$. Note that in the case $k=0$ in the first part of the fifth line, the whole summand is in the kernel.
  
  From \Cref{lemma:trfamilies} and \Cref{lemma:FG} there is an inclusion
  \begin{align*}
     \bigoplus_{n=0} ^m \left( A_{n,\ell} ^{[m]} \oplus B_{n,\ell} ^{[m]} \oplus C_{n,\ell} ^{[m]} \oplus \bigoplus_{r=1} ^n \left(D_{n,\ell,r} ^{[m]} \oplus E_{n,\ell,r} ^{[m]} \right) \right) \oplus F_{m,\ell} ^{[m]} \oplus \bigoplus_{r=1} ^m \left( G_{m,\ell,r} ^{[m]} \right) \\
     \subseteq \gr^* \pi_* \gr_{\mot} ^* \TR^{[m]} (\Zp; \Sigma^{2\ell} \Zp)/p.
  \end{align*}

  To check that it is an isomorphism, it then suffices to check that the kernel of the map
  \begin{align*}
    \bigoplus_{n=0} ^m \left( A_{n,\ell} ^{[m]} \oplus B_{n,\ell} ^{[m]} \oplus C_{n,\ell} ^{[m]} \oplus \bigoplus_{r=1} ^n \left(D_{n,\ell,r} ^{[m]} \oplus E_{n,\ell,r} ^{[m]} \right) \right) \oplus F_{m,\ell} ^{[m]} \oplus \bigoplus_{r=1} ^m \left( G_{m,\ell,r} ^{[m]} \right) \to \\
    \bigoplus_{n=0} ^{m-1} \left( A_{n,\ell} ^{[m-1]} \oplus B_{n,\ell} ^{[m-1]} \oplus C_{n,\ell} ^{[m-1]} \oplus \bigoplus_{r=1} ^n \left(D_{n,\ell,r} ^{[m-1]} \oplus E_{n,\ell,r} ^{[m-1]} \right) \right) \oplus F_{m-1,\ell} ^{[m-1]} \oplus \bigoplus_{r=1} ^{m-1} \left( G_{m-1,\ell,r} ^{[m-1]} \right)
  \end{align*}
  maps isomorphically onto the kernel of
  $$\can : \gr_{v_1} ^* \pi_*\gr_{\mot}^* (\THH(\Zp) \otimes \Ss^{\ell \rho_m})^{h C_{p^m}} \to \gr_{v_1} ^* \pi_* \gr_{\mot}^* (\THH(\Zp) \otimes \Ss^{\ell \rho_m})^{t C_{p^m}}.$$
  
  We explain how to match up the terms, where the reader may refer to \Cref{lemma:v1torsquot} for a description of the kernels between each summand as $m$ vaies.
  
  To do this, the terms with positive power of $\mu$ in the kernel of the canonical map coming from the third summand of \Cref{prop:hfpsstwisted} correspond to the kernel of $A_{m,\ell}^{[m]} \to A_{m,\ell}^{[m-1]}$, the kernel of $B_{m,\ell}^{[m]} \to B_{m,\ell}^{[m-1]}$, and $F_{m,\ell}^{[m]} \to 0$. The terms with positive powers of $\mu$ in the kernel of the canonical map from the second part of the fifth line and the last line of \Cref{prop:hfpsstwisted} correspond to the kernel of $D_{m,\ell,r}^{[m]} \to D_{m,\ell,r}^{[m-1]}$ and the kernel of $E_{m,\ell,r}^{[m]} \to E_{m,\ell,r}^{[m-1]}$ in the cases $r<m$ and $r=m$ respectively, and $G_{m,\ell,r}^{[m]} \to 0$.
  
  The $v_1$-multiples of elements from the second line in the kernel of the canonical map correspond to the kernel of $B_{m,\ell}^{[m+1]} \to B_{m,\ell}^{[m]}$, elements from the first part of the fifth and second to last lines correspond to the kernels of $E_{m,\ell,r}^{[m+1]} \to E_{m,\ell,r}^{[m]}$ in the case $1\leq r<m$ and $r=m$ respectively. In the case $k=0$ in the first part of the fifth line, this corresponds to the kernel of $C_{m,\ell}^{[m]} \to C_{m,\ell}^{[m-1]}$.

  Finally, to obtain the claim for $\gr^{*}\pi_*\gr_{\mot}^* \TR (\Zp; \Sigma^{2\ell} \Zp)/p$ from $\gr^{*}\pi_*\gr_{\mot}^* \TR^{[m]} (\Zp; \Sigma^{2\ell} \Zp)/p$, we use the equivalence
  $\gr^{*}\pi_*\gr_{\mot}^* \TR (\Zp; \Sigma^{2\ell} \Zp)/p \cong \varprojlim_{m} \gr^{*}\pi_*\gr_{\mot}^* \TR^{[m]} (\Zp; \Sigma^{2\ell} \Zp)/p$
  along with the fact that 
  $\gr^{*}\pi_*\gr_{\mot} ^* \TR^{[m]} (\Zp; \Sigma^{2\ell} \Zp)/p \to \gr^{*}\pi_*\gr_{\mot} ^* \TR^{[m-1]} (\Zp; \Sigma^{2\ell} \Zp)/p$
  is an equivalence in a range increasing in $m$\footnote{Because $\can : \pi_*\gr_{\mot}^* (\THH(\Zp) \otimes \Ss^{\ell \rho_m})^{hC_{p^m}} / p \to \pi_* \gr_{\mot}^* (\THH(\Zp) \otimes \Ss^{\ell \rho_m})^{tC_{p^m}} / p$ is, by inspection.} to deduce that the $\varprojlim^1$-terms vanish.
\end{proof}

\begin{cor}\label{cor:tczpeps}
  The bigraded $\F_p [v_1]$-module $\pi_*\gr_{\mot}^* \TC(\ZZ_p\langle \epsilon\rangle)/p$ is isomorphic to
  \[\pi_* \gr_{\mot}^* \TC(\ZZ_p)/p \oplus \bigoplus_{p \nmid \ell>0} \pi_* \gr_{\mot}^* \TR (\ZZ_p;\Sigma^{2\ell} \ZZ_p)/p,\]
  where for each value of $\ell$, the homotopy groups are given by \Cref{thm:computation}.

  If $k\leq p^{n-2}$, then the mod $(p,v_1^k)$-syntomic cohomology of $\ZZ_p/p^n$ is isomorphic to $\coker_{v_1^k}(\pi_* \gr_{\mot}^*\TC(\ZZ_p\langle \epsilon\rangle)) \oplus \ker_{v_1^k}\pi_* \gr_{\mot}^* \TC(\ZZ_p\langle \epsilon\rangle)[(2p-2)k-1,1]$, where $[(2p-2)k-1,1]$ is the shift determined by $\pi_* \gr_{\mot}^* \TC(\ZZ_p\langle \epsilon\rangle)[(2p-2)k-1,1] = \pi_{*-(2p-2)k+1} \gr_{\mot}^{*-(2p-2)k} \TC(\ZZ_p\langle \epsilon\rangle)$.
\end{cor}

\begin{proof}
  The first claim follows from combining \Cref{cor:mot-fil-epsilon-main} in the case $R = \Z_p$ with \Cref{thm:computation}. The second claim follows from taking the quotient by $v_1^k$, and applying \Cref{thm:main}(1) to identify the result for $\Z_p\langle \epsilon \rangle$ with that of $\ZZ/p^n$. 
\end{proof}

\begin{rmk}\label{remark:liuwang}
  We note that the homotopy groups of the other term $\gr_{\mot}^* \TC(\ZZ_p)/p$ are given in \cite[Theorem 1.5]{LiuWang} as $\FF_p[v_1]\otimes E(\lambda_1)\otimes E(\partial) \oplus \FF_p[v_1]\{t\lambda_1,\dots,t^{p-1}\lambda_1\}$.
\end{rmk}

\section{Consequences for the algebraic $K$-theory of $\mathbb{Z}/p^n$}\label{sec:consequences}
In this section, we use our results on syntomic cohomology to draw consequences for topological cyclic homology and algebraic $K$-theory.

\subsection{Consequences for even $K$-groups}

First, we prove some results about the groups $K_{2i}(\Z/p^n)$, which vanish for $i \gg 0$ by the even vanishing theorem of \cite{kzpnother}.  We note that, as explained in \cite{kzpnother}, when displayed in Adams grading the motivic spectral sequence computing $\pi_*\TC(\Z/p^n)$ is concentrated on the $0$-line, $1$-line, and $2$-line.  Furthermore, the $0$-line is trivial above degree zero, by \cite[Corollary 2.13(ii)]{kzpnother}.  In particular, the motivic spectral sequence degenerates on the $\mathrm{E}_2=\mathrm{E}_{\infty}$-page.

\begin{thm}\label{thm:2-linesurjectivity}
  For $n \geq 2$, the maps
  \[H^2 (\F_p (i) (\Z_p)/v_1 ) \to H^2 (\F_p (i) (\Z/p^n)/v_1 )\]
  \[H^2 (\F_p (i) (\Z_p)) \to H^2 (\F_p (i) (\Z/p^n))\]
  \[H^2 (\Z_p (i) (\Z_p)) \to H^2 (\Z_p (i) (\Z/p^n))\]
  are surjective.
\end{thm}

\begin{proof}
By \Cref{thm:main}, the map 
  \[H^2 (\F_p (i) (\Z_p)/v_1 ) \to H^2 (\F_p (i) (\Z/p^n)/v_1 )\]
  is isomorphic to the map
  \[H^2 (\F_p (i) (\Z_p)/v_1 ) \to H^2 (\F_p (i) (\Z_p \langle \epsilon \rangle)/v_1 ).\]
  That the first map is surjective now follows from \Cref{cor:tczpeps}, since each of the terms $\gr^*_{\mot}\TR (\ZZ_p;\Sigma^{2\ell} \ZZ_p)/(p,v_1)$ has homotopy concentrated between the $-1$-line and the $1$-line.

Running the $v_1$-Bockstein and then the $p$-Bockstein, it then follows that the second and third maps are also surjective. 
%	Suppose that the first map is surjective. Because the syntomic cohomology vanishes above degree $2$ on the source and target, by running the $v_1$-Bockstein spectral sequence, and the $p$-Bockstein spectral sequence, it follows that the second and third map are surjections too.
%	It thus suffices to show the first map is surjective.
%  By \Cref{thm:main}, the map 
%  \[H^2 (\F_p (i) (\Z_p)/v_1 ) \to H^2 (\F_p (i) (\Z/p^n)/v_1 )\]
%  is isomorphic to the map
%  \[H^2 (\F_p (i) (\Z_p)/v_1 ) \to H^2 (\F_p (i) (\Z_p \langle \epsilon \rangle)/v_1 ).\]
%  The result now follows from \Cref{cor:tczpeps}, since each of the terms $\gr_{\mot}\TR (\ZZ_p;\Sigma^{2\ell} \ZZ_p)/(p,v_1)$ are concentrated between the $-1$-line and the $1$-line.
\end{proof}

\begin{cor}
  For $n \geq 2$ and $i \ge 0$, the map
  \[K_{2i} (\Z_p) \to K_{2i} (\Z/p^n)\]
  is surjective.
  In particular, $K_{2i}(\Z/p^n)$ is cyclic.
\end{cor}

\begin{proof}
  It is clear that on $K_0$, the map is the isomorphism $\ZZ \xrightarrow{\simeq} \ZZ$.
  
  From the Dundas--Goodwillie--McCarthy theorem \cite{dundas2012local}, it follows that the fiber of the map $K(\ZZ/p^n) \to \TC(\ZZ/p^n)$ is the fiber of the map $K(\FF_p) \to \TC(\FF_p)$. It follows that $K(\ZZ/p^n)[\frac 1 p] \to K(\FF_p)[\frac 1 p]$ is an isomorphism, and since $K(\FF_p)[\frac 1 p]$ vanishes in positive degrees, it suffices to prove surjectivity after $p$-localization. Moreover the $p$-local $K$-groups are finite in positive degrees (see \cite[Theorem B]{angeltveit2011algebraic}), so it suffices to prove surjectivity after $p$-completion.
  
  The fiber of $K(\FF_p) \to \TC(\FF_p)$ is $p$-adically isomorphic to $\ZZ_p$ in degree $-2$, so the same is true for the fiber of the map $K(\ZZ/p^n) \to \TC(\ZZ/p^n)$.  It thus suffices to show that the $p$-completion of $\TC(\ZZ_p) \to \TC(\ZZ/p^n)$ is surjective in positive even degrees. We achieve this by looking at the map of motivic spectral sequences, which are concentrated on the $1$-line and $2$-line in positive degrees. Since the classes on the $1$-line are in odd degrees, it suffices to prove surjectivity on the $2$-line, which is \Cref{thm:2-linesurjectivity}.
\end{proof}

\subsection{Consequences for mod $(p,v_1^k)$ $K$-groups}

\begin{rec}
  When $p \geq 3$, there is a $v_1$-self map $v_1 : \Sigma^{2p-2} \Ss/p \to \Ss/p$ \cite{adams1966groups}.
  On the other hand, when $p=2$ we have a $v_1 ^4$-self map
  $v_1 ^4 : \Sigma^8 \Ss/2 \to \Ss/2$ \cite{adams1966groups}.

  As a consequence, the$\mod p$ homotopy groups of a spectrum are equipped with a self-map $v_1$ (resp $v_1^4$ when $p=2$).
  Moreover, we may form the generalized Moore complexes $\Ss/ (p,v_1^k)$ and consider the$\mod (p,v_1^k)$ homotopy of spectra (where $k$ is assumed to be divisible by $4$ if $p=2$).
\end{rec}

As input, we will require the following result of Achim Krause and the third named author \cite{KS}:

\begin{thm}[Krause--S.] \label{thm:2-line-tors}
  In $\pi_* \gr^{\mot} _* \TC(\Z/p^n)/p$, we have $v_1 ^{p^{n-2}} \partial \lambda_1 = 0$.
\end{thm}

In the above result, $\partial \lambda_1$ refers to the $\mathbb{F}_p[v_1]$-module generator of the $2$-line of $\pi_*\gr^*_{\mot} \TC(\mathbb{Z}_p) / p$, or more precisely its image along the natural map $\pi_*\gr^*_{\mot} \TC(\mathbb{Z}_p) / p \to \pi_*\gr^*_{\mot} \TC(\mathbb{Z}/p^n)/p$. Note that, by our \Cref{thm:2-linesurjectivity}, the $\mathbb{F}_p[v_1]$-module generated by $\partial \lambda_1$ contains all non-zero elements on the $2$-line of $\pi_*\gr^*_{\mot} \TC(\mathbb{Z}/p^n)/p$.

\begin{prop}
  Let $p$ denote a prime and $R=\Z_p \langle \epsilon \rangle$ or $R =\Z/p^n$ with $n \geq 2$.
  Then the motivic spectral sequence
  \[\pi_* \gr^\mot _* \TC(R)/ p \Rightarrow \pi_* \TC(R) / p\]
  degenerates at the $\mathrm{E}_2$-page. It has no hidden $v_1$-extensions ($v_1^4$-extensions if $p=2$), in the sense that any class $x\in \pi_* \gr^\mot _* \TC(R)/ p$ with $v_1^jx=0$ detects a class $\tilde{x} \in\pi_* \TC(R) / p$ with $v_1^j\tilde{x}=0$ (with $4$ dividing $j$ if $p=2$).
\end{prop}

%(Should I explain better what I mean by $v_1$-extension here? I just mean that things are never more $v_1^k$-divisible than expected. Alternatively, you can find lifts which are the expected level of torsion.)

\begin{proof}
Since motivic spectral sequences exhibit a checkerboard pattern, to prove their degeneration it suffices to show that they are concentrated on the $0$-line, $1$-line, and $2$-line.  In the case of $\ZZ/p^n$, we recall that $\pi_*\gr^*_{\mot}\TC(\ZZ/p^n)$ is concentrated on the $0$-line, $1$-line, and $2$-line, and furthermore the $0$-line is $p$-torsionfree \cite[Corollary 2.13]{kzpnother} (indeed, the $0$-line consists of a single copy of $\mathbb{Z}_p$ in degree $0$).  It follows that $\pi_*\gr^*_{\mot}\TC(\ZZ/p^n) / p$ is also concentrated on the $0$-line, $1$-line and $2$-line.
To prove degeneration of the motivic spectral sequence computing $\pi_*(\TC(\ZZ_p\langle \epsilon\rangle) / p)$, we also note that its $\mathrm{E}_2$-page is concentrated on the $0$-line, $1$-line, and $2$-line, for example by our explicit calculation of the $\mathrm{E}_2$-page as \Cref{cor:tczpeps}.

%to see it is on the $0,1,2$ line, we first observe that the mod $p$ prismatic cohomology/Nygaard filtered prismatic cohomology lie on the $0$ and $1$-lines. This follows from \Cref{prop:hfpsstwisted} and \Cref{prop:THHmodzero} using qrsp descent and left Kan extension. Then it follows that the syntomic cohomology lies on the $0,1,2$ lines. Since the $0$th syntomic cohomology is $p$-torsion free (it is just $\ZZ_p$ in degree $0$), the syntomic cohomology mod $p$ remains on the $0,1,2$ lines.

  It remains to prove that the are no hidden $v_1$-extensions (resp., hidden $v_1^4$-extensions).
  We begin with the case of $R = \Z/p^n$.
  Such an extension would have to go from the $0$-line to the $2$-line, and we recall that as an $\mathbb{F}_p$-vector space the $2$-line is generated by classes of the form $v_1^k \partial \lambda_1$ (which has degree $(2p-2)(k+1)$).
  The source of the extension would therefore be on the $0$-line in a degree divisible by $2p-2$ but less than $2p-2$.
  The only possible source of a hidden $v_1$-extension is therefore $1$.
  But the $v_1$-torsion order of $1$ is equal to $1+\dots+p^{n-1}$ by \cite[Theorem 6.3]{kzpnother}, and $v_1 ^{p^{n-2}} \partial \lambda_1 = 0$ by \Cref{thm:2-line-tors}, so there is no possible hidden $v_1$-extension.

  Next, we consider the case of $R = \Z_p \langle \epsilon \rangle$, where we have the splitting
  
  	$$\fil^*_{\mot}\TC(\ZZ_p)/p \oplus \bigoplus_{p \nmid \ell>0} \fil^*_{\mot}\TR (\ZZ_p;\Sigma^{2\ell} \ZZ_p)/p$$ of \Cref{thm:mot-fil-epsilon-main}. For each value of $\ell$, the homotopy groups of the associated graded are given by \Cref{thm:computation}.
  	
  The summand $\gr^*_{\mot}\TC(\ZZ_p)/p$ is $v_1$-torsionfree \cite[Theorem 1.5]{LiuWang}, so there are no possible $v_1$-extensions associated to that summand.
  The summands indexed by $\ell$ are concentrated on the $0$- and $1$-lines, so again there are no possible $v_1$-extensions for sparsity reasons.
\end{proof}

The following is an immediate corollary of the fact that there are no hidden $v_1$-extensions:

\begin{cor}
  Let $p$ denote a prime, $n \geq 2$, and $k \geq 1$ be divisible by $4$ if $p=2$.

  Then the spectral sequences
  \[\pi_* \gr^\mot _* \TC(\Z/p^n)/ (p,v_1^{k}) \Rightarrow \pi_* \TC(\Z/p^n) / (p,v_1 ^{k})\]
  and
  \[\pi_* \gr^\mot _* \TC(\Z_p \langle \epsilon \rangle)/ (p,v_1 ^{k}) \Rightarrow \pi_* \TC(\Z_p \langle \epsilon \rangle) / (p, v_1^{k})\]
  degenerate at the $\mathrm{E}_2$-page.
\end{cor}

Combining this with \Cref{thm:mainintro}, we obtain the following corollary:

\begin{cor}
  Let $p$ denote a prime, $n \geq 2$, and $k \leq p^{n-2}$.

  If $p$ is odd, then there is an isomorphism of $\F_p$-vector spaces
  \[\pi_* \gr^\mot _* \TC(\Z_p \langle \epsilon \rangle)/ (p,v_1^{k}) \cong \pi_* \TC(\Z/p^n) / (p,v_1 ^{k}).\]

  If $p=2$ and $k$ is divisible by $4$, then the graded abelian group $\pi_* \TC(R) / (2,v_1 ^{k})$ admits a filtration with associated graded given by $\pi_* \gr^\mot _* \TC(\Z_2 \langle \epsilon \rangle)/ (2,v_1 ^{k})$.
\end{cor}

%\begin{thm}
%  Let $p \geq 3$, $n \geq 2$, and $k \geq 1$.
%  Then there are isomorphisms of $\F_p$-vector spaces
%  \[\pi_* \TC(\Z/p^n) / (p,v_1^k) \cong \pi_* \gr^\mot _* \TC(\Z/p^n) / (p,v_1^k).\]
%  and
%  \[\pi_* \TC(\Z_p \langle \epsilon \rangle) / (p,v_1^k) \cong \pi_* \gr^\mot _* \TC(\Z_p \langle \epsilon \rangle) / (p,v_1^k).\]
%
%  If $p \geq 5$, then these are isomorphisms of $\F_p[v_1]$-modules.
%  If $p = 3$, then they are isomorphisms of $\F_3 [v_1^m]$-modules for all $m \geq 2$.
%\end{thm}
%
%By \Cref{thm:???}, all four of the objects appearing in THM are isomorphic when $k \leq 1+ p + \dots + p^{n-1}$.
%
%At $p=2$, the best we are able to do is the following:
%
%\begin{thm}
%  At $p=2$, the spectral sequences
%  \[\pi_* \gr^\mot _* \TC(\Z/2^n)/ (2,v_1^{4k}) \Rightarrow \pi_* \TC(R) / (2,v_1 ^{4k})\]
%  and
%  \[\pi_* \gr^\mot _* \TC(\Z_2 \langle \epsilon \rangle)/ (2,v_1 ^{4k}) \Rightarrow \pi_* \TC(\langle \epsilon \rangle) / (2, v_1^{4k})\]
%  degenerate at the $\mathrm{E}_2$-page.
%\end{thm}
%
%By \Cref{thm:???}, it follows that all four of the objects appearning in THM have the length as $\Z/4$-modules when $k \leq 1+ 2 + \dots + 2^{n-1}$.
%
\begin{rmk}
  There are $2$-extensions mapped forward from $\pi_*(\TC(\Z_2)/2)$ \cite[Theorem 8.21]{LiuWang}, but it is not clear that these constitute all $2$-extensions.
\end{rmk}

Finally, we conclude with a consequence for algebraic $K$-theory.

\begin{cor}
  Let $p$ denote a prime, $n \geq 2$, and let $k \leq p^{n-2}-1$. If $p=2$, then assume that $k$ is divisible by $4$.
  Then there is an exact sequence
  \[0 \to \F_p \{v_1 ^k \partial\} \to \pi_* K(\Z/p^n)/(p,v_1 ^k) \to \pi_* \TC(\Z/p^n)/(p,v_1 ^k) \to \F_p \{\partial\} \to 0.\]
\end{cor}

\begin{proof}
  By the Dundas-Goodwillie-McCarthy theorem, we have a pullback square:
  \begin{center}
    \begin{tikzcd}
      K(\Z/p^n) \ar[d] \ar[r] & K(\F_p) \ar[d] \\
      \TC(\Z/p^n) \ar[r] & \TC(\F_p).
    \end{tikzcd}
  \end{center}
  Using the identification of
  $\pi_* K(\F_p)/p \to \pi_* \TC(\F_p)/p$
  with
  $\F_p \to \F_p \{1,\partial\}$,
  we find that $K(\Z/p^n)/p \to \TC (\Z/p^n)/p$
  identifies the source with the connective cover of the target.
  Equivalently, $\pi_* K(\Z/p^n)/p$ obtained from $\pi_* \TC(\Z/p^n)/p$ by removing $\partial$.

  Now, it follows from \Cref{thm:mainintro} that $\partial$ is not $v_1 ^{p^{n-2}-1}$-torsion.
  The stated exact sequence is an immediate consequence.
\end{proof}

\begin{rmk}
  To determine what happens when $k=p^{n-2}$, one would need to figure out the precise $v_1$-torsion degree of $\partial$.
\end{rmk}

\bibliographystyle{alpha}
\bibliography{bibliography}

\end{document}